\documentclass[letterpaper,english,titlepage]{amsart}
\usepackage[T1]{fontenc}
\usepackage[latin9]{inputenc}
\usepackage{babel}
\usepackage{calc}
\usepackage{amstext}
\usepackage{amsthm}
\usepackage{amssymb}
\usepackage{xargs}[2008/03/08]
\usepackage[unicode=true,pdfusetitle,
 bookmarks=true,bookmarksnumbered=false,bookmarksopen=false,
 breaklinks=false,pdfborder={0 0 1},backref=page,colorlinks=false]
 {hyperref}

\makeatletter

\numberwithin{equation}{section}
\numberwithin{figure}{section}
\theoremstyle{plain}
\newtheorem{thm}{\protect\theoremname}[section]
\theoremstyle{definition}
\newtheorem{example}[thm]{\protect\examplename}
\theoremstyle{plain}
\newtheorem{lem}[thm]{\protect\lemmaname}
\theoremstyle{plain}
\newtheorem{prop}[thm]{\protect\propositionname}
\theoremstyle{definition}
\newtheorem{defn}[thm]{\protect\definitionname}
\theoremstyle{remark}
\newtheorem{rem}[thm]{\protect\remarkname}
\theoremstyle{plain}
\newtheorem{cor}[thm]{\protect\corollaryname}
\theoremstyle{plain}

\usepackage{babel}
\usepackage{amsaddr}

\usepackage{tikz}

\usepackage{pgfplots}

\usetikzlibrary{arrows}

\usepackage{amssymb,amsmath}
\usepackage{stmaryrd}

\usepackage{pdflscape}

\usepackage{tikz}

\usepackage{pgfplots}

\usetikzlibrary{arrows}

\usetikzlibrary{positioning}

\usetikzlibrary{backgrounds}

\usepackage{amssymb,tikz-cd}
\RequirePackage{tikz-cd}
\RequirePackage{amssymb}
\usetikzlibrary{calc}
\usetikzlibrary{decorations.pathmorphing}

\tikzset{curve/.style={settings={#1},to path={(\tikztostart)
    .. controls ($(\tikztostart)!\pv{pos}!(\tikztotarget)!\pv{height}!270:(\tikztotarget)$)
    and ($(\tikztostart)!1-\pv{pos}!(\tikztotarget)!\pv{height}!270:(\tikztotarget)$)
    .. (\tikztotarget)\tikztonodes}},
    settings/.code={\tikzset{quiver/.cd,#1}
        \def\pv##1{\pgfkeysvalueof{/tikz/quiver/##1}}},
    quiver/.cd,pos/.initial=0.35,height/.initial=0}

\tikzset{tail reversed/.code={\pgfsetarrowsstart{tikzcd to}}}
\tikzset{2tail/.code={\pgfsetarrowsstart{Implies[reversed]}}}
\tikzset{2tail reversed/.code={\pgfsetarrowsstart{Implies}}}
\tikzset{no body/.style={/tikz/dash pattern=on 0 off 1mm}}

\usepackage{enumitem}
\setenumerate{ label= (\alph*)}
\setenumerate[2]{ label= (\alph{enumi}\arabic*)} 

\let\myTOC\tableofcontents
\renewcommand{\tableofcontents}{%
  \frontmatter
  \pdfbookmark[1]{\contentsname}{}
  \myTOC
  \mainmatter }

\usepackage{xr}
\usepackage{etoolbox}

\patchcmd{\@maketitle}
{\if@titlepage \newpage \else}
{\if@titlepage
 \vspace{\baselineskip}
 \else}
{}{}

\providecommand{\corollaryname}{Corollary}
\providecommand{\definitionname}{Definition}
\providecommand{\examplename}{Example}
\providecommand{\lemmaname}{Lemma}
\providecommand{\propositionname}{Proposition}
\providecommand{\remarkname}{Remark}
\providecommand{\theoremname}{Theorem}

\makeatother

\providecommand{\corollaryname}{Corollary}
\providecommand{\definitionname}{Definition}
\providecommand{\examplename}{Example}
\providecommand{\factname}{Fact}
\providecommand{\lemmaname}{Lemma}
\providecommand{\propositionname}{Proposition}
\providecommand{\remarkname}{Remark}
\providecommand{\theoremname}{Theorem}

\begin{document}
\newcommandx\suc[5][usedefault, addprefix=\global, 1=N, 2=M, 3=K, 4=, 5=]{#1\overset{#4}{\hookrightarrow}#2\overset{#5}{\twoheadrightarrow}#3}%

\newcommandx\p[2][usedefault, addprefix=\global, 1=\mathcal{A}, 2=\mathcal{B}]{\left(#1, #2\right)}%

\global\long\def\Ext{\mathrm{Ext}}%

\global\long\def\Hom{\mathrm{Hom}}%

\global\long\def\End{\mathrm{End}}%

\global\long\def\A{\mathcal{A}}%
\global\long\def\B{\mathcal{B}}%
\global\long\def\I{\mathcal{I}}%
\global\long\def\C{\mathcal{C}}%
\global\long\def\F{\mathcal{F}}%
\global\long\def\G{\mathcal{G}}%
\global\long\def\Q{\mathcal{P}}%
\global\long\def\T{\mathcal{T}}%
\global\long\def\X{\mathcal{X}}%
\global\long\def\Y{\mathcal{Y}}%
\global\long\def\Z{\mathcal{Z}}%
\global\long\def\W{\mathcal{W}}%

\global\long\def\pd{\mathrm{pd}}%
 
\global\long\def\Gpd{\mathrm{Gpd}}%
 
\global\long\def\WGpd{\mathrm{WGpd}}%
 
\global\long\def\WGid{\mathrm{WGid}}%
 
\global\long\def\Dpd{\mathrm{Dpd}}%
 
\global\long\def\Proj{\mathrm{Proj}}%
 
\global\long\def\proj{\mathrm{proj}}%
 
\global\long\def\Inj{\mathrm{Inj}}%
 
\global\long\def\inj{\mathrm{inj}}%
 
\global\long\def\id{\mathrm{id}}%
 
\global\long\def\Gid{\mathrm{Gid}}%
 
\global\long\def\resdim{\mathrm{resdim}}%
 
\global\long\def\Findim{\mathrm{Fin.dim}\, }%
 
\global\long\def\findim{\mathrm{fin.dim}\, }%
 
\global\long\def\coresdim{\mathrm{coresdim}}%
 
\global\long\def\add{\mathrm{add}}%
 
\global\long\def\smd{\mathrm{smd}}%
 
\global\long\def\Add{\mathrm{Add}}%
 
\global\long\def\Prod{\mathrm{Prod}}%
 
\global\long\def\Flat{\mathrm{Flat}}%
 
\global\long\def\FP{\mathrm{FP}}%
 
\global\long\def\FPInj{\mathrm{FPInj}}%
 
\global\long\def\modu{\mathrm{mod}}%
 
\global\long\def\Mod{\mathrm{Mod}}%
 
\global\long\def\Ker{\mathrm{Ker}}%
 
\global\long\def\Ima{\mathrm{Im}}%
\global\long\def\im{\mathrm{Im}}%
 
\global\long\def\GP{\mathcal{GP}}%
 
\global\long\def\Gproj{\mathrm{Gproj}}%
 
\global\long\def\GI{\mathcal{GI}}%
 
\global\long\def\Ginj{\mathrm{Ginj}}%
 
\global\long\def\FGPD{\mathrm{FGPD}}%
 
\global\long\def\FGID{\mathrm{FGID}}%
 
\global\long\def\GPD{\mathrm{GPD}}%
 
\global\long\def\GID{\mathrm{GID}}%
 
\global\long\def\WGPD{\mathrm{WGPD}}%
 
\global\long\def\WGID{\mathrm{WGID}}%
 
\global\long\def\glPD{\mathrm{gl.PD}}%

\global\long\def\gldim{\mathrm{gl.dim}}%
 
\global\long\def\glID{\mathrm{gl.ID}}%
 
\global\long\def\glGPD{\mathrm{gl.GPD}}%
 
\global\long\def\glGpd{\mathrm{gl.Gpd}}%
 
\global\long\def\glGID{\mathrm{gl.GID}}%
 
\global\long\def\glGid{\mathrm{gl.Gid}}%
 
\global\long\def\glDPD{\mathrm{gl.DPD}}%
 
\global\long\def\glDID{\mathrm{gl.DID}}%
 
\global\long\def\FPD{\mathrm{FPD}}%
 
\global\long\def\fpd{\mathrm{fpd}}%
 
\global\long\def\FID{\mathrm{FID}}%
 
\global\long\def\fid{\mathrm{fid}}%
 
\global\long\def\WFGPD{\mathrm{WFGPD}}%
 
\global\long\def\glWGID{\mathrm{gl.WGID}}%
 
\global\long\def\glWGPD{\mathrm{gl.WGPD}}%
 
\global\long\def\WFGID{\mathrm{WFGID}}%
 
\global\long\def\DP{\mathcal{DP}}%
 
\global\long\def\DI{\mathcal{DI}}%
 
\global\long\def\FDPD{\mathrm{FDPD}}%
 
\global\long\def\DPD{\mathrm{DPD}}%
 
\global\long\def\Coker{\mathrm{Coker}}%
\global\long\def\Cok{\mathrm{Coker}}%
\global\long\def\colim{\mathrm{colim}}%
\global\long\def\Rep{\operatorname{Rep}}%
\global\long\def\lmcn{\mathrm{lmcn}}%
\global\long\def\rmcn{\mathrm{rmcn}}%
\global\long\def\mcn{\mathrm{mcn}}%
\global\long\def\lccn{\mathrm{lccn}}%
\global\long\def\rccn{\mathrm{rccn}}%
\global\long\def\ltccn{\mathrm{ltccn}}%
\global\long\def\rtccn{\mathrm{rtccn}}%
\global\long\def\ccn{\mathrm{ccn}}%
\global\long\def\rscn{\mathrm{rscn}}%
\global\long\def\lscn{\mathrm{lscn}}%
\global\long\def\scn{\mathrm{scn}}%

\global\long\def\tccn{\mathrm{tccn}}%
\global\long\def\Supp{\mathrm{Supp}}%
\global\long\def\Fun{\mathrm{Fun}}%
\global\long\def\Sets{\mathrm{Sets}}%
\global\long\def\Ab{\mathrm{Ab}}%

\title{Categories of quiver representations and relative cotorsion pairs}
\author{Alejandro Argud\'in-Monroy}
\email{argudin@ciencias.unam.mx}
\address{Centro de Ciencias Matem\'aticas, Universidad Nacional Aut\'onoma
de M\'exico, Morelia, Michoac\'an, MEXICO. }
\thanks{The first named author was supported by a postdoctoral fellowship
from DGAPA-UNAM}
\author{Octavio Mendoza-Hern\'andez}
\email{omendoza@matem.unam.mx}
\address{Instituto de Matem\'aticas, Universidad Nacional Aut\'onoma de M\'exico,
Ciudad de M\'exico, MEXICO. }
\keywords{abelian category, quiver representations, cotorsion pairs}
\subjclass[2000]{16G20; 18A40; 18E10; 18G25 }
\begin{abstract}
We study the category $\Rep(Q,\C)$ of representations of a quiver
$Q$ with values in an abelian category $\C$. For this purpose we
introduce the mesh and the cone-shape cardinal
numbers associated to the quiver $Q$ and we use them
to impose conditions on $\C$ that allow us to prove interesting homological properties of $\Rep (Q,\C)$  that can be constructed from $\C.$ For example, we 
compute the global dimension of $\Rep(Q,\C)$ in terms of the global one of $\C.$ We also review a result of H. Holm and P. J{\o}rgensen which states that
(under certain conditions on $\C$)  every hereditary complete cotorsion pair $\p$ in $\C$ induces the
hereditary complete cotorsion pairs $(\operatorname{Rep}(Q,\mathcal{A}),\operatorname{Rep}(Q,\mathcal{A})^{\bot_{1}})$
and $(^{\bot_{1}}\Psi(\mathcal{B}),\Psi(\mathcal{B}))$ in $\Rep(Q,\C)$,
and then we  obtain a strengthened version of this and others related
results. Finally, we will apply the above developed theory to study the following full abelian subcategories of $\Rep(Q,\C),$ 
finite-support, finite-bottom-support and finite-top-support representations. We show that the above mentioned cotorsion pairs in $\Rep(Q,\C)$ can be restricted nicely on the aforementioned subcategories and under mild conditions we also get hereditary complete cotorsion pairs.
\end{abstract}

\maketitle

\section{Introduction }

Let $Q$ be a quiver, that is, a directed graph given by a set of
vertices $Q_{0}$ and a set of arrows $Q_{1}$. From $Q$ we can define
the so-called \emph{free category}, or \emph{category of paths}, generated
by $Q$ and denoted by $\C_{Q}$. Now, given an abelian category $\C$,
we may understand the category $\Rep(Q,\C)$ of $\C$-representations of the quiver
$Q$ as the category of functors from $\C_{Q}$ to $\C$. 

Categories of quiver representations have been studied for decades and
for different reasons. Among their applications, they  are relevant for the study of module
categories over algebras, graded module categories over graded algebras
(see \cite{gabriel1992representations} and \cite[Chap. IX]{mitchell}),
and more recently, they have been used in topological data analysis
 \cite{escolar2016persistence, bauer2020cotorsion}.

To motivate our work, we would like to highlight the work of H. Holm
and P. J{\o}rgensen in \cite{holm2019cotorsion}. In this work, inspired
by a series of remarkable results obtained in \cite{enochs2005projective,enochs2004flat,eshraghi2013total},
Holm and J{\o}rgensen proved a first version of the following theorem (see \cite[Thms. 7.4 and 7.9]{holm2019cotorsion}
and \cite[Thm. B, Prop. 7.2 (b)]{holm2019cotorsion}), which
would later be revisited by S. Odaba{\c{s}}{\i} in \cite[Thm. 4.6]{odabacsi2019completeness}.

\begin{thm}\label{thm:carcajes 1} \label{thm:carcajes 2} \cite{holm2019cotorsion,odabacsi2019completeness}
Let $\mathcal{C}$ be an AB4 and AB4{*} abelian category having enough
projectives and injectives, $\p[\A][\B]$ be a cotorsion pair in $\mathcal{C}$
and $Q$ be a quiver. Then, the following statements hold true. 
\begin{itemize}
\item[$\mathrm{(a)}$] $(\operatorname{Rep}(Q,\mathcal{A}),\operatorname{Rep}(Q,\mathcal{A})^{\bot_{1}})$
and $(^{\bot_{1}}\Psi(\mathcal{B}),\Psi(\mathcal{B}))$ are cotorsion
pairs in the abelian category $\operatorname{Rep}(Q,\mathcal{C})$
such that $^{\bot_{1}}\Psi(\mathcal{B})\subseteq\operatorname{Rep}(Q,\mathcal{A})$. 
\item[$\mathrm{(b)}$] If $\p[\A][\B]$ is complete and hereditary 
and $Q$ is right-rooted, then we have that  $(\operatorname{Rep}(Q,\mathcal{A}),\Psi(\mathcal{B}))$
is a complete hereditary cotorsion pair in $\operatorname{Rep}(Q,\mathcal{C})$
and $\Psi(\B)\subseteq\operatorname{Rep}(Q,\B).$ 
\end{itemize}
\end{thm}

Here $\Rep(Q,\mathcal{A})$ denotes the representations $F\in\Rep(Q,\C)$
such that $F_i\in\mathcal{A}$ for all $i\in Q_{0}$, and $\Psi(\mathcal{B})$
denotes the class of the representations $G\in\Rep(Q,\C)$ such that
the canonical morphism $\psi_{i}:G_i\rightarrow\prod_{i\rightarrow j\in Q_{1}}G_j$
is an epimorphism and $G(i),\Ker(\psi_{i})\in\mathcal{B}$ for all
$i\in Q_{0}$ (see Definition \ref{def:phi, psi}). We point out that,  under the hypotheses given in Theorem \ref{thm:carcajes 1}, it is shown in \cite{holm2019cotorsion, odabacsi2019completeness} the following:  every
cotorsion pair $\p$ in $\C$ induces two cotorsion pairs in $\Rep(Q,\C)$ which are
$(\operatorname{Rep}(Q,\mathcal{A}),\operatorname{Rep}(Q,\mathcal{A})^{\bot_{1}})$
and $(^{\bot_{1}}\Psi(\mathcal{B}),\Psi(\mathcal{B}));$ and there
are conditions for these two cotorsion pairs to coincide and to be complete
and hereditary. It is worth mentioning that the authors in \cite{holm2019cotorsion, odabacsi2019completeness}  have pointed out that the
hypotheses of Theorem \ref{thm:carcajes 1}, and other related results, could be weakened if we assume some conditions on the quiver $Q$ instead of conditions on the abelian category $\C $. One of the
first objectives of our work will be to give specific conditions that
allow us to make this change of hypotheses. This will lead us to define
certain invariants that measure the geometric complexity of the quiver
(see Definitions \ref{def:mesh cardinal} and \ref{def:cone shape cardinal}).
Aided by these measures, we will be able to give a strengthened version
of some results from \cite{holm2019cotorsion}. In
particular, Theorem \ref{hccp+ep} and its dual can be seen as our modified version
of Theorem \ref{thm:carcajes 1}.

A second objective of this paper is to study some full abelian subcategories
of $\Rep(Q,\C)$. Specifically, we  consider the categories of \emph{finite-support},
\emph{finite-bottom-support} and \emph{finite-top-support} 
representations, which are respectively denoted by $\Rep^{f}(Q,\C),$ $\Rep^{fb}(Q,\C)$ and $\Rep^{ft}(Q,\C).$  Notice that these categories can be understood
as a generalization of the finite-dimension\-al representations previously
studied in \cite{paquette2013auslander,paquette2016irreducible,bautista2013representation}.
Among our results, we will see that under certain conditions, we can
restrict the cotorsion pairs of Theorem \ref{thm:carcajes 1} on the
subcategories aforementioned, see Corollary \ref{fcsq-hcp} and Theorems \ref{ThmX} and \ref{ThmY}, and thus we get hereditary complete cotorsion pairs in such subcategories. We also describe, under very general conditions, the projectives and the injectives in all these categories (see Proposition \ref{prop:f y g vs clases ortogonales} and the aforementioned theorems). Furthermore, we describe the global dimensions of all mentioned categories of representations in terms of the global dimension of the abelian category $\C,$ see Theorem \ref{thm:dimensiones homologicas}, Propositions \ref{prop:Repft2} and \ref{prop:Repft2-dual} and Corollary \ref{basic-fcsq}. As an application of the above, we obtain the Hochschild-Mitchell's theorem (see Example \ref{E-HM}) and the Eilenberg-Rosenberg-Zelinsky-Michell's theorem (see Example \ref{E-ERZM}).

In what follows,  we describe the structure of the paper as well as
some of the main results. In order to do that, let us recall that, for each vertex $x\in Q_{0},$ we have the \emph{evaluation functor} $e_{x}:\Rep(Q,\C)\rightarrow\C,\;
(F\xrightarrow{\alpha}G)\mapsto (F_x\xrightarrow{\alpha_x} G_x)$. Throughout the paper, we will see that under certain conditions on $Q$ and $\C,$ the functor
$e_{x}$ admits a left adjoint  $f_{x}$ and a right adjoint $g_{x}$.

In section 2 we start by giving some necessary preliminaries
on category theory, abelian categories, quivers and quiver representations to develop the paper.
It is worth mentioning Definitions \ref{def:mesh cardinal} and \ref{def:cone shape cardinal}
where it is introduced the \emph{mesh} and the \emph{cone-shape} cardinal numbers of the quiver $Q.$ These numbers consist of certain cardinals that
can be associated to the shape (mesh and cone) of the quiver $Q$ to measure its complexity. In subsection
2.11 we will see how these cardinal numbers help us to give more precise
statements about the nature and existence of the functors $f_{x},g_{x}:\C\to\Rep(Q,\C).$

In section 3, we construct  the functor $f:\C^{op}_Q\to\Fun(\C,\Rep(Q,\C))$
induced by  the family $\{f_x:\C\to \Rep(Q,\C)\}_{x\in Q_0}.$ It is also shown that, for any 
$F\in\Rep(Q,\C),$ there exists an exact sequence 
$
\eta_{F}:\;\suc[\coprod_{\rho\in Q_{1}}f_{t(\rho)}e_{s(\rho)}(F)][\coprod_{i\in Q_{0}}f_{i}e_{i}(F)][F][{}]
$
which is known as the canonical presentation induced by the family of functors $\{f_x:\C\to \Rep(Q,\C)\}_{x\in Q_0}.$ It is worth mentioning that this result was previously proved for
finite and acyclic quivers in \cite{bauer2020cotorsion}. In contrast, our result will be valid
for any quiver in exchange for requiring certain conditions on the
category $\C$ which depend on the \emph{cone-shape cardinal} numbers
of $Q,$ see the details in Corollary \ref{cor:presentacion canonica}. The rest of this section is devoted to show some
applications of the existence of the canonical presentation. First,
we  give conditions for the functor $g_{x}:\C\to\Rep(Q,\C)$ (resp. $f_x$) to have a right (resp. left) adjoint,
a result also previously proved for finite acyclic quivers in \cite{bauer2020cotorsion}, see Propositions \ref{prop:exactitud g} and \ref{prop:exactitud f}.
Second, we show that under some conditions we can compute  the global dimension 
$\gldim\,\Rep(Q,\C)$ of $\Rep(Q,\C)$ 
in terms of the global dimension $\gldim(\C)$ of $\C,$ see Theorem \ref{thm:dimensiones homologicas}. We also show how to use the canonical presentation  to get a generator set of $\Rep(Q,\C)$ from a generator one of $\mathcal{C}.$

In section 4 we will do the necessary work  to fulfil the first of
our objectives. The idea is to use the \emph{mesh} 
and the \emph{cone-shape} cardinal numbers to give conditions on the
category $\C$ so that we can develop a theory that allows us to get
the central results of \cite{odabacsi2019completeness,holm2019cotorsion}
with more general hypotheses. Specifically, in Theorem \ref{thm:cotrsion pairs rep dual}
we give conditions under which, for a given cotorsion pair $\p$ in $\C$,
we can construct the cotorsion pairs $(\operatorname{Rep}(Q,\mathcal{A}),\operatorname{Rep}(Q,\mathcal{A})^{\bot_{1}})$
and $(^{\bot_{1}}\Psi(\mathcal{B}),\Psi(\mathcal{B}))$ in $\Rep(Q,\C).$  Moreover, in the dual of Theorem \ref{hccp+ep} it is given enough
 conditions for these cotorsion pairs to be hereditary and complete. We close this section with a special description, see Proposition \ref{Descrip-Phi-Psi}, of the classes $\Psi(\B)$ and $\Phi(\A)$ that will be very useful in \cite{AM22} to study the 
 tilting cotorsion pairs in $\Rep(Q,\C ),$ for a finite-cone-shape quiver $Q.$

Finally, in section 5, we  use the previous work to study the categories
$\Rep^{f}(Q,\C),$ $\Rep^{fb}(Q,\C)$ and $\Rep^{ft}(Q,\C).$ 
First, we show that each of these categories is a full abelian subcategory of
$\Rep(Q,\C)$. Our strategy will be to study the poset $\F\B(Q)$
of finite-bottom subquivers and the poset $\F\T(Q)$ of finite-top
subquivers together with their induced subcategories of representations.
It should be noted that this type of strategy has been used previously
in the study of finite-dimensional algebras \cite{assem1998strongly,green2017convex}.
In Propositions \ref{prop:Repft} and \ref{prop:Repft2-dual}, we show that most of the results (and their duals) proved for $\Rep(Q,\C)$ can be restricted on
$\Rep^{fb}(Q,\C)$ and $\Rep^{ft}(Q,\C)$. Our central results will
be Theorem \ref{ThmX}, Theorem \ref{ThmY} and Corollary \ref{fcsq-hcp},
where conditions are given for which the cotorsion pairs $(\operatorname{Rep}(Q,\mathcal{A}),\operatorname{Rep}(Q,\mathcal{A})^{\bot_{1}})$
and $(^{\bot_{1}}\Psi(\mathcal{B}),\Psi(\mathcal{B}))$ can be restricted nicely
on the categories $\Rep^{fb}(Q,\C)$, $\Rep^{ft}(Q,\C)$ and $\Rep^{f}(Q,\C)$. It is worth mentioning that ``these good restrictions preserving nice properties'' is a phenomenon
which is rarely observed for arbitrary full abelian subcategories of an abelian category. 

\section{Preliminaries }

\subsection{Categories of functors}

Let $\mathcal{A}$ and $\mathcal{B}$ be categories. It is well-known
that, in case $\mathcal{A}$ is small, the functors from $\mathcal{A}$
to $\mathcal{B}$ and the natural transformations between them constitute
a category usually denoted by $\Fun(\mathcal{A},\mathcal{B}).$ We also set 
$\End(\A):=\Fun(\A,\A).$
Moreover, if $\mathcal{B}$ is an abelian category, then $\Fun(\mathcal{A},\mathcal{B})$
is also an abelian category where limits, colimits, kernels and cokernels
are computed point-wise \cite[Prop. 1.4.4]{borceux1994handbook}.
In case $\mathcal{A}$ is not small and $\mathcal{B}$ is an abelian
category, we will take the liberty to continue thinking of $\Fun(\mathcal{A},\mathcal{B})$
as an abelian category where limits, colimits, kernels and cokernels
are computed point-wise. 

\subsection{Godement product }

Recall that, for functors $F_{1}:\mathcal{X}\rightarrow\mathcal{Y}$,
$F_{2}:\mathcal{X}\rightarrow\mathcal{Y}$, $G_{1}:\mathcal{Y}\rightarrow\mathcal{Z}$ and
$G_{2}:\mathcal{Y}\rightarrow\mathcal{Z}$, and natural transformations
$a:F_{1}\rightarrow F_{2}$ and $b:G_{1}\rightarrow G_{2}$, the \textbf{Godement
product} $b\cdot a:G_{1}F_{1}\rightarrow G_{2}F_{2}$ is the natural transformation 
defined by 
$
\left(b\cdot a\right)_{X}:=b_{F_{2}X}\circ G_{1}(a_{X})=G_{2}(a_{X})\circ b_{F_{1}X}\quad\text{\ensuremath{\forall X\in\mathcal{X}}}.
$
Moreover, we set $b\cdot F_{1}:=b\cdot1_{F_{1}}$ and $G_{1}\cdot a:=1_{G_{1}}\cdot a.$ Notice that, for any functor $F:\mathcal{X}\rightarrow\mathcal{Y},$ the Godement product induces the functor
 $$F^{*}:\Fun(\mathcal{Y},\mathcal{Z})\rightarrow\Fun(\X,\Z),\;(G_1\xrightarrow{b} G_2)\mapsto (G_1F\xrightarrow{b\cdot F} G_2F),$$
which preserves monomorphisms and epimorphisms. In particular, $F^*$ is an exact functor if $\mathcal{Z}$ is abelian.

\subsection{Abelian categories and cardinal numbers}

Throughout the paper,  $\C$ denotes an abelian category. It is also used the Grothendieck's notation $AB3,$ $AB4$ and $AB5$. We will also consider the dual conditions $AB3^*,$  $AB4^*$ and $AB5^*.$ 
\

For an infinite cardinal $\kappa,$ we recall from \cite[Def.  A.3.1]{neeman2001triangulated}
the following conditions over the abelian category $\C.$ It is said
that $\mathcal{C}$ is $AB3(\kappa)$ if it admits coproducts
of sets of cardinality $<\kappa$; and $\mathcal{C}$ is $AB4(\kappa)$
if it is $AB3(\kappa)$ and the coproduct functor over all set-indexed of cardinality $<\kappa$ is exact. As usual, the dual conditions are denoted, respectively,
by $AB3^{*}(\kappa)$ and $AB4^{*}(\kappa).$ Notice that every abelian
category is $AB4(\aleph_{0})$ and $AB4^{*}(\aleph_{0}).$ Furthermore,
for any set $I,$ we denote by $|I|$ its cardinal number.

We recall \cite[Lem. 3.4]{jech2003set} that
for any set $X$ of cardinal numbers $\sup(X)$ is also a cardinal
number. In particular, for each cardinal number $\alpha,$ there exists
the least cardinal number greater than $\alpha,$ that is the cardinal
successor of $\alpha$ and it is usually denoted by $\alpha^{+}.$ Thus, for any $n\in\mathbb{N},$ the cardinal number $\aleph_n$ is 
defined recursively as follows $\aleph_0:=|\mathbb{N}|$ and $\aleph_{n+1}:=\aleph^{+}_n.$ Recall that a cardinal $\lambda$ 
is \textbf{regular} if it is an
infinite cardinal which is not a sum of less than $\lambda$ cardinals,
all smaller than $\lambda$. In what follows, we give a list of examples
illustrating the above notions.

\begin{example} Let $\kappa$ be a regular cardinal. 
\

(a)  Let $k$ be a field and $\mathcal{C}$ be the category of $k$-vector
spaces of dimension $<\kappa$. It can be shown that $\mathcal{C}$
is an AB4($\kappa$) abelian category.
\

(b) Let $R$ be a ring and $\modu_{\kappa}(R)$ be the category of left
$R$-modules with a generating set of cardinality $<\kappa$. In general
$\modu_{\kappa}(R)$ is not an abelian subcategory of $\Mod(R)$ since
it is not closed under submodules. However, if every left ideal of
$R$ is generated by a set of cardinality $<\kappa$, then $\modu_{\kappa}(R)$
is an abelian subcategory of $\Mod(R),$ see \cite[Lem. 6.31]{Approximations}.
Moreover, $\modu_{\kappa}(R)$ is AB4($\kappa$).
\

(c)  Let $\mathcal{G}$ be a category. We recall \cite[Def. 1.13]{adamek1994locally}
that, for a regular cardinal $\lambda$, an object $X\in\mathcal{G}$
is \textbf{$\lambda$-presentable} if $\Hom_{\mathcal{G}}(X,-)$ preserves
$\lambda$-directed colimits. Let $\mathcal{G_{\lambda}}$ be the
class of $\lambda$-presentable objects. It is a known fact that,
see \cite[Cor. 2.5.23]{krause2021homological}, for every Grothendieck
abelian category $\mathcal{G}$, there is a regular cardinal $\kappa$
such that $\mathcal{G}_{\lambda}$ is an abelian subcategory of $\mathcal{G}$
for all regular cardinals $\lambda\geq\kappa.$ Moreover, for $\lambda\geq\kappa$,
$\mathcal{G}_{\lambda}$ is AB4($\lambda$). Indeed, a coproduct indexed
by a set $I$ with $|I|<\lambda$ can be seen as a colimit indexed
by $F(I):=\{J\subseteq I\,|\:|J|<\aleph_{0}\}$. Therefore, since
$|F(I)|=|I|$ (see \cite[p.52]{jech2003set}), we can conclude that
$\mathcal{G}_{\lambda}$ is closed under coproducts of $<\lambda$
objects by \cite[Lem. 2.5.11]{krause2021homological} and hence,
$\mathcal{G}_{\lambda}$ is AB4($\lambda$) since $\mathcal{G}$ is
AB4.
\

(d)  Assume the generalized continuum hypothesis. In particular, we have
that $\aleph_{0}^{\aleph_{0}}=2^{\aleph_{0}}=\aleph_{1}$ \cite[p.55]{jech2003set}.
Let $\kappa:=\aleph_{2}$ and $\mathcal{C}:=\modu_{\kappa}(\mathbb{Z})\cap\mathcal{T}$,
where $\mathcal{T}$ is the category of torsion abelian groups. On
one hand, it is a known fact that $\mathcal{T}$ is an abelian subcategory
of $\Mod(\mathbb{Z})$ which is closed under coproducts. On the other hand,
it follows from (b) that $\modu_{\kappa}(\mathbb{Z})$ is an AB4($\kappa$)
abelian subcategory of $\Mod(\mathbb{Z}).$ Thus, $\mathcal{C}$ is
an abelian AB4($\kappa$) category. Moreover, we know that the torsion
radical $t:\Mod(\mathbb{Z})\rightarrow\mathcal{T}$ maps products
in $\Mod(\mathbb{Z})$ to products in $\mathcal{T}$ \cite[p.63, Ex. 4]{Popescu}.
Therefore, since every countable product of objects in $\modu_{\kappa}(\mathbb{Z})$
(which are sets of cardinality $<\kappa$) is of cardinality $<\kappa^{\aleph_{0}}=\aleph_{2}^{\aleph_{0}}=\aleph_{2},$
we have that $\mathcal{C}$ is AB3{*}($\aleph_{1}$) (the last equality
follows from the Hausdorff formula, see \cite[p.57]{jech2003set}).
However, note that it is not AB4{*}($\aleph_{1}$)  \cite[p.63, Ex. 4]{Popescu}. 
\end{example}

Let $\kappa$ be an infinite cardinal number and $\X\subseteq\C.$ We denote
by $\smd(\X)$ the class of all the direct summands in $\C$ of objects in
$\X,$ and by $\coprod_{<\kappa}\X$ the class of all the coproducts in $\C$
of families $\{X_{i}\}_{i\in I}$ of objects in $\X$ such that $|I|<\kappa;$
and $\Add_{<\kappa}(\X):=\smd(\coprod_{<\kappa}\X).$ Dually, we introduce
$\prod_{<\kappa}\X$ and $\Prod_{<\kappa}(\X):=\smd(\prod_{<\kappa}\X).$
Similarly, we can introduce the same notions for $\leq$ instead of
$<.$ Note that $\Add_{\leq\kappa}(\X)=\Add_{<\kappa^{+}}(\X)$ and
the same is true for the product. Moreover, we also have the class
$\add(\X):=\Add_{<\aleph_0}(\X).$

\subsection{The Yoneda's extension groups}

Let $\C$ be an abelian category. Given $A,B\in\mathcal{C}$, the
class of all equivalence classes of short exact sequences in $\C$
of the form $\eta:\suc[B][E][A][\,][\,]$ is denoted by $\Ext_{\C}^{1}(A,B),$
where $\hookrightarrow$ stands for a monomorphism in $\C$
and $\twoheadrightarrow$ represents an epimorphism in $\C$. Moreover
the respective equivalence class associated to $\eta$ is denoted
by $\overline{\eta}.$ Similarly, for $n>1$, the class of all
equivalence classes of exact sequences in $\C$ of the form $\eta:\suc[B][E_{n}\rightarrow\cdots\rightarrow E_{1}][A][\,][\,]$
is denoted by $\Ext_{\C}^{n}(A,B)$. It is a known fact, see more details in \cite[Chap. VII]{mitchell}, that  we get 
an additive bifunctor $\Ext_{\C}^{n}(?,-):\C^{op}\times\C\rightarrow\Ab.$

\subsection{Adjoint functors}

Let $S:\mathcal{D}\rightarrow\mathcal{C}$ and $T:\mathcal{C}\rightarrow\mathcal{D}$
be covariant functors. Recall that $T$ is \textbf{right adjoint}
for $S$ (or $S$ is \textbf{left adjoint} for $T$) if there is a
natural equivalence $h:\Hom_{\C}(S(?),-)\xrightarrow{\sim}\Hom_{\mathcal{D}}(?,T(-))$ of bifunctors. 
In this case, we say that $(S,T)$ is an \textbf{adjoint pair} between
$\mathcal{C}$ and $\mathcal{D}$. For all $C\in\C$ and $D\in\mathcal{D}$,
we will consider the morphisms
$\varphi_{D}:=h_{D,S(D)}(1_{S(D)}):D\rightarrow TS(D)$ and $\psi_{C}:=h_{T(C),C}^{-1}(1_{T(C)}):ST(C)\rightarrow C.$ 
It is a known fact that, for all $\alpha\in\Hom_{\C}(S(D),C)$ and
$\beta\in\Hom_{\mathcal{D}}(D,T(C))$, we have that $h_{D,C}(\alpha)=T(\alpha)\varphi_{D}$
and $h_{D,C}^{-1}(\beta)=\psi_{C}S(\beta)$ \cite[p.118]{mitchell}.
Moreover, it turns out that these morphisms define natural transformations
$\varphi:1_{\mathcal{D}}\rightarrow TS$ and $\psi:ST\rightarrow1_{\C}$
\cite[Chap. V, Prop. 1.1]{mitchell}, which are known, respectively, as the unit and the co-unit of the adjoint pair.

Let us recall the following well-known facts about the unit and the co-unit.

\begin{lem} \cite[Chap. V, Prop. 1.1]{mitchell}\label{lem: mitchell adjoint pairs}
For the adjoint pair $(S,T)$ between $\C$ and $\mathcal{D},$ the following statements hold true. 
\begin{enumerate}
\item For $\beta:B\rightarrow T(A)$ in $\mathcal{D}$, $\alpha:=h^{-1}(\beta):S(B)\rightarrow A$
is the unique morphism  such that $T(\alpha)\circ\varphi_{B}=\beta$. 
\item For $\gamma:S(B)\rightarrow A$ in $\mathcal{C}$, $\delta:=h(\gamma):B\rightarrow T(A)$
is the unique morphism  such that $\psi_{A}\circ S(\delta)=\gamma$. 
\end{enumerate}
\end{lem}

\begin{lem}\label{lem: coproducto pares adjuntos} Let $\kappa$ be an infinite cardinal,
$\mathcal{C}$ be $AB3(\kappa)$ and $\mathcal{D}$
be $AB3^{*}(\kappa).$  Then, for a set $\{(S_{i},T_{i})\}_{i\in I}$
of adjoint pairs between $\mathcal{C}$ and $\mathcal{D}$ with $|I|<\kappa$,
we have that $(\coprod_{i\in I}S_{i},\prod_{i\in I}T_{i})$ is an
adjoint pair between $\mathcal{C}$ and $\mathcal{D}.$
\end{lem}

\subsection{Adjoint functors and extension groups}

Let $\mathcal{C}$ and $\mathcal{D}$ be abelian categories.
If there is an equivalence $\Hom_{\C}(S(?),-)\xrightarrow{\sim}\Hom_{\mathcal{D}}(?,T(-)),$ it is quite natural to ask if the above equivalence induces a natural one $\Ext_{\C}^{n}(S(D),C)\xrightarrow{\sim}\Ext_{\mathcal{D}}^{n}(D,T(C))\text{.}$
For example, in the literature we can find the following results \cite[Prop. 3.10]{odabacsi2019completeness},
\cite[Lem. 5.1]{holm2019cotorsion}, \cite[Thm. 4.3 and 4.4]{argudin2019yoneda},
\cite[Thm. 4.2]{argudin2022exactness}. Inspired by them, in this
section we will seek to give some basic results in a general context.

Throughout this section, we will consider the functors $S:\mathcal{D}\rightarrow\mathcal{C}$
and $T:\mathcal{C}\rightarrow\mathcal{D},$ where $(S,T)$ is an adjoint
pair between the abelian categories $\mathcal{C}$ and $\mathcal{D}.$
In this case, it is known that $S$ is right exact and $T$ is left
exact \cite[Chap. V, Sect. 5]{maclane}. For an exact sequence
$\eta:\suc[B][E][A][\,][\,]$ and an exact functor $F$, we denote by
$F(\eta)$ the exact sequence $F(B)\hookrightarrow F(E_{1})\twoheadrightarrow F(A)$. 

\begin{lem}\label{lem:Ext vs adjoint 1} For the adjoint pair $(S,T)$ between $\C$ and $\mathcal{D},$ the following
statements hold true. 
\begin{enumerate}
\item Let $T$ be exact and $n\geq1.$ Then the map 
\[
\tau{}_{D,C}^{n}:\Ext_{\C}^{n}(S(D),C)\rightarrow\Ext_{\mathcal{D}}^{n}(D,T(C)),\quad\overline{\eta}\mapsto\overline{T(\eta)}\cdot\varphi_{D},
\]
induces the morphism $\tau^{n}:\Ext_{\C}^{n}(S(?),-)\rightarrow\Ext_{\mathcal{D}}^{n}(?,T(-))$
of functors. Moreover $\tau{}_{D,C}^{1}$ is injective, for all $D\in\mathcal{D}$
and $C\in\mathcal{C}$. 
\item Let $S$ be exact and $n\geq1.$ Then the map 
\[
\sigma{}_{D,C}^{n}:\Ext_{\mathcal{D}}^{n}(D,T(C))\rightarrow\Ext_{\C}^{n}(S(D),C),\quad\overline{\eta}\mapsto\psi_{D}\cdot\overline{S(\eta)},
\]
induces the morphism $\sigma^{n}:\Ext_{\mathcal{D}}^{n}(?,T(-))\rightarrow\Ext_{\C}^{n}(S(?),-)$
of functors. Moreover $\sigma{}_{D,C}^{1}$ is injective, for all
$D\in\mathcal{D}$ and $C\in\mathcal{C}$. 
\end{enumerate}
\end{lem}

\begin{minipage}[t]{0.65\columnwidth}%
\begin{proof}
Let $T$ be exact and $n\geq1.$ Firstly, note that, for any $\overline{\eta}\in\Ext_{\C}^{n}(X,Y),$
$g:Y\to Y'$ and $f:X'\to X,$ we have that $\overline{T(g\cdot\eta)}=T(g)\cdot\overline{T(\eta)}$
and $\overline{T(\eta\cdot f)}=\overline{T(\eta)}\cdot T(f)$ since
$T:\C\to\mathcal{D}$ is an exact functor. Thus, the naturality of
$\tau{}_{D,C}^{n}$ follows straightforward by using \cite[Chap. V, Prop. 1.1]{mitchell}
and \cite[Chap. VII, Lem.1.3]{mitchell}. \
 Let us show that $\tau{}_{D,C}^{1}$ is injective. Consider an exact
sequence $\eta:\;\suc[C][E][SD][f][g]$ such that $\tau{}_{D,C}^{1}(\overline{\eta})=0.$
Then, there is $\beta\in\Hom_{\mathcal{D}}(D,TE)$ such that $\left(Tg\right)\circ\beta=\varphi_{D}$.
Therefore $\overline{\eta}=0$ since
\[
g\circ(\psi_{E}\circ S\beta)=\psi_{SD}\circ S\varphi_{D}=h_{D,SD}^{-1}(\varphi_{D})=1_{SD}\text{.}
\]
\end{proof}
\end{minipage}\hfill{}%
\fbox{\begin{minipage}[t]{0.3\columnwidth}%
\[ \begin{tikzpicture}[-,>=to,shorten >=1pt,auto,node distance=1cm,main node/.style=,x=2cm,y=1.5cm]
 \node[main node] (1) at (0,0){$E$};  \node[main node] (2) at (1,0){$S(D)$};  \node[main node] (3) at (0,1){$ST(E)$};  \node[main node] (4) at (1,1){$STS(D)$};  \node[main node] (5) at (0,2){$S(E')$};  \node[main node] (6) at (1,2){$SD$};
\draw[->>, thin]   (5)  to  node  {$$}    (6); \draw[->>, thin]   (3)  to  [below] node  {$ST(g)$}    (4); \draw[->>, thin]   (1)  to  node  {$g$}    (2); \draw[->, thin]   (5)  to  node  {$$}    (3); \draw[->, thin]   (3)  to  node  {$\psi _E$}    (1); \draw[->, thin]   (6)  to  node  {$S(\varphi_{D})$}    (4); \draw[->, thin]   (4)  to  node  {$\psi _{S(D)}$}    (2); \draw[->, thin]   (6)  to [above left]  node  {$S(\beta)$}    (3);
   
    \end{tikzpicture} \]
\end{minipage}}
\begin{prop}\label{lem:Ext vs adjoint 3} For the adjoint pair $(S,T)$ between $\C$ and $\mathcal{D},$ the following
statements hold true.
\begin{enumerate}
\item Let $T:\C\to\mathcal{D}$ be exact. Then, the following statements
are equivalent. 
\begin{enumerate}
\item $\tau^{1}:\Ext_{\C}^{1}(S(?),-)\xrightarrow{\sim}\Ext_{\mathcal{D}}^{1}(?,T(-)).$ 
\item $S:\mathcal{D}\to\C$ is exact. 
\item $\tau^{n}:\Ext_{\C}^{n}(S(?),-)\xrightarrow{\sim}\Ext_{\mathcal{D}}^{n}(?,T(-))$
$\forall\,n\geq1.$ 
\end{enumerate}
\item Let $S:\mathcal{D}\to\C$ be exact. Then, the following statements
are equivalent. 
\begin{enumerate}
\item $\sigma^{1}:\Ext_{\mathcal{D}}^{1}(?,T(-))\xrightarrow{\sim}\Ext_{\C}^{1}(S(?),-).$ 
\item $T:\C\to\mathcal{D}$ is exact. 
\item $\sigma^{n}:\Ext_{\mathcal{D}}^{n}(?,T(-))\xrightarrow{\sim}\Ext_{\C}^{n}(S(?),-)$
$\forall\,n\geq1.$ 
\end{enumerate}
\item If $S$ and $T$ are exact functors, then $\left(\tau^{n}\right)^{-1}=\sigma^{n}$
$\forall\,n\geq1.$ 
\end{enumerate}
\end{prop}

\begin{proof}
Note that (b) is the dual of (a), and (c) follows by the proof of
\cite[Lem. 5.1]{holm2019cotorsion} and Lemma \ref{lem:Ext vs adjoint 1}.
\
 Let us show (a). Since $T$ is exact, for each $n\geq1,$ we have
by Lemma \ref{lem:Ext vs adjoint 1} the morphism of functors $\tau^{n}:\Ext_{\C}^{n}(S(?),-)\rightarrow\Ext_{\mathcal{D}}^{n}(?,T(-))$.
Now, the implication $(a2)\Rightarrow(a3)$ follows by the proof of
\cite[Lem. 5.1]{holm2019cotorsion}; and $(a3)\Rightarrow(a1)$ is
trivial. Thus, it is enough to prove $(a1)\Rightarrow(a2)$. Indeed,
let $\eta:\:\suc[X][Y][Z][a][b]$ be an exact sequence in $\mathcal{D}$.
Since $S$ is right exact, we need only to show that $S(a)$ is a
monomorphism. Consider the exact sequence $\eta':\:\suc[TS(X)][Y'][Z][f][g]$
in $\mathcal{D}$ such that $\overline{\eta'}=\varphi_{X}\cdot\overline{\eta}.$
Since $\tau{}_{Z,SX}^{1}:\Ext_{\C}^{1}(S(Z),S(X))\xrightarrow{\sim}\Ext_{\mathcal{D}}^{1}(Z,TS(X)),$
there is an exact sequence $\eta'':\:\suc[S(X)][Y''][S(Z)][f'][g']$
in $\mathcal{C}$ such that $\overline{T(\eta'')}\cdot\varphi_{Z}=\overline{\eta'}.$
Thus, by using that $\overline{T(\eta'')}\cdot\varphi_{Z}=\overline{\eta'}=\varphi_{X}\cdot\overline{\eta},$
$\psi:ST\to1_{\C}$ and $\psi_{SZ}\circ S(\varphi_{Z})=1_{SZ},$ we
get the following exact and commutative diagram in $\C$\\
\begin{minipage}[t]{1\columnwidth}%
\[  \begin{tikzpicture}[-,>=to,shorten >=1pt,auto,node distance=2.5cm,main node/.style=,x=1.2cm,y=1.2cm]    

      \node[main node] (2) at (0,0)  {$STS(X)$};  
     \node[main node] (3) [right of=2]  {$S(Y') $};  
     \node[main node] (4) [right of=3]  {$S(Z)$};  

       \node[main node] (2') at (0,-1)   {$ S(X)$}; 
      \node[main node] (3') [right of=2']  {$ Y''$};        \node[main node] (4') [right of=3']  {$S(Z) $};   

     \node[main node] (2'') at (0,1)   {$S(X)$};      
\node[main node] (3'') [right of=2'']  {$S(Y)$};  
     \node[main node] (4'') [right of=3'']  {$S(Z) $};  
\node[main node] (A) at (-1,-1)  {$\eta '':$};       
\node[main node] (B) at (-1,1)  {$S(\eta ) :$};      
\node[main node] (B) at (-1,0)  {$S(\eta ') :$};

\draw[->, thin]   (2)  to node  {$\scriptstyle S( f)$}  (3); 
\draw[->>, thin]   (3)  to node  {$\scriptstyle S(  g)$}  (4);
\draw[right hook->, thin]   (2')  to node  {$\scriptstyle   f'$}  (3');

\draw[->>, thin]   (3')  to node  {$\scriptstyle   g'$}  (4');
\draw[->, thin]   (2'')  to node  {$\scriptstyle   S(a)$}  (3'');
\draw[->>, thin]   (3'')  to node  {$\scriptstyle  S(b)$}  (4''); 
\draw[->, thin]   (2)  to node  {$  \psi _{SX}$}  (2');
\draw[->, thin]   (3)  to node  {$  \psi '$}  (3');
\draw[-, double]   (4)  to node  {$$}  (4');
\draw[->, thin]   (2'')  to node  {$S(\varphi _{X})$}  (2); 
\draw[->, thin]   (3'')  to node  {$S(\varphi ')$}  (3);
\draw[-, double]   (4'')  to node  {$$}  (4);  
  \end{tikzpicture} 
\]
\end{minipage}\\
Now, since $\psi_{SX}\circ S(\varphi_{X})=1_{SX}$, we have that $\psi'\circ S(\varphi')\circ S(a)$
is a monomorphism. Therefore, $S(a)$ is a monomorphism. 
\end{proof}

\subsection{Cotorsion pairs}

Let $\mathcal{C}$ be an abelian category, $\mathcal{A},\mathcal{B}\subseteq\mathcal{C}$
and $\mathcal{X}\subseteq\mathcal{C}$. Following \cite[sect. 3]{AM21},
we recall that $\p$ is a \textbf{left} (resp. \textbf{right}) cotorsion pair in $\mathcal{X}$
 if $\mathcal{A}\cap\mathcal{X}={}^{\bot_{1}}\mathcal{B}\cap\mathcal{X}$
(resp. $\mathcal{B}\cap\mathcal{X}=\mathcal{A}^{\bot_{1}}\cap\mathcal{X}$
). Moreover $\p$ is a \textbf{cotorsion pair in $\mathcal{X}$} if
it is both left and right cotorsion pair in $\mathcal{X}$. Cotorsion
pairs are known for their relation with approximations. Namely, given
$\mathcal{Z}\subseteq\mathcal{C}$, we say that a morphism $f:Z\rightarrow M$
is called a \textbf{$\mathcal{Z}$-precover} if $Z\in\mathcal{Z}$
and $\Hom_{\mathcal{C}}(Z',f)$ is an epimorphism $\forall Z'\in\mathcal{Z}$.
In case $f$ fits in a short exact sequence $\suc[M'][Z][M][\,][\,]$,
where $M'\in\mathcal{Z}^{\bot_{1}}$, we say that $f$ is a \textbf{special
$\mathcal{Z}$-precover}. It is said \cite[Def. 3.12]{AM21}
that $\mathcal{Z}$ is \textbf{special precovering in} $\mathcal{X}$
if any $X\in\mathcal{X}$ admits an exact sequence 
$\suc[B][A][X],$ with $A\in\mathcal{Z}\cap\mathcal{X}$ and $B\in\mathcal{Z}^{\bot_{1}}\cap\mathcal{X}.$
The notions of \textbf{$\mathcal{Z}$-preenvelope}, \textbf{special
$\mathcal{Z}$-preenvelope} and \textbf{special preenveloping in
$\mathcal{X}$} are defined dually.

We recall that a cotorsion pair $\p$ is \textbf{left complete} if
$\mathcal{A}$ is special precovering in $\mathcal{C}$. As a generalization,
see \cite[Def. 3.13]{AM21}, we say that a (not necessarily
cotorsion) pair $\p\subseteq\mathcal{C}^{2}$ is \textbf{left $\mathcal{X}$-complete} if any $X\in\mathcal{X}$ admits an exact sequence $\suc[B][A][X]$,
with $A\in\mathcal{A}\cap\mathcal{X}$ and $B\in\mathcal{B}\cap\mathcal{X}$.
\textbf{Right $\mathcal{X}$-completeness} is defined dually, and
\textbf{$\mathcal{X}$-completeness} means right and left $\mathcal{X}$-completeness.
Finally, by \cite[Def. 3.7]{AM21}, we say that a pair is \textbf{$\mathcal{X}$-hereditary}
if $\Ext_{\mathcal{C}}^{i}(A,B)=0$ for all $i\geq1,$ $\forall\,A\in\mathcal{A}\cap\mathcal{X}$
and $\forall\,B\in\mathcal{B}\cap\mathcal{X}$.

\subsection{Quivers and subquivers}

In what follows, $Q$ is a quiver (an oriented graph) defined by a
set of vertices $Q_{0}$ and a set of arrows $Q_{1}$, together with
two functions: a source map $s:Q_{1}\rightarrow Q_{0}$ and a target
map $t:Q_{1}\rightarrow Q_{0}$. We will give no restrictions on this
quiver unless otherwise mentioned. In particular, $Q_{0}$ and $Q_{1}$
are allowed to be infinite sets. The category of paths of $Q$ is
the category $\mathcal{C}_{Q}$ whose objects are the vertices in
$Q_{0};$ the morphisms (paths $\gamma$ of length $n,$ for $n\geq0$) from
a vertex $i$ to a vertex $j$ are given by the finite sequences of
$n$-arrows $\gamma:=\alpha_{n}\cdots\alpha_{1},$ where $s(\alpha_{k+1})=t(\alpha_{k})$
$\forall k\in[1,n-1]$ such that $s(\gamma):=s(\alpha_{1})=i$ and $t(\gamma):=t(\alpha_{n})=j$; the identity morphism of each vertex $k\in Q_0$ (paths $\epsilon_k$ of length zero) and the
composition in $\mathcal{C}_{Q}$ is given by the concatenation of
these paths. Throughout the paper $Q(i,j):=\mathcal{C}_{Q}(i,j)$ denotes
the set of paths from $i$ to $j$. If $Q$ does not have oriented
cycles, we say that $Q$ is \textbf{acyclic}.

Following \cite{holm2019cotorsion}, for each $i\in Q_{0}$, we consider
$Q_{1}^{i\rightarrow*}:=\{\alpha\in Q_{1}\,|\:s(\alpha)=i\}$ (the right mesh at $i$) and
$Q_{1}^{*\rightarrow i}:=\{\alpha\in Q_{1}\,|\:t(\alpha)=i\}$ (the left mesh at $i$). We
say that: $Q$ is \textbf{locally finite }if $Q_{1}^{*\rightarrow i}$
and $Q_{1}^{i\rightarrow*}$ are finite sets for all $i\in Q_{0}$;
$Q$ is \textbf{interval-finite} if $Q(i,j)$ is finite for all $i,j\in Q_{0}$;
$Q$ is \textbf{strongly locally finite} if it is locally finite and
interval-finite \cite{bautista2013representation}. 

Following \cite{enochs2004flat}, we recall that a quiver $Q$ is
\textbf{right-rooted} if there exists no infinite sequence $\bullet\rightarrow\bullet\rightarrow\bullet\rightarrow\bullet\rightarrow\cdots$
of not necessarily different composable arrows in $Q$. Dually, $Q$
is \textbf{left-rooted} if the opposite quiver $Q^{op}$ is right-rooted.
Moreover $Q$ is \textbf{rooted} if it is left and right rooted. In
particular, any right or left rooted quiver is always acyclic.

\begin{defn}
A quiver $Q$ is of \textbf{finite-cone-shape } type if the \textbf{ left
cone} $Q(-,j):=\bigcup_{i\in Q_{0}}Q(i,j)$ and the \textbf{ right
cone} $Q(j,-):=\bigcup_{i\in Q_{0}}Q(j,i)$ are finite sets $\forall j\in Q_{0}$. We also say that $Q$ is of {\bf right support finite} type if the set of vertices $t(Q(i,-))$ is finite $\forall\, i\in Q_0.$ Dually, $Q$  is of {\bf left support finite} type if the opposite quiver $Q^{op}$ is of right support finite type. Finally, $Q$ is of {\bf support finite} type if it is of left and right support finite type. 
\end{defn}

Notice that if $Q$ is finite-cone-shape then it is also support finite. However the converse fails in general. For example, take a finite-cone-shape quiver $Q$ and construct a new one $Q'$ from $Q$ by adding loops at some vertices of $Q.$ Then $Q'$ will be support finite but not finite-cone-shape. We point out that a finite-cone-shape quiver is strongly locally finite, acyclic
and rooted.

\begin{example} Let $Q$ be $A_{\infty}^{\infty}$, $D_{\infty}$ or $A_{\infty}$
with the zig-zag orientation (see the diagrams below). In this case, we have that $Q$ is a finite-cone-shape quiver. 
\\
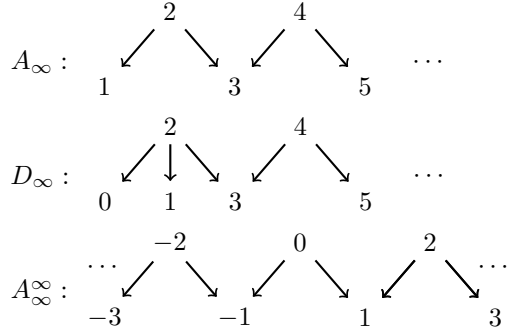
\begin{minipage}[t]{1\columnwidth}%
\[ \begin{tikzpicture}[-,>=to,shorten >=1pt,auto,node distance=3cm,main node/.style=,x=0.866025cm,y=1cm]
  \coordinate  (a)   at (4,0);   \coordinate  (b)   at (6,0);   \coordinate  (c)   at (5,1);   \coordinate  (a'')   at (-2,0);   \coordinate  (b'')   at (0,0);   \coordinate  (c'')   at (-1,1);
\node (1) at (0,0) {$1$}; \node (3) at (2,0) {$3$}; \node (2) at (1,1) {$2$}; \node (5) at (4,0) {$5$}; \node (4) at (3,1) {$4$};
\node (u) at (8,0) {$$}; \node (v) at (9,1) {$$};
\node (c) at (barycentric cs:a=1 ,b=1 ,c=1) {$\cdots $}; \node (c'') at (barycentric cs:a''=1 ,b''=1 ,c''=1) {$A_\infty :$};
\draw[->, thick]  (2)  to  node  {$$} (1); \draw[->, thick]  (2)  to  node  {$$} (3); \draw[->, thick]  (4)  to  node  {$$} (3); \draw[->, thick]  (4)  to  node  {$$} (5);
\end{tikzpicture} \]
\end{minipage}\\
\begin{minipage}[t]{1\columnwidth}%
\[ \begin{tikzpicture}[-,>=to,shorten >=1pt,auto,node distance=3cm,main node/.style=,x=0.866025cm,y=1cm]
  \coordinate  (a)   at (4,0);   \coordinate  (b)   at (6,0);   \coordinate  (c)   at (5,1);   \coordinate  (a'')   at (-2,0);   \coordinate  (b'')   at (0,0);   \coordinate  (c'')   at (-1,1);
\node (0) at (0,0) {$0$}; \node (1) at (1,0) {$1$}; \node (3) at (2,0) {$3$}; \node (2) at (1,1) {$2$}; \node (5) at (4,0) {$5$}; \node (4) at (3,1) {$4$};
\node (u) at (8,0) {$$}; \node (v) at (9,1) {$$};
\node (c) at (barycentric cs:a=1 ,b=1 ,c=1) {$\cdots $}; \node (c'') at (barycentric cs:a''=1 ,b''=1 ,c''=1) {$D_\infty :$};
\draw[->, thick]  (2)  to  node  {$$} (0); \draw[->, thick]  (2)  to  node  {$$} (1); \draw[->, thick]  (2)  to  node  {$$} (3); \draw[->, thick]  (4)  to  node  {$$} (3); \draw[->, thick]  (4)  to  node  {$$} (5);
\end{tikzpicture} \]
\end{minipage}\\
\begin{minipage}[t]{1\columnwidth}%
\[ \begin{tikzpicture}[-,>=to,shorten >=1pt,auto,node distance=3cm,main node/.style=,x=0.866025cm,y=1cm]
  \coordinate  (a)   at (4,0);   \coordinate  (b)   at (6,0);   \coordinate  (c)   at (5,1);   \coordinate  (c')   at (7,1);
\coordinate  (A)   at (2,0);   \coordinate  (B)   at (0,0);   \coordinate  (C)   at (1,1);   \coordinate  (A')   at (-2,0);   \coordinate  (B')   at (-4,0);   \coordinate  (C')   at (-1,1);
\node (1) at (0,0) {$-3$}; \node (3) at (2,0) {$-1$}; \node (2) at (1,1) {$-2$}; \node (5) at (4,0) {$1$}; \node (4) at (3,1) {$0$}; \node (6) at (5,1) {$2$}; \node (7) at (6,0) {$3$};
\node (u) at (8,0) {$$}; \node (v) at (9,1) {$$};
\node (c) at (barycentric cs:b=1 ,c'=1 ,c=1) {$\cdots $}; \node (C) at (barycentric cs:B=1 ,C'=1 ,C=1) {$\cdots $}; \node (c'') at (barycentric cs:A'=1 ,B=1 ,C'=1) {$A_\infty ^{\infty} :$};
\draw[->, thick]  (4)  to  node  {$$} (3); \draw[->, thick]  (4)  to  node  {$$} (5); \draw[->, thick]  (6)  to  node  {$$} (5); \draw[->, thick]  (6)  to  node  {$$} (7); \draw[->, thick]  (6)  to  node  {$$} (5); \draw[->, thick]  (6)  to  node  {$$} (7); \draw[->, thick]  (2)  to  node  {$$} (1); \draw[->, thick]  (2)  to  node  {$$} (3);
\end{tikzpicture} \]
\end{minipage}\\

We point out that 
I. Reiten and M. Van den Bergh studied the representations of these quivers in \cite[Sect. III.3]{reiten2002noetherian}.
\end{example}

\begin{minipage}[t]{0.5\columnwidth}%

\begin{rem}\label{VW} \cite[Sect. 3]{enochs2004flat}\label{rem:transfinite sequence and rooted quivers}
Let $Q$ be a quiver. For every ordinal $\gamma$ there is a set $V_{\gamma}\subseteq Q_{0}$
which is defined by transfinite induction as follows: $V_{0}:=\emptyset,$
\[ V_{\beta +1} := (Q_0 - t(Q_1(Q_0 - V_{\beta} , -))) \cup V_{\beta} \]
and $V_{\gamma}:=\bigcup_{\beta<\gamma}V_{\beta}$ for a limit ordinal
$\gamma.$ It turns out that $Q$ is left-rooted if and only if there
is an ordinal $\gamma$ such that $Q_{0}=V_{\gamma}$. \\
Dually, $Q$ is right-rooted if and only if there is an ordinal $\gamma$
such that $Q_{0}=W_{\gamma}$, where $W_{0}:=\emptyset$, 
 \[ W_{\alpha +1} = (Q_0 - s(Q_1 (-,Q_0 -W_{\alpha}))) \cup W_\alpha \]
and $W_{\gamma}:=\bigcup_{\alpha<\gamma}W_{\alpha}$ for a limit ordinal
$\gamma.$
\end{rem}

\end{minipage}\hfill{}%
\fbox{\begin{minipage}[t]{0.45\columnwidth}%
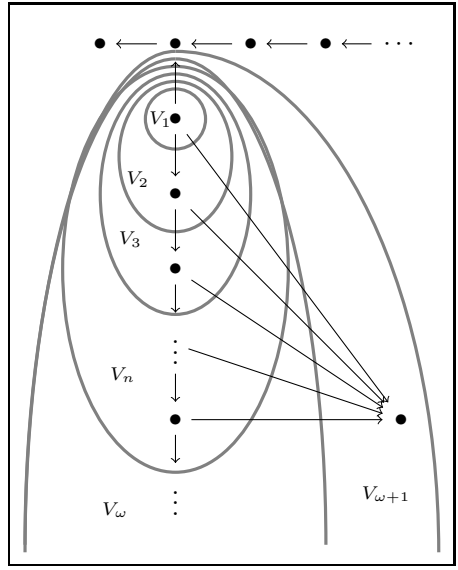
\[ \begin{tikzpicture}[-,>=to,shorten >=1pt,auto,node distance=1.5cm,main node/.style=]
\node (1) at (0,0) {$\bullet$}; \node (2) at (0,-1) {$\bullet$}; \node (3) at (0,-2) {$\bullet$}; \node (c1) at (0,-3) {$\vdots$}; \node (n) at (0,-4) {$\bullet$}; \node (w) at (3,-4) {$\bullet$}; \node (c2) at (0,-5) {$\vdots$};
\node (v1) at (-.2,0) {$\scriptstyle V_1$}; \node (v2) at (-.5,-.8) {$\scriptstyle V_2$}; \node (v3) at (-.6,-1.6) {$\scriptstyle V_3$}; \node (vn) at (-.7,-3.4) {$\scriptstyle V_n$}; \node (vo) at (-.8,-5.2) {$\scriptstyle V_\omega$}; \node (vo') at (2.8,-5) {$\scriptstyle V_{\omega +1}$};
\node (d) at (-1,1) {$\bullet$}; \node (a) at (0,1) {$\bullet$}; \node (b) at (1,1) {$\bullet$}; \node (c) at (2,1) {$\bullet$}; \node (v) at (3,1) {$\cdots$};
\draw[-, very thick, gray] (0,0) ellipse (.4cm and .4cm); \draw[-, very thick, gray] (0,-.5) ellipse (.75cm and 1cm); \draw[-, very thick, gray] (0,-1) ellipse (1.cm and 1.6cm); \draw[-, very thick, gray] (0,-2) ellipse (1.5cm and 2.7cm); \draw[-, very thick, gray] (0,.8)  arc(90:180:2cm and 6.5cm); \draw[-, very thick, gray] (0,.8)  arc(90:0:2cm and 6.5cm); \draw[-, very thick, gray] (0,.9)  arc(90:180:2cm and 6.7cm); \draw[-, very thick, gray] (0,.9)  arc(90:0:3.5cm and 6.7cm);
\draw[->, thin]  (1)  to  node  {$$} (2); \draw[->, thin]  (2)  to  node  {$$} (3); \draw[->, thin]  (3)  to  node  {$$} (c1); \draw[->, thin]  (c1)  to  node  {$$} (n); \draw[->, thin]  (n)  to  node  {$$} (c2);
\draw[->, thin]  (1)  to  node  {$$} (w); \draw[->, thin]  (2)  to  node  {$$} (w); \draw[->, thin]  (3)  to  node  {$$} (w); \draw[->, thin]  (c1)  to  node  {$$} (w); \draw[->, thin]  (n)  to  node  {$$} (w);
\draw[->, thin]  (1)  to  node  {$$} (a); \draw[<-, thin]  (d)  to  node  {$$} (a); \draw[<-, thin]  (a)  to  node  {$$} (b); \draw[<-, thin]  (b)  to  node  {$$} (c); \draw[<-, thin]  (c)  to  node  {$$} (v);
\end{tikzpicture} \]
\end{minipage}}

In order to reflect some properties from the abelian category $\C$
to the category of representations $\Rep(Q,\C),$ we will make use
of the following cardinal numbers which are (somehow) a measure of
the geometric complexity of the quiver $Q.$ Namely, the mesh and
the cone-shape cardinal numbers associated with the quiver $Q.$ In
order to do that,  for a set $X$ of cardinal numbers, we define the \textbf{size}
of $X,$ denoted by $\mathrm{size}(X),$ as follows:  $\mathrm{size}(X):=\sup(X)$
if $x<\sup(X)$ $\forall\,x\in X;$ and $\mathrm{size}(X):=\sup(X)^{+}$
if $x=\sup(X)$ for some $x\in X.$

\begin{defn}\label{def:mesh cardinal}For a quiver $Q,$ we have the following
mesh cardinal numbers. 
\begin{itemize}
\item[(a)] The \textbf{left mesh cardinal} number of $Q,$ $\lmcn(Q):=\mathrm{size}(\{|Q_{1}^{*\rightarrow i}|\}_{i\in Q_{0}}).$ 
\item[(b)] The \textbf{right mesh cardinal} number of $Q,$ $\rmcn(Q):=\mathrm{size}(\{|Q_{1}^{i\rightarrow*}|\}_{i\in Q_{0}}).$ 
\item[(c)] The \textbf{mesh cardinal} number of $Q,$ $\mcn(Q):=\max\{\lmcn(Q),\rmcn(Q)\}.$ 
\end{itemize}
\end{defn}

\begin{defn}\label{def:cone shape cardinal}For a quiver $Q,$ we have the following
cone-shape cardinal numbers. 
\begin{itemize}
\item[(a)] For each $i\in Q_{0},$ we have that: 
\begin{itemize}
\item[(a1)] $\lccn_{i}(Q):=\mathrm{size}(\{|Q(j,i)|\}_{j\in Q_{0}})$ is the
\textbf{left $i$-cone-shape cardinal} number of $Q;$ 
\item[(a2)] $\rccn_{i}(Q):=\mathrm{size}(\{|Q(i,j)|\}_{j\in Q_{0}})$ the \textbf{right
$i$-cone-shape cardinal} number of $Q;$ 
\end{itemize}
Hence $\lccn(Q):=\sup(\{\lccn_{i}(Q)\}_{i\in Q_{0}})$ is the \textbf{left
cone-shape cardinal} number of $Q;$ and $\rccn(Q):=\sup(\{\rccn_{i}(Q)\}_{i\in Q_{0}})$
the \textbf{right cone-shape cardinal} number of $Q.$ 
\item[(b)] $\ltccn(Q):=\mathrm{size}(\{|Q(-,i)|\}_{i\in Q_{0}})$ is the \textbf{left
thick cone-shape cardinal} number of $Q.$ 
\item[(c)] $\rtccn(Q):=\mathrm{size}(\{|Q(i,-)|\}_{i\in Q_{0}})$ is the \textbf{right
thick cone-shape cardinal} number of $Q.$ 
\item[(d)] $\ccn(Q):=\max\{\lccn(Q),\rccn(Q)\}$ is the \textbf{cone-shape cardinal}
number of $Q;$ and $\tccn(Q):=\max\{\ltccn(Q),\rtccn(Q)\}$ is the
\textbf{thick cone-shape cardinal} number of $Q.$ 
\end{itemize}
\end{defn}

Let $Q$ be a quiver. A \textbf{subquiver} of $Q$ is a quiver $Q'=(Q'_{0},Q'_{1})$
such that $Q'_{0}\subseteq Q_{0},$ $Q'_{1}\subseteq Q_{1}$, $s'=s|_{Q'_{1}}$
and $t'=t|_{Q'_{1}}.$ A subquiver
$Q'$ of $Q$ is \textbf{0-finite} if $Q'_{0}$ is finite. A subquiver
$Q'$ of $Q$ is \textbf{full} if every arrow $\alpha\in Q_{1},$
with $s(\alpha),t(\alpha)\in Q'_{0},$ belongs to $Q'_{1}$. Note
that a full subquiver $Q'$ of $Q$ is uniquely determined by its
vertex set $Q'_{0}.$ Hence, for the given full subquivers $Q'$ and
$Q''$ of $Q$, we can define $Q'\cup Q''$ as the full subquiver
with vertex set $Q'_{0}\cup Q''_{0}.$ We denote by $\F(Q)$ the set
of all the full subquivers of $Q$ which are $0$-finite. Now, let
$S$ be a full subquiver of $Q.$ Following \cite{green2017convex},
we consider the sets 
\begin{center}
$S^{-}:=\{i\in Q_{0}-S_{0}\;|\:\;\text{there is a path }\gamma\text{ with }i=s(\gamma)\text{ and }t(\gamma)\in S_{0}\},$ 
\par\end{center}

\begin{center}
$S^{+}:=\{i\in Q_{0}-S_{0}\;|\:\;\text{there is a path }\gamma\text{ with }i=t(\gamma)\text{ and }s(\gamma)\in S_{0}\}.$ 
\par\end{center}

We say that $S$ is in the \textbf{bottom} of $Q$ if the set $S^{-}$
is finite. In case $S^{+}$ is finite, we say that $S$ is in the
\textbf{top} of $Q.$ 

\begin{rem}\label{main-quivers} For a quiver $Q,$ the following statements
hold true. 
\begin{itemize}
\item[$\mathrm{(a)}$] If $\{H_i\}_{i\in I}$ is a non empty family of full subquivers of $Q,$ then we have that $(\cup_{i\in I}H_i)^{-}\subseteq \cup_{i\in I}H_i^{-}$
and $(\cup_{i\in I}H_i)^{+}\subseteq \cup_{i\in I}H_i^{+}.$ 
\item[$\mathrm{(b)}$] If $H$ and $S$ are full subquivers of $Q$ and $H\subseteq S,$
then $H^{-}\subseteq S^{-}\cup(S-H)$ and $H^{+}\subseteq S^{+}\cup(S-H).$ 
\end{itemize}
\end{rem}

\begin{defn}
Let $Q$ be a quiver. We denote by $\B(Q)$ (respectively, $\T(Q)$)
the set of all the full subquivers of $Q$ which are in the bottom
(respectively, top) of $Q.$ Thus $\F\B(Q):=\F(Q)\cap\B(Q)$ is the
set of all the full subquivers of $Q$ which are $0$-finite and in
the bottom of $Q.$ Similarly, we have $\F\T(Q):=\F(Q)\cap\T(Q)$
and $\F\B\T(Q):=\F(Q)\cap\B(Q)\cap\T(Q).$ 
\end{defn}

\subsection{Quiver representations}

Let $\mathcal{C}$ be an abelian category and $Q$ be a quiver. The
category of \textbf{$\mathcal{C}$-valued representations} $\operatorname{Rep}(Q,\mathcal{C})$
of a quiver $Q$ can be identified with the category of functors $\mathcal{C}_{Q}\rightarrow\mathcal{C},$
see more details in \cite{holm2019cotorsion}. It is a well known
fact that $\operatorname{Rep}(Q,\mathcal{C})$ inherits many properties
from $\mathcal{C}$. For example, it is an abelian category and, in
case $\mathcal{C}$ is AB3, AB4 or AB5, $\operatorname{Rep}(Q,\mathcal{C})$
will also have these properties \cite[Sect. 2]{holm2019cotorsion}.
Moreover, kernels, images and cokernels are computed vertex-wise in
$\C.$ Therefore, a sequence $X\rightarrow Y\rightarrow Z$ in $\operatorname{Rep}(Q,\mathcal{C})$ is exact
 if, and only if, the sequence
$X_i\rightarrow Y_i\rightarrow Z_i$ is exact in $\C$ $\forall i\in Q_{0}.$
Finally, for a class $\X\subseteq\C,$ we set $\operatorname{Rep}(Q,\mathcal{\X}):=\left\{ F\in\operatorname{Rep}(Q,\mathcal{C})\,|\:F_i\in\X\;\forall i\in Q_{0}\right\} .$

\subsection{Extensions and restrictions of representations}

Let $\mathcal{C}$ be an abelian category and $Q$ be a quiver. For
every subquiver $Q'\subseteq Q$, there is an exact, full and faithful
functor $\iota_{Q'}:\operatorname{Rep}(Q',\mathcal{C})\rightarrow\operatorname{Rep}(Q,\mathcal{C})$ called the {\bf extension} functor
which is defined as follows: for $F\in\operatorname{Rep}(Q',\mathcal{C}),$
set $(\iota_{Q'}(F))_i:=F_i$ $\forall i\in Q'_{0},$ $(\iota_{Q'}(F))_\alpha:=F_\alpha$
$\forall\alpha\in Q'_{1},$ $(\iota_{Q'}(F))_i:=0$ $\forall i\in Q_{0}-Q'_{0}$
and $(\iota_{Q'}(F))_\alpha:=0$ $\forall\alpha\in Q_{1}-Q'_{1}$.
Given a morphism $\gamma:F\rightarrow F'$ in $\operatorname{Rep}(Q',\mathcal{C}),$
we set $(\iota_{Q'}(\gamma))_{i}:=\gamma_{i}$ $\forall i\in Q'_{0},$
and $(\iota_{Q'}(\gamma))_{i}:=0$ $\forall i\in Q_{0}-Q'_{0}$. We
also have the \textbf{restriction} functor $\pi_{Q'}:\operatorname{Rep}(Q,\mathcal{C})\rightarrow\operatorname{Rep}(Q',\mathcal{C})$
which sends $F$ to $F|_{Q'}.$ Note that $\pi_{Q'}$ and $\iota_{Q'}$
are additive exact functors such that $\pi_{Q'}\circ\iota_{Q'}=1_{\operatorname{Rep}(Q',\mathcal{C})}.$
This leads us to the following interesting property. 

\begin{lem}\label{lem: Ext vs pi, iota} Let $Q$ be a quiver, $S$ be a subquiver of $Q,$
and $\mathcal{C}$ be an abelian category. Then, for $n\geq 1$, the
following statements hold true: 
\begin{enumerate}
\item The map $\Ext_{\Rep(S,\C)}^{n}(F,G)\rightarrow\Ext_{\Rep(Q,\C)}^{n}(\iota_{S}F,\iota_{S}G)$,
$\overline{\eta}\mapsto\overline{\iota_{S}(\eta)}$, induces a morphism
of functors $\iota_{S}^{n}:\Ext_{\Rep(S,\C)}^{n}(-,?)\rightarrow\Ext_{\Rep(S,\C)}^{n}(\iota_{S}(-),\iota_{S}(?))$.
\item The map $\Ext_{\Rep(Q,\C)}^{n}(F,G)\rightarrow\Ext_{\Rep(S,\C)}^{n}(\pi_{S}F,\pi_{S}G)$,
$\overline{\eta}\mapsto\overline{\pi_{S}(\eta)}$, induces a morphism
of functors $\pi_{S}^{n}:\Ext_{\Rep(Q,\C)}^{n}(-,?)\rightarrow\Ext_{\Rep(S,\C)}^{n}(\pi_{S}(-),\pi_{S}(?))$.
\item $\pi_{S}^{n}\circ\iota_{S}^{n}=1.$
\end{enumerate}
\end{lem}

Let $S$ be a subquiver of $Q$. In general, we can say nothing more
about the functors $\pi_{S}:\operatorname{Rep}(Q,\mathcal{C})\rightarrow\operatorname{Rep}(S,\mathcal{C})$
and $\iota_{S}:\operatorname{Rep}(S,\mathcal{C})\rightarrow\operatorname{Rep}(Q,\mathcal{C})$.
However, in case $S$ is a full subquiver with $S^{-}=\emptyset$
or $S^{+}=\emptyset$, we will have that these functors form an adjoint
pair. 

\begin{lem}\label{lem:par adjunto i p}\label{adjoint-ip} Let $Q$ be a quiver,
$S$ be a full subquiver of $Q$ with $S^{-}=\emptyset$ (respectively, $S^{+}=\emptyset$)
and let $\mathcal{C}$ be an abelian category. Then, $\iota_{S}:\operatorname{Rep}(S,\mathcal{C})\to\operatorname{Rep}(Q,\mathcal{C})$
is right (respectively, left) adjoint for $\pi_{S}:\operatorname{Rep}(Q,\mathcal{C})\to\operatorname{Rep}(S,\mathcal{C}).$ 
\end{lem}

\begin{proof}
Assume that $S^{-}=\emptyset.$ Let $F\in\operatorname{Rep}(Q,\C)$
and $G\in\operatorname{Rep}(S,\C).$ Since $\pi_{S}\iota_{S}=1_{\operatorname{Rep}(S,\mathcal{C})},$
the map $\pi_{S}:\Hom_{\operatorname{Rep}(Q,\C)}(F,\iota_{S}(G))\to\Hom_{\operatorname{Rep}(S,\C)}(\pi_{S}(F),G)$
is well defined. Now, we define a map $\nu_{S}:\Hom_{\operatorname{Rep}(S,\C)}(\pi_{S}(F),G)\to\Hom_{\operatorname{Rep}(Q,\C)}(F,\iota_{S}(G))$
as follows: for $f=\left\{ f_{s}:F_s\rightarrow G_s\right\} _{s\in S_{0}}\in\Hom_{\operatorname{Rep}(S,\C)}(\pi_{S}(F),G),$
we set \\
$
\nu_{S}(f):=\left\{ \hat{f}_{i}:F_i\rightarrow(\iota_{S}(G))_i\right\} _{i\in Q_{0}},
$
where $\hat{f}_{s}=f_{s}$ $\forall s\in S_{0}$ and $\hat{f}_{i}=0$
$\forall i\in Q_{0}-S_{0}.$ Observe that this family defines a morphism
$\nu_{S}(f):F\rightarrow\iota_{S}\left(G\right)$ in $\operatorname{Rep}(Q,\mathcal{C}).$
Indeed, let $\alpha:i\rightarrow j$ be an arrow in $Q_{1}$. We need
to show that $(\iota_{S}(G))_\alpha\hat{f}_{i}=\hat{f}_{j}F_\alpha$.
In case that $i,j\in S_{0}$, this follows from the fact that $S$
is a full subquiver. Since $S^{-}=\emptyset$, the only remaining
case is when $j\in Q_{0}-S_{0}$. Thus, in this case, we have $(\iota_{S}(G))_\alpha\hat{f}_{i}=0=\hat{f}_{j}F_\alpha.$
Finally, it is straightforward to check the naturality of $\pi_{S}$
and that $\pi_{S}$ is an isomorphism with inverse $\nu_{S}.$ The
case when $S^{+}=\emptyset$ is quite similar and it is left to the
reader. 
\end{proof}

The adjoint pair shown above will be of particular importance to us
because of the following lemma which is a consequence of Lemma \ref{adjoint-ip} and Proposition
\ref{lem:Ext vs adjoint 3}.
 
\begin{lem}\label{lem:Ext subcarcaj} For a quiver $Q,$ a full subquiver $S\subseteq Q,$
$n\geq0$ and an abelian category $\C,$ the following statements
hold true. 
\begin{itemize}
\item[$\mathrm{(a)}$] $\Ext_{\operatorname{Rep}(Q,\C)}^{n}(?,\iota_{S}(-))\cong\Ext_{\operatorname{Rep}(S,\C)}^{n}(\pi_{S}(?),-)$
if $S^{-}=\emptyset.$ 
\item[$\mathrm{(b)}$] $\Ext_{\operatorname{Rep}(Q,\C)}^{n}(\iota_{S}(?),-)\cong\Ext_{\operatorname{Rep}(S,\C)}^{n}(?,\pi_{S}(-))$
if $S^{+}=\emptyset.$ 
\end{itemize}
\end{lem}

\subsection{Left and right adjoints for the evaluation functor}

Let $\mathcal{C}$ be an abelian category and $Q$ be a quiver. For
$i\in Q_{0}$, we recall that there is the evaluation functor $e_{i}:\Rep(Q,\C)\rightarrow\C,\;(F\xrightarrow{\alpha}G)\mapsto(F_i\xrightarrow{\alpha_i}G_i).$
 In this section, by following \cite{enochs2004flat, enochs1999homotopy} and \cite[Sect. 3]{holm2019cotorsion}, we describe the 
left and right  adjoints for $e_{i}.$ 

\begin{defn}\cite{enochs2004flat, enochs1999homotopy, holm2019cotorsion} Let $Q$ be a quiver, $\mathcal{C}$
an abelian category and $\kappa$ an infinite cardinal number. 
\begin{enumerate}
\item If $\kappa\geq\rccn(Q)$ and $\mathcal{C}$ is $AB3(\kappa),$ we
define the following: For every $C\in\mathcal{C}$ and $i\in Q_{0},$
let $f_{i}(C):=C{}^{(Q(i,-))}\in\operatorname{Rep}(Q,\mathcal{C}).$
That is, for every $j\in Q_{0}$, 
\[
(f_{i}(C))_j:=C{}^{(Q(i,j))}:=\coprod_{p\in Q(i,j)}C_{p},
\]
where $C_{p}:=C$ $\forall\,p\in Q(i,j);$ and for every arrow $\alpha:k\rightarrow l$ in $Q$
define
 $(f_{i}(C))_\alpha:(f_{i}(C))_k\to (f_{i}(C))_l$ as the (unique) morphism in $\C$  such that $(f_{i}(C))_\alpha\circ\mu_{\rho}^{C}=\mu_{\alpha\rho}^{C}$ $\forall\rho\in Q(i,k)$,
where $\mu_{\rho}^{C}:C\rightarrow C^{(Q(i,k))}$ and $\mu_{\alpha\rho}^{C}:C\rightarrow C^{(Q(i,l))}$
are the canonical inclusions. Note that the map $C\mapsto f_{i}(C)$
defines a functor $f_{i}:\mathcal{C}\rightarrow\operatorname{Rep}(Q,\mathcal{C})$ 
as follows: Let $C\xrightarrow{h}C'$ in $\C,$ defines, for each $k\in Q_0,$ 
$(f_i(h))_k:(f_i(C))_k\to (f_i(C'))_k$ as the unique morphism in $\C$ such that 
$(f_i(h))_k\circ\mu^{C}_\rho=\mu^{C'}_{\rho}\circ h$ $\forall\rho\in Q(i,k).$
For a class $\mathcal{S}\subseteq\mathcal{C},$ we consider the class
$f_{*}(\mathcal{S}):=\{f_{i}(S)\in\Rep(Q,\C)\,|\:i\in Q_{0}\text{ and }S\in\mathcal{S}\}$.
In case it is necessary to highlight in which quiver we are working,
we will use the notation $f_{i}^{Q}=f_{i}$.  

\item If $\kappa\geq\lccn(Q)$ and $\mathcal{C}$ is $AB3^{*}(\kappa),$
we define the following: For every $C\in\mathcal{C}$ and $i\in Q_{0},$
let $g_{i}(C):=C{}^{Q(-,i)}\in\operatorname{Rep}(Q,\mathcal{C}).$
That is, for every $j\in Q_{0},$ 
\[
(g_{i}(C))_j:=C{}^{Q(j,i)}:=\prod_{p\in Q(j,i)}C_{p},
\]
where $C_{p}:=C$ $\forall\,p\in Q(i,j);$ and for every arrow $\alpha:k\rightarrow l$ in $Q$,
define $(g_{i}(C))_\alpha:(g_i(C))_k\to (g_i(C))_l$ as the (unique) morphism in $\C$
satisfying that $\pi_{\rho}^{C}\circ (g_{i}(C))_\alpha=\pi_{\rho\alpha}^{C}$ $\forall\rho\in Q(l,i)$,
where $\pi_{\rho}^{C}:C^{Q(l,i)}\rightarrow C$ and $\pi_{\rho\alpha}^{C}:C^{Q(k,i)}\rightarrow C$
are the canonical projections. Note that the map $C\mapsto g_{i}(C)$
defines a functor $g_{i}:\mathcal{C}\rightarrow\operatorname{Rep}(Q,\mathcal{C})$.
For a class $\mathcal{S}\subseteq\mathcal{C},$ we consider the class
$g_{*}(\mathcal{S}):=\{g_{i}(S)\in\Rep(Q,\C)\,|\:i\in Q_{0}\text{ and }S\in\mathcal{S}\}$.
In case it is necessary to highlight in which quiver we are working,
we will use the notation $g_{i}^{Q}=g_{i}$. 
\end{enumerate}
\end{defn}

As we will see below, under mild conditions, the functors $f_{i},g_{i}:\mathcal{C}\rightarrow\operatorname{Rep}(Q,\mathcal{C})$
defined above are important because they are, respectively, the left
and right adjoint functors for the evaluation functor $e_{i}:\operatorname{Rep}(Q,\mathcal{C})\rightarrow\mathcal{C}$
 which is an exact functor. Thus \cite[Thm. 3.7]{holm2019cotorsion}
can be strengthened as follows.

\begin{prop}\label{rem:adjuntos de ei} For a quiver $Q,$ an infinite cardinal
number $\kappa$ and an abelian category $\C,$ the following statements
are true. 
\begin{itemize}
\item[(a)] Let $\C$ be AB3($\kappa$) and $\kappa\geq\rccn(Q).$ Then 
\begin{center}
$\Hom_{\operatorname{Rep}(Q,\mathcal{C})}(f_{i}\left(-\right),?)\cong\Hom_{\mathcal{C}}(-,e_{i}(?)),$
for any $i\in Q_{0}.$ 
\par\end{center}

Moreover $f_{*}({}^{\perp_{1}}\X)\subseteq{}^{\perp_{1}}\Rep(Q,\X)$
$\forall\,\X\subseteq\C.$ 
\item[(b)] Let $\C$ be AB3{*}($\kappa$) and $\kappa\geq\lccn(Q).$ Then 
\begin{center}
$\Hom_{\operatorname{Rep}(Q,\mathcal{C})}(-,g_{i}(?))\cong\Hom_{\mathcal{C}}(e_{i}(-),?),$
for any $i\in Q_{0}.$ 
\par\end{center}

Moreover $g_{*}(\X^{\perp_{1}})\subseteq\Rep(Q,\X)^{\perp_{1}}$ $\forall\,\X\subseteq\C.$ 
\end{itemize}
\end{prop}

\begin{proof}
The adjointness in (a) follows as in the proof of \cite[Thm. 3.7]{holm2019cotorsion}.
Finally, for $\X\subseteq\C,$ the inclusion $f_{*}({}^{\perp_{1}}\X)\subseteq{}^{\perp_{1}}\Rep(Q,\X)$
follows from Lemma \ref{lem:Ext vs adjoint 1} (a) since $(f_{i},e_{i})$
is an adjoint pair and $e_{i}$ is an exact functor $\forall\,i\in Q_{0}.$ 
\end{proof}

In the next proposition we can see that, under some mild conditions
on $\mathcal{C}$, there is a similar isomorphism for the extensions
groups. Note that this result is a generalization of \cite[Prop. 5.2]{holm2019cotorsion} and follows from Propositions \ref{lem:Ext vs adjoint 3} and \ref{rem:adjuntos de ei} and the description, see in \cite[Thm 3.7]{holm2019cotorsion}, of the isomorphism 
$$h_{C,F}:\Hom_{\Rep(Q,\C)}(f_i(C),F)\to \Hom_\C(C,e_i(F)),\;t\mapsto t_i\circ\mu_{\epsilon_{i}}^{C},$$ where $\epsilon_{i}\in Q(i,i)$ is the path of length $0$ and $\mu_{\epsilon_{i}}^{C}:C\rightarrow C^{(Q(i,i))}$ is the natural inclusion.

\begin{prop}\label{prop: Ext vs f vs e} For a quiver $Q,$ an infinite cardinal
number $\kappa,$ an abelian category $\C$ and $i\in Q_{0},$ the
following statements are true. 
\begin{itemize}
\item[(a)] Let $\C$ be AB3($\kappa$) and $\kappa\geq\rccn(Q).$ Then $f_{i}:\C\to\Rep(Q,\C)$
is exact if, and only if, $\tau^{n}:\operatorname{Ext}_{\operatorname{Rep}(Q,\mathcal{C})}^{n}(f_{i}(?),-)\xrightarrow{\sim}\operatorname{Ext}_{\mathcal{C}}^{n}(?,e_{i}(-))\;$
$\forall\,n\geq1.$
\item[(b)] Let $\C$ be AB3{*}($\kappa$) and $\kappa\geq\lccn(Q).$ Then $g_{i}:\C\to\Rep(Q,\C)$
is exact if, and only if, $\sigma^{n}:\operatorname{Ext}_{\operatorname{Rep}(Q,\mathcal{C})}^{n}(?,g_{i}(-))\xrightarrow{\sim}\operatorname{Ext}_{\mathcal{C}}^{n}(e_{i}(?),-)\;$
$\forall\,n\geq1.$ 
\end{itemize}
Moreover, for $C\in \C$ and $G\in \Rep(Q,\C),$ we have that $\tau^n_{C,G}(\overline{\eta})=\overline{e_i(\eta)}\cdot \mu^C_{\epsilon_i}$ and 
$\sigma^n_{G,C}(\overline{\delta})=\pi^C_{\epsilon_i}\cdot\overline{e_i(\delta)},$ where $\pi^C_{\epsilon_i}:C^{Q(i,i)}\to C$ is the canonical projection.
\end{prop}

In what follows we discuss enough conditions to have the exactness of $f_i$ and $g_i.$ 

\begin{rem}\label{AB3+Inj=AB4} For an abelian category $\C$ and an infinite cardinal number $\kappa,$ we have the following:
\

(a) By the beginning of the proof of \cite[Prop. 5.2]{holm2019cotorsion},
it can be shown that: If $\C$ is $AB4(\kappa)$ and $\kappa\geq\rccn(Q),$
then $f_{i}:\C\to\Rep(Q,\C)$ is exact $\forall\,i\in Q_{0}.$ Dually,
we also have that: If $\C$ is $AB4^{*}(\kappa)$ and $\kappa\geq\lccn(Q),$
then $g_{i}:\C\to\Rep(Q,\C)$ is exact $\forall\,i\in Q_{0}.$ 
\

(b) By the proof of \cite[Cor. 2.9]{Popescu}, it can be shown that:
If $\C$ has enough injectives and $\C$ is $AB3(\kappa),$ then $\C$ is $AB4(\kappa).$ 
\

(c) Let $i\in Q_{0}.$ If $\rccn_{i}(Q)\leq\aleph_{0}$ then $f_{i}:\C\to\Rep(Q,\C)$
is exact. Dually, If $\lccn_{i}(Q)\leq\aleph_{0}$ then $g_{i}:\C\to\Rep(Q,\C)$
is exact. 
\end{rem}

As we have seen in Remark \ref{AB3+Inj=AB4}, an enough condition
to get that $f_{i}$ (respectively, $g_{i}$) is exact $\forall\,i\in Q_{0}$
is that $\C$ be AB4($\kappa$) with $\kappa\geq\rccn(Q)$ (respectively,
AB4{*}($\kappa$) with $\kappa\geq\lccn(Q)$). In what follows, we
give an example where $g_{i}$ exists $\forall\,i\in Q_{0}.$ However
$\exists\,j\in Q_{0}$ such that $g_{j}$ is not exact but $g_{i}$
is exact $\forall\,i\in Q_{0}-\{j\}$ (see Proposition \ref{prop: Ext vs f vs e}
and \cite[Prop. 3.9]{odabacsi2019completeness}).

\begin{example}
Let $Ab$ be the category of abelian groups. Consider the full subcategory
$\C$ whose objects are the torsion groups. Note that, for any family
$\{X_{i}\}_{i\in I}$ of objects in $\C,$ the coproduct $\coprod_{i\in I}X_{i}$
in $Ab$ is also the coproduct in $\C.$ Moreover, the product in
$\C,$ denoted by $\prod_{i\in I}^{\C}X_{i},$ also exists and it
is the torsion subgroup of the product $\prod_{i\in I}X_{i}$ in $Ab.$
Thus $\C$ is an abelian category which is $AB3$ and $AB3^{*}.$ %
\begin{minipage}[t]{0.7\columnwidth}%
Now, let $Q$ be the quiver consisting of 
\begin{alignat*}{1}
Q_{0} & :=\{x,y\}\cup\{n\in\mathbb{Z}\,|\:n>0\}\\
Q_{1} & :=\{\alpha_{n}:x\rightarrow n\}_{n>0}\cup\{\beta_{n}:n\rightarrow y\}_{n>0}.
\end{alignat*}

Note that $\lccn(Q)=\aleph_{1}$ and thus $g_{i}$ exists for all
$i\in Q_{0}.$ We assert that $g_{y}:\C\to\Rep(Q,\C)$ is not an exact
functor. Indeed, for each $n>0,$ consider the morphism
\begin{center}
$\pi_{n}:\mathbb{Z}_{2^{n}}\rightarrow\mathbb{Z}_{2},\;a+2^{m}\mathbb{Z}\mapsto a+2\mathbb{Z}.$
\par\end{center}%
\end{minipage}\hfill{}%
\fbox{\begin{minipage}[t]{0.25\columnwidth}%
\[ \begin{tikzpicture}[-,>=to,shorten >=1pt,auto,node distance=3cm,main node/.style=,x=1cm,y=.5cm]
\node (Q) at (-.5,0) {$Q:$}; \node (x) at (0,0) {$x$}; \node (y) at (2,0) {$y$}; \node (2) at (1,.5) {$2$}; \node (1) at (1,1.5) {$1$}; \node (d) at (1,-.5) {$\vdots$};
\node (d1) at (1,-2) {$\vdots$}; \node (n) at (1,-1.5) {$n$};
\draw[->, thick]  (x)  to  node [above left] {$\scriptstyle \alpha _1$} (1); \draw[->, thick]  (x)  to  node [below right] {$$} (2); \draw[->, thick]  (x)  to  node [below left] {$\scriptstyle \alpha _n$} (n); \draw[->, thick]  (1)  to  node [above right] {$\scriptstyle \beta _1$} (y); \draw[->, thick]  (2)  to  node [below left] {$$} (y); \draw[->, thick]  (n)  to  node [below right] {$\scriptstyle \beta _n$} (y);
\end{tikzpicture} \]
\end{minipage}}

Thus, we have that 
\[
\pi:=\coprod_{n>0}\pi_{n}:\coprod_{n>0}\mathbb{Z}_{2^{n}}\rightarrow\coprod_{n>0}\mathbb{Z}_{2}
\]
is an epimorphism in $\mathcal{C}.$ Let $h:=(g_{y}(\pi))_x=\prod_{p\in Q(x,y)}^{\C}\pi.$
Then 
\begin{center}
$h:=\prod_{k>0}^{\C}\pi=\prod_{k>0}^{\C}\left(\coprod_{n>0}\pi_{n}\right):\prod_{k>0}^{\C}\left(\coprod_{n>0}\mathbb{Z}_{2^{n}}\right)\rightarrow\prod_{k>0}^{\C}\left(\coprod_{n>0}\mathbb{Z}_{2}\right).$ 
\par\end{center}
Let us show that $h$ is not an epimorphism in $\mathcal{C}.$ Indeed,
firstly we recall that $A:=\prod_{k>0}^{\C}\left(\coprod_{n>0}\mathbb{Z}_{2^{n}}\right)$
is the torsion subgroup of $\prod_{k>0}\left(\coprod_{n>0}\mathbb{Z}_{2^{n}}\right)$,
and $B:=\prod_{k>0}^{\C}\left(\coprod_{n>0}\mathbb{Z}(2)\right)=\prod_{k>0}\left(\coprod_{n>0}\mathbb{Z}(2)\right)$.
Now, consider the object $b:=(b_{k})_{k>0}\in B$ with $b_{k}:=(b{}_{n}^{k})_{n>0}$
defined as $b_{k}^{k}=1+2\mathbb{Z}$ and $b_{n}^{k}=0$ for all $n\neq k$.
If we assume that there is $a:=(a_{k})_{k>0}\in A$, with $a_{k}=(a{}_{n}^{k})_{n>0}$,
such that $h(a)=b$, then $\pi_{k}(a_{k}^{k})=1+2\mathbb{Z}$ for
all $k>0$. Observe that this implies that $a_{k}^{k}$ is an odd
number, and therefore the order of $a_{k}^{k}$ is $2^{k}$ for all
$k>0.$ Hence $a$ is necessarily of infinite order, which is a contradiction
since $A$ is a torsion group. Finally, the discussion above shows
also that $\C$ is not $AB4^{*}(\aleph_{1}).$ Moreover, note that
$\lccn_{y}(Q)=\aleph_{1}.$ On the other hand, for $i\in Q_{0}-\{y\},$
we have that $\lccn_{i}(Q)=2;$ and thus $g_{i}$ is exact for those
vertices $i.$ 
\end{example}

The next lemma will be useful in the following section.

\begin{lem}\label{LemAux} For a quiver $Q,$ an infinite cardinal number $\kappa,$
an abelian category $\C$ and a family $\{X_{i}\}_{i\in Q_{0}}$ of
objects in $\C,$ the following statements hold true. 
\begin{itemize}
\item[(a)] Let $\kappa\geq\max\{\rccn(Q),\ltccn(Q)\}$ and $\C$ be $AB3(\kappa).$
Then, there exists $\coprod_{i\in Q_{0}}f_{i}(X_{i})$ and $\coprod_{\rho\in Q_{1}}f_{t(\rho)}(X_{s(\rho)})$
in $\Rep(Q,\C).$ 
\item[(b)] Let $\kappa\geq\max\{\lccn(Q),\rtccn(Q)\}$ and $\C$ be $AB3^{*}(\kappa).$
Then, there exists $\prod_{i\in Q_{0}}g_{i}(X_{i})$ and $\prod_{\rho\in Q_{1}}g_{s(\rho)}(X_{t(\rho)})$
in $\Rep(Q,\C).$ 
\end{itemize}
\end{lem}

\begin{proof}
We only prove (a). Note that $f_{i}$ exists for each $i\in Q_{0}$
since $\kappa\geq\rccn(Q).$ Now, for each $j\in Q_{0},$ consider
the sets $H_{j}:=\{k\in Q_{0}\,|\:Q(k,j)\neq\emptyset\}$ and $H'_{j}:=\{\rho\in Q_{1}\,|\:Q(t(\rho),j)\neq\emptyset\}$. We assert that $\max\{|H_j|,|H'_j|\}<\kappa.$ Indeed, for $j\in Q_{0}$ and $\rho\in H'_j,$ we fix some $\tilde{\rho}\in Q(t(\rho),j).$ Thus, the map $H'_j\to \bigcup_{i\in Q_{0}}Q(i,j),\;\rho\mapsto \tilde{\rho}\rho,$ is injective and therefore $|H'_j|\leq |\bigcup_{i\in Q_{0}}Q(i,j)|<\ltccn(Q)\leq\kappa.$ Similarly, it can be shown that $|H_j|<\kappa$ and so our assertion follows. In particular, we get that $\coprod_{t\in H_{j}}X_{t}^{(Q(t,j))}$ and $\coprod_{\rho\in H'_{j}}X_{s(\rho)}^{(Q(t(\rho),j))}$
exists in $\C,$ for any $j\in Q_{0}.$

 Consider the representation $\overline{X}\in\operatorname{Rep}(Q,\mathcal{C})$
defined by 
$\overline{X}_j:=\coprod_{t\in H_{j}}X_{t}^{(Q(t,j))},$ for all $j\in Q_{0};$ 
and for each $\alpha:k\to l$ in $Q_{1},$ define $\overline{X}_\alpha:\overline{X}_k\rightarrow\overline{X}_l$
as the (unique) morphism satisfying that $\overline{X}_\alpha\circ\mu_{i}^{k}=\mu_{i}^{l}\circ (f_{i}(X_{i}))_\alpha\;$
$\forall i,$ where $\mu_{i}^{k}:X_{i}^{(Q(i,k))}\rightarrow\coprod_{i\in H_{k}}X_{i}^{(Q(i,k))}=\overline{X}_k$
and $\mu_{i}^{l}:X_{i}^{(Q(i,l))}\rightarrow\coprod_{i\in H_{l}}X_{i}^{(Q(i,l))}=\overline{X}_l$
are the canonical inclusions. We let the reader to show that $\mu_{i}:=\{\mu_{i}^{k}\}_{k\in Q_{0}}:f_{i}(X_{i})\to \overline{X}$
is the coproduct in $\Rep(Q,\C)$ of the family $\{f_{i}(X_{i})\}_{i\in Q_{0}}.$
The construction for $\coprod_{\rho\in Q_{1}}f_{t(\rho)}(X_{s(\rho)})$
is obtained in a similar way.
\end{proof}

\section{The canonical presentation given by a family of functors}

Let $\C$ be an abelian category and $Q$ be a quiver. In this section
we will prove under mild conditions that, for every $F\in\Rep(Q,\C),$ there is an exact sequence of functors
$
\suc[\coprod_{\rho\in Q_{1}}f_{t(\rho)}e_{s(\rho)}(F)][\coprod_{i\in Q_{0}}f_{i}e_{i}(F)][F]
$
called the canonical presentation given by the family $\{f_i\}_{i\in Q_0}$ (the dual result also holds and it is also stated). Once this is
done, we will use this exact sequence to bound the global dimension
of $\Rep(Q,\C)$, to give sufficient conditions for $g_{i}$ to have a right adjoint, and we will give a new proof of some known results. 

\subsection{The canonical presentation}

Let us begin observing that, when exists, the families of functors $\{f_{i}\}_{i\in Q_0}$ and $\{g_{i}\}_{i\in Q_0},$ which were introduced before Proposition \ref{rem:adjuntos de ei},  induce
the functors given below.

\begin{prop}\label{funtor-f}
For a quiver $Q$ and an infinite cardinal number $\kappa$ the following
statements are true. 
\begin{enumerate}
\item Let $\C$ be $AB3(\kappa)$ and $\kappa\geq\rccn(Q).$ Then, the family of functors $\{f_{i}\}_{i\in Q_0}$
induces the functor 
$
f:\mathcal{C}_{Q}^{op}\rightarrow\Fun(\mathcal{C},\Rep(Q,\mathcal{C}))
$
 such that $f(i)(C)=f_i(C):=C{}^{(Q(i,-))}$ for all $C\in\C$ and $i\in Q_0.$
 
\item Let $\C$ be $AB3^{*}(\kappa)$ and $\kappa\geq\lccn(Q).$ Then, the family of functors $\{g_{i}\}_{i\in Q_0}$
induces the functor 
$
g:\mathcal{C}_{Q}^{op}\rightarrow\Fun(\mathcal{C},\Rep(Q,\mathcal{C}))
$
 such that $g(i)(C)=g_{i}(C):=C{}^{Q(-,i)}$ for all $C\in\C$ and $i\in Q_0.$
\end{enumerate}
\end{prop}

\begin{proof}
Let us prove (a). For $i\in Q_{0}$, we have already defined the functor
$f_{i}\in\Fun(\mathcal{C},\Rep(Q,\mathcal{C}))$. Now, for a path
$\rho\in Q(i,j)$, define the natural transformation $f(\rho):=f_{\rho}:f_{j}\rightarrow f_{i}$
as follows. Let $X\in\C$. For each $k\in Q_{0}$, define the morphism $(f_{\rho}(X))_k:(f_j(X))_k\rightarrow (f_i(X))_k$
as the unique one such that 
$(f_{\rho}(X))_k\circ\mu_{\lambda}^{X}=\mu_{\lambda\rho}^{X}$ $\forall\,\lambda\in Q(j,k),$
where $\mu_{\lambda}^{X}:X\rightarrow X^{(Q(j,k))}$ and $\mu_{\lambda\rho}^{X}:X\rightarrow X^{(Q(i,k))}$
are the canonical inclusions. It is straightforward to check that
$f_{\rho}(X):f_{j}(X)\rightarrow f_{i}(X)$ is a morphism of representations
 and that $f_{\rho}:=\{f_{\rho}(X)\}_{X\in\C}$
is a natural transformation.
\end{proof}
%

We also have that the family of evaluation functors $\{e_{i}:\Rep(Q,\mathcal{C})\to \C\}_{i\in Q}$ induces a functor, as can be seen below.

\begin{prop}\label{funtor-e}
Let $Q$ be a quiver and $\C$ be an abelian category. Then, the family of evaluation functors $\{e_{i}\}_{i\in Q_0}$ induces the functor  $e:\mathcal{C}_{Q}\rightarrow\Fun(\Rep(Q,\mathcal{C}),\mathcal{C})$ such that 
 $e(i)(F)=e_{i}(F):=F_i$ for all $F\in\Rep(Q,\mathcal{C})$ and $i\in Q_0.$
\end{prop}
\begin{proof} For each $F\in\Rep(Q,\mathcal{C})$ and $i\in Q_0,$ we set $e(i)(F):=e_{i}(F).$ Consider a path $\rho\in Q(i,j),$ and define $e(\rho):=e_\rho:e_i\to e_j$ as follows. For each $F\in \Rep(Q,\mathcal{C}),$ we can extend $F$ on the path $\rho$ by taking $F_\rho:=F_{\alpha_n}F_{\alpha_{n-1}}\cdots F_{\alpha_1}$ if $\rho=\alpha_n\alpha_{n-1}\cdots \alpha_1$ and $\alpha_i\in Q_1$ $\forall\,i.$ Then, define 
$(e_{\rho})_F:=F_\rho:e_i(F)\to e_j(F).$ We let to the reader to show that $e_\rho:e_i\to e_j$ is a morphism of functors.
\end{proof}

\begin{rem}\label{fefe-nat} Let $\C$ be $AB3(\kappa)$ and $\kappa\geq\rccn(Q)$ be an infinite cardinal. For every $\rho\in Q(j,i)$, by Propositions \ref{funtor-f} and \ref{funtor-e}, we get the following diagram of functors and natural transformations
\[\begin{tikzcd}
	\Rep(Q,\mathcal{C}) &&& \C &&& \Rep(Q,\mathcal{C})
	\arrow[""{name=0, anchor=center, inner sep=0}, "e_i"', curve={height=24pt}, from=1-1, to=1-4]
	\arrow[""{name=1, anchor=center, inner sep=0}, "f_i", curve={height=-24pt}, from=1-4, to=1-7]
	\arrow[""{name=2, anchor=center, inner sep=0}, "f_j"', curve={height=24pt}, from=1-4, to=1-7]
	\arrow[""{name=3, anchor=center, inner sep=0}, "e_j", curve={height=-24pt}, from=1-1, to=1-4]
	\arrow["e_{\rho}", shorten <=6pt, shorten >=6pt, Rightarrow, from=3, to=0]
	\arrow["f_{}\rho", shorten <=6pt, shorten >=6pt, Rightarrow, from=1, to=2]
\end{tikzcd}\]

Thus $f_{i}\cdot e_{\rho}:f_{i}\circ e_{j}\rightarrow f_{i}\circ e_{i}$ and $f_{\rho}\cdot e_{j}:f_{i}\circ e_{j}\rightarrow f_{j}\circ e_{j}$ are natural transformations (use Godement's product). In particular, $(f_{i}\cdot e_{\rho})_F=f_i(F_\rho):f_i(F_j)\to f_i(F_i)$ and $(f_{\rho}\cdot e_{j})_F=f_\rho(F_j):f_i(F_j)\to f_j(F_j),$ for any $F\in\Rep(Q,\C).$
\end{rem}

\begin{lem}\label{LMf-Psi} For a quiver $Q,$ an infinite cardinal number $\kappa\geq\rccn(Q)$ and an abelian $AB3(\kappa)$ category $\C,$ the following statements hold true.
\begin{itemize}
\item[(a)] For each $i\in Q_0,$ let $\psi^{i}:f_{i}\circ e_{i}\rightarrow1_{\Rep(Q,\C)}$ be given by the adjoint pair $(f_{i},e_{i})$ (see Proposition \ref{rem:adjuntos de ei} (a)). Then, for $F\in\Rep(Q,\C)$ 
and $k\in Q_{0},$ we have that $(\psi_{F}^{i})_k:(f_i(F_i))_k\rightarrow F_k$
is the unique morphism such that $(\psi_{F}^{i})_k\circ\mu^{F_i}_{\lambda}=F_\lambda$
for all $\lambda\in Q(i,k),$ were $\mu^{F_i}_{\lambda}:F_i\rightarrow F_i^{(Q(i,k))}$ is the canonical inclusion.
\item[(b)] For any $\rho\in Q(j,i),$ the following diagram commutes
\[\begin{tikzcd}
	{f_i\circ e_j} && {f_i\circ e_i} \\
	\\
	{f_j\circ e_j} && 1_{\Rep(Q,\C)}.
	\arrow["{f_i\cdot e_\rho}", from=1-1, to=1-3]
	\arrow["{f_\rho\cdot e_j}"', from=1-1, to=3-1]
	\arrow["{\psi^i}", from=1-3, to=3-3]
	\arrow["{\psi^j}"', from=3-1, to=3-3]
\end{tikzcd}\]
\end{itemize}
\end{lem}
\begin{proof} (a) Let $i\in Q_0.$ By the proof of \cite[Thm 3.7]{holm2019cotorsion}, we have the isomorphism $h_i:\Hom_{\Rep(Q,\C)}(f_i(C),F)\to \Hom_\C(C,e_i(F)),\;t\mapsto t_i\circ\mu_{\epsilon_{i}}^{C},$ where $\epsilon_{i}\in Q(i,i)$ is the path of length $0$ and $\mu_{\epsilon_{i}}^{C}:C\rightarrow C^{(Q(i,i))}$ is the natural inclusion. Moreover, for $x\in \Hom_\C(C,e_i(F))$ and $j\in Q_0,$ we have that
$(h^{-1}_i(x))_j:C^{(Q(i,j))}\to F_j$ is defined as the unique morphism such that the following diagram commutes
\[\begin{tikzcd}
	C && {F_i} \\
	\\
	{C^{(Q(i,j))}} && {F_j,}
	\arrow["x", from=1-1, to=1-3]
	\arrow["{\mu_{\lambda}^{C}}"', from=1-1, to=3-1]
	\arrow["{(h^{-1}_i(x))_j}"', from=3-1, to=3-3]
	\arrow["{F_\lambda}", from=1-3, to=3-3]
\end{tikzcd}\]

where $\lambda\in Q(i,j)$ and $\mu_{\lambda}^{C}:C\rightarrow C^{(Q(i,j))}$ is the canonical inclusion. Finally, by subsection 2.5, we know that $(\psi^{i}_{F})_j=(h^{-1}_i(1_{F_i}))_j$ and thus (a) follows.
\

(b) It follows from the universal property of coproducts, see the diagram below\\
\begin{minipage}[t]{1\columnwidth}%
\[
\begin{tikzpicture}[-,>=to,shorten >=1pt,auto,node distance=2.5cm,main node/.style=,x=2.5cm,y=2.5cm,framed]

 \begin{scope}[rotate=-45]
\node (1) at (0,0)  {$F_j^{(Q(i,k))}$};
\node (2) at (1,0)  {$F_j^{(Q(j,k))}$};
\node (3) at (0,-1)  {$F_i^{(Q(i,k))}$};
\node (4) at (1,-1)  {$F_k$};
\node (5) at (-.7,-1.7)  {$F_i$};
\node (6) at (-.5,.5)  {$F_j$};
\node (7) at (1.7,.7)  {$F_j$};
   \end{scope}

\draw[->, thin]  (6)  to  node  {$\mu^{F_j}_{\lambda}$} (1);
\draw[->, thin]  (1)  to [above left] node  {$\footnotesize{(f_i(F_\rho))_k}$} (3);
\draw[->, thin]  (1)  to [above right] node  {$\footnotesize{(f_\rho(F_j))_k}$} (2);
\draw[->, thin]  (2)  to [above left] node  {$\footnotesize{\psi^{j}_{F,k}} $} (4);
\draw[->, thin]  (3)  to  node  {$\footnotesize{\psi^{i}_{F,k}}$} (4);
\draw[->, thin]  (6)  to [above left] node  {$F_\rho$} (5);
\draw[->, thin]  (6)  to  node  {$1$} (7);
\draw[->, thin]  (5)  to [below left] node  {$F_\lambda$} (4);
\draw[->, thin]  (7)  to  node  {$F_{\lambda \rho}$} (4);
\draw[->, thin]  (5)  to [above] node  {$\mu^{F_i}_{\lambda}$} (3);
\draw[->, thin]  (7)  to [above] node  {$\mu^{F_j}_{\lambda \rho}$} (2);

\end{tikzpicture}
\]%
\end{minipage}\\
\end{proof}

Our goal now is to show, see Proposition \ref{POEndRep}, that the commutative diagram in Lemma \ref{LMf-Psi} (b) is a push-out in $\End(\Rep(Q,\C)):=\Fun(\Rep(Q,\C),\Rep(Q,\C)).$ That is, we will show that, for any family $\{\beta^{i}:f_{i}e_{i}\rightarrow\chi\}_{i\in Q_{0}}$ in $\End(\Rep(Q,\C))$  such that 
$\beta^{i}\circ(f_{i}\cdot e_{\rho})=\beta^{j}\circ(f_{\rho}\cdot e_{j})$
for all $\rho:j\to i$ in $Q_1,$ 
there exists a unique morphism $\beta':1_{\Rep(Q,\C)}\rightarrow\chi$ in $\End(\Rep(Q,\C))$ such that $\beta'\circ\psi^{i}=\beta^{i}$ for all $i\in Q_{0}.$
In order to prove that, we start with the following Lemma.
 
\begin{lem}\label{Fact3.3}
For a quiver $Q,$ an infinite cardinal $\kappa\geq\rccn(Q)$, an $AB3(\kappa)$ abelian category $\C,$ a family $\{\beta^{i}:f_{i}e_{i}\rightarrow\chi\}_{i\in Q_{0}}$ in 
$\End(\Rep(Q,\C))$ and the isomorphism $h_i:\Hom_{\Rep(Q,\C)}(f_i(?),-)\xrightarrow{\sim}\Hom_\C(?,e_i(-)),$ for each $i\in Q_0,$ the following statements hold true, where $\epsilon_{i}\in Q(i,i)$ is the trivial path and $\mu_{\epsilon_{i}}^{C}:C\to C^{(Q(i,i))}$ is the inclusion map in $\C.$
\begin{itemize}
\item[(a)] For $F\in\Rep(Q,\C)$ and $i\in Q_{0}$, let $\alpha_{F}^{i}:e_{i}\left(F\right)\rightarrow e_{i}\left(\chi(F)\right)$
be the morphism defined by $\alpha_{F}^{i}:=h_{i}(\beta_{F}^{i})$. Then: 
 \begin{itemize}
 \item[(a1)] $\alpha_{F}^{i}=(\beta_{F}^{i})_i\circ\mu_{\epsilon_{i}}^{F_i}$ and  $\alpha_{F}^{i}:F_i\rightarrow(\chi(F))_i$ is the unique morphism 
such that $\psi_{\chi(F)}^{i}\circ f_{i}(\alpha_{F}^{i})=\beta_{F}^{i}.$
 \item[(a2)] The family $\alpha^{i}:=\{\alpha_{F}^{i}\}_{F\in\Rep(Q,\C)}$ defines the natural transformation $\alpha^{i}:e_i\to e_i\circ\chi.$
 \end{itemize}

 \item[(b)] Let $\rho\in Q(j,i)$ be such that $\beta^{i}\circ(f_{i}\cdot e_{\rho})=\beta^{j}\circ(f_{\rho}\cdot e_{j}).$ For $F\in\Rep(Q,\C),$ let
 $\gamma_{F}^{\rho}:e_{j}\left(F\right)\rightarrow e_{i}\left(\chi(F)\right)$
be the morphism defined by\\
 $\gamma_{F}^{\rho}:=h_{i}(\beta_{F}^{j}\circ(f_{\rho}\cdot e_{j})_{F})=h_{i}(\beta_{F}^{i}\circ(f_{i}\cdot e_{\rho})_{F})$.
Then:
  \begin{itemize}
  \item[(b1)] $\gamma_{F}^{\rho}=(\beta_{F}^{i})_i\circ ((f_{i}\cdot e_{\rho})_{F})_i\circ\mu_{\epsilon_{i}}^{F_j}=(\beta_{F}^{j})_i\circ ((f_{\rho}\cdot e_{j})_{F})_i\circ\mu_{\epsilon_{i}}^{F_j}$.
Moreover, $\gamma_{F}^{\rho}:F_j\rightarrow (\chi(F))_i$ is the unique morphism 
such that\\ $\psi_{\chi(F)}^{i}\circ f_{i}(\gamma_{F}^{\rho})=\beta_{F}^{j}\circ(f_{\rho}\cdot e_{j})_{F}=\beta_{F}^{i}\circ(f_{i}\cdot e_{\rho})_{F}$.
 \item[(b2)] The family $\gamma^{\rho}:=\{\gamma_{F}^{\rho}\}_{F\in\Rep(Q,\C)}$ defines the natural transformation $\gamma^{\rho}:e_j\to e_i\circ\chi.$
  \end{itemize}
\end{itemize}
\end{lem}

\begin{proof} Notice that (a1) and (b1) follow from the explicit description of $h_i,$ given in the proof of 
Lemma \ref{LMf-Psi} (a), and Lemma \ref{lem: mitchell adjoint pairs} (b). Finally the naturality in (a2) and (b2) is left to the reader.
\end{proof}

As a consequence of Lemma \ref{Fact3.3}, we get the following result.

\begin{cor}\label{coro3.4}
For a quiver $Q,$ an infinite cardinal $\kappa\geq\rccn(Q)$, an $AB3(\kappa)$ abelian category $\C$ and a family $\{\beta^{i}:f_{i}e_{i}\rightarrow\chi\}_{i\in Q_{0}}$ in 
$\End(\Rep(Q,\C)),$ the following statements hold true.
\begin{enumerate}
\item For $i\in Q_{0}$, there is a unique natural transformation $\alpha^{i}:e_{i}\rightarrow e_{i}\circ\chi$
such that $\left(\psi^{i}\cdot\chi\right)\circ\left(f_{i}\cdot\alpha^{i}\right)=\beta^{i}$.

\item Let $\rho\in Q(j,i)$ be such that $\beta^{i}\circ(f_{i}\cdot e_{\rho})=\beta^{j}\circ(f_{\rho}\cdot e_{j}).$  Then, there is a unique
natural transformation $\gamma^{\rho}:e_{j}\rightarrow e_{i}\circ\chi$
such that\\ $\left(\psi^{i}\cdot\chi\right)\circ\left(f_{i}\cdot\gamma^{\rho}\right)=\beta^{j}\circ(f_{\rho}\cdot e_{j})=\beta^{i}\circ(f_{i}\cdot e_{\rho})$. 
\end{enumerate}
\end{cor}

We are ready to prove the universal property of the family {\small{ $\{\psi^i : f_ie_i\to 1_{\Rep(Q,\C)}\}_{i\in Q_0}.$}}

\begin{thm}\label{POEndRep}
Let $Q$ be a quiver,  $\kappa\geq\rccn(Q)$  be an infinite cardinal, $\C$
be an $AB3(\kappa)$ abelian category and $\{\beta^{i}:f_{i}e_{i}\rightarrow\chi\}_{i\in Q_{0}}$
be a family in $\End(\Rep(Q,\C))$ such that $\beta^{i}\circ(f_{i}\cdot e_{\rho})=\beta^{j}\circ(f_{\rho}\cdot e_{j})$
for all $j\xrightarrow{\rho} i\in Q_{1}$. Then, there is
a unique natural transformation $\alpha:1_{\Rep(Q,\C)}\rightarrow\chi$
such that $\alpha\circ\psi^{i}=\beta^{i}$ for all $i\in Q_{0}$.
\end{thm}

\begin{proof}
Let $\alpha:=\{\alpha_{F}:F\rightarrow\chi(F)\}_{F\in\Rep(Q,\C)}$,
where $(\alpha_{F})_i:=\alpha_{F}^{i}: F_i\rightarrow (\chi(F))_i$ 
for all $i\in Q_{0},$ see Lemma \ref{Fact3.3} (a).  We assert that $\alpha:1_{\Rep(Q,\C)}\rightarrow\chi$ 
is a natural transformation. In order to do that, we start by proving that $\alpha_{F}:F\rightarrow\chi(F)$ is a morphism in $\Rep(Q,\C).$ That is, we will show that 
$(\chi(F))_\rho\circ\alpha_{F}^{j}=\alpha_{F}^{i}\circ F_\rho;$
for all $\overset{\rho}{j\rightarrow i}\in Q_{1}.$ Consider $\overset{\rho}{j\rightarrow i}\in Q_{1}.$ Since $\beta^{i}\circ(f_{i}\cdot e_{\rho})=\beta^{j}\circ(f_{\rho}\cdot e_{j}),$ we get the following commutative diagram
\[\begin{tikzcd}
	{f_i(F_j)} && {f_i(F_i)} \\
	\\
	{f_j(F_j)} && {\chi(F)}
	\arrow["{f_i(F_\rho)}", from=1-1, to=1-3]
	\arrow["{f_\rho(F_j)}"', from=1-1, to=3-1]
	\arrow["{\beta^j_F}"', from=3-1, to=3-3]
	\arrow["{\beta^i_F}", from=1-3, to=3-3]
\end{tikzcd}\]
In particular, from $\beta^j_F:f_j(F_j)\to \chi(F),$ we get 
$(\chi(F))_\rho\circ(\beta^j_F)_j=(\beta^j_F)_i\circ (f_j(F_j))_\rho.$ Moreover, we also have the following commutative diagram
\[\begin{tikzcd}
	{F_i} && {(f_i(F_i))} && {(\chi(F))_i} \\
	\\
	{F_j} && {(f_i(F_j))} && {(f_j(F_i))_i}
	\arrow["{\alpha^i_F}", curve={height=-30pt}, from=1-1, to=1-5]
	\arrow["{\mu^{F_i}_{\epsilon_i}}", from=1-1, to=1-3]
	\arrow["{(\beta^i_F)_i}", from=1-3, to=1-5]
	\arrow["{F_\rho}", from=3-1, to=1-1]
	\arrow["{(\beta^j_F)_i}"', from=3-5, to=1-5]
	\arrow["{\mu^{F_j}_\rho}"', curve={height=30pt}, from=3-1, to=3-5]
	\arrow["{\mu^{F_j}_{\epsilon_i}}"', from=3-1, to=3-3]
	\arrow["{(f_\rho(F_j))_i}"', from=3-3, to=3-5]
	\arrow["{(f_i(F_\rho))_i}"', from=3-3, to=1-3]
\end{tikzcd}\]
and thus
\begin{alignat*}{1}
(\chi(F))_\rho\circ\alpha_{F}^{j} & =(\chi(F))_\rho\circ(\beta_{F}^{j})_j
\circ\mu_{\epsilon_{j}}^{F_j}\\
 & =(\beta_{F}^{j})_i\circ (f_{j}(F_j))_\rho\circ\mu_{\epsilon_{j}}^{F_j}\\
 & =(\beta_{F}^{j})_i\circ\mu_{\rho}^{F_j}\\
 & =(\beta_{F}^{j})_i\circ(f_\rho(F_j))_i\circ\mu_{\epsilon_{i}}^{F_j}\\
 & =(\beta_{F}^{j})_i\circ\mu_{\epsilon_{i}}^{F_i}\circ F_\rho\\
 & =\alpha^i_F\circ F_\rho;
\end{alignat*}
proving that $\alpha_{F}:F\rightarrow\chi(F)$ is a morphism in $\Rep(Q,\C).$ Finally, to show that $\alpha:1_{\Rep(Q,\C)}\rightarrow\chi$ 
is natural, consider $\theta:F\to G$ in $\Rep(Q,\C)$ and let us prove that $\chi(\theta)\circ\alpha_F=\alpha_G\circ\theta.$ However, this equality follows from Lemma \ref{Fact3.3} (a2).
\

Now, by applying the natural morphism $\psi^{i}:f_i\circ e_i\to 1_{Rep(Q,\C)}$ to 
each morphism $\alpha_F:F\to \chi(F),$ for all $F\in\Rep(Q,\C),$ we get that  
$\alpha_{F}\circ\psi_{F}^{i}=\psi_{\chi(F)}^{i}\circ f_{i}(\alpha^i_{F})=\beta_{F}^{i}$ (see Lemma \ref{Fact3.3} (a1)). Therefore $\alpha\circ\psi^{i}=\beta^{i}$
for all $i\in Q_{0}$. Similarly, if $\alpha':1_{\Rep(Q,\C)}\rightarrow\chi$
is a natural transformation such that $\alpha'\circ\psi^{i}=\beta^{i}$
for all $i\in Q_{0}$, we have that 
$
\psi_{\chi(F)}^{i}\circ f_{i}(\alpha'_{F})_i=\alpha'_{F}\circ\psi_{F}^{i}=\beta_{F}^{i}\quad\forall i\in Q_{0},\forall F\in\Rep(Q,\C)\text{.}
$
Thus, by the unicity property in Lemma \ref{Fact3.3} (a1), we get that 
$\alpha'_{F}=\alpha_{F},$ for all $F\in\Rep(Q,\C).$
Therefore, $\alpha:1_{\Rep(Q,\C)}\to \chi$ is the unique natural transformation
such that $\alpha\circ\psi^{i}=\beta^{i}$ for all $i\in Q_{0}$.
\end{proof}
\begin{defn}
Let $Q$ be a quiver, $\kappa\geq\rccn(Q)$ be an infinite cardinal and
$\C$ be an $AB3(\kappa)$ abelian category. Consider the category
$S_{Q}$ with $\operatorname{Ob}(S_{Q}):=Q_{0}\cup Q_{1}$ and whose
only non-identity morphisms are, for every $\rho\in Q_{1}$, the morphisms
$\rho\rightarrow s(\rho)$ and $\rho\rightarrow t(\rho)$. 
\begin{enumerate}
\item Define a functor $\mathfrak{F}:S_{Q}\rightarrow\End(\Rep(Q,\C))$
as  $\mathfrak{F}(\rho):=f_{i}\circ e_{j}$ for all $\overset{\rho}{j\rightarrow i}\in Q_{1}$,
$\mathfrak{F}(i)=f_{i}\circ e_{i}$ for all $i\in Q_{0}$ and, for $\overset{\rho}{j\rightarrow i}\in Q_{1}$
define $\mathfrak{F}(\rho\rightarrow i):=f_{i}\cdot e_{\rho}$ and $\mathfrak{F}(\rho\rightarrow j)=f_{\rho}\cdot e_{j}$. 
\item Similarly, for $F\in\Rep(Q,\C)$, define a functor $\mathfrak{F}(F):S_{Q}\rightarrow\Rep(Q,\C)$
as follows: $\mathfrak{F}(F)(\rho):=f_{i}\circ e_{j}(F)=f_i(F_j)$ for all 
$\overset{\rho}{j\rightarrow i}\in Q_{1}$,
$\mathfrak{F}(F)(i):=f_{i}\circ e_{i}(F)=f_i(F_i)$ for all $i\in Q_{0}$ and, for
$\overset{\rho}{j\rightarrow i}\in Q_{1},$ define $\mathfrak{F}(F)(\rho\rightarrow i):=\left(f_{i}\cdot e_{\rho}\right)_{F}=f_i(F_\rho)$
and $\mathfrak{F}(F)(\rho\rightarrow j):=\left(f_{\rho}\cdot e_{j}\right)_{F}=f_\rho(F_j).$ 
\end{enumerate}
\end{defn}

By using the definition above, we get the following new version of Theorem \ref{POEndRep}  in terms of the functor  $\mathfrak{F}:S_{Q}\rightarrow\End(\Rep(Q,\C)).$
 
\begin{thm}\label{thm:colim fe} Let $Q$ be a quiver, $\kappa\geq\rccn(Q)$ be an infinite cardinal and
$\C$ be an $AB3(\kappa)$ abelian category. Then, $\colim(\mathfrak{F})=1_{\Rep(Q,\C)}$.
In particular, $\colim(\mathfrak{F}(F))=F,$ for all $F\in\Rep(Q,\C)$.
\end{thm}

As a consequence of the above theorem, we get the following result
which is a generalization of \cite[Lem. 3.5]{bauer2020cotorsion},
where the statement was proved for finite acyclic quivers. The exact
sequence appearing in the statement will be called the canonical presentation
given by the family of functors $\{f_i\}_{i\in Q_0}.$

\begin{cor}\label{cor:presentacion canonica} Let $Q$ be a quiver, $\kappa\geq\max\{\rccn(Q),\ltccn(Q)\}$ be an infinite cardinal and $\C$ be an $AB3(\kappa)$
abelian category. Then, there is an exact sequence 
\[
\eta:\;\suc[\coprod_{\rho\in Q_{1}}f_{t(\rho)}e_{s(\rho)}][\coprod_{i\in Q_{0}}f_{i}e_{i}][1_{\Rep(Q,\C)}][\Gamma]
\]
in $\End(\Rep(Q,\C))$. Moreover, for each $F\in\Rep(Q,\C)$,
the  exact sequence 
$
\eta_{F}:\;\suc[\coprod_{\rho\in Q_{1}}f_{t(\rho)}e_{s(\rho)}(F)][\coprod_{i\in Q_{0}}f_{i}e_{i}(F)][F][\Gamma_{F}]
$ in $\Rep(Q,\C)$
 splits vertex-wise in $\C.$
\end{cor}

\begin{proof} Let $F\in\Rep(Q,\C).$ Then, by Lemma \ref{LemAux}, we get that $\coprod_{\rho\in Q_{1}}f_{t(\rho)}e_{s(\rho)}(F)$
and $\coprod_{i\in Q_{0}}f_{i}e_{i}(F)$ exist in $\Rep(Q,\C).$ Therefore, $\coprod_{\rho\in Q_{1}}f_{t(\rho)}e_{s(\rho)}$ and $\coprod_{i\in Q_{0}}f_{i}e_{i}(F)$ exist in $\End(\Rep(Q,\C)),$ and let $\mu_{\rho}:f_{t(\rho)}e_{s(\rho)}\rightarrow\coprod_{\rho\in Q_{1}}f_{t(\rho)}e_{s(\rho)}$
and $\mu_{i}:f_{i}e_{i}\rightarrow\coprod_{i\in Q_{0}}f_{i}e_{i}$ be the natural inclusions,  for all $\rho\in Q_{1}$ and $i\in Q_{0}.$
Now, for each arrow $\rho\in Q_1,$ consider the following diagram in  $\End(\Rep(Q,\C))$
\[\begin{tikzcd}
	{f_{t(\rho)}\circ e_{s(\rho)}} && {f_{t(\rho)}\circ e_{t(\rho)}} \\
	\\
	{f_{s(\rho)}\circ e_{s(\rho)}} && {\coprod_{i\in Q_0} f_i\circ e_i}
	\arrow["{f_{t(\rho)}\cdot e_{\rho}}", from=1-1, to=1-3]
	\arrow["{f_{\rho}\cdot e_{s(\rho)}}"', from=1-1, to=3-1]
	\arrow["{\mu_{s(\rho)}}"', from=3-1, to=3-3]
	\arrow["{\mu_{t(\rho)}}", from=1-3, to=3-3]
\end{tikzcd}\]
Therefore, by the universal property of the coproduct there are morphisms\\
 $\gamma_1,\gamma_2:\coprod_{\rho\in Q_{1}}f_{t(\rho)}\circ e_{s(\rho)}\to \coprod_{i\in Q_0} f_i\circ e_i$ such that 
 $\gamma_1\circ \mu_\rho=\mu_{t(\rho)}\circ(f_{t(\rho)}\cdot e_{\rho})$ and $\gamma_2\circ \mu_\rho=\mu_{s(\rho)}\circ(f_{\rho}\cdot e_{s(\rho)}).$ Thus $\Gamma:=\gamma_2-\gamma_1:\coprod_{\rho\in Q_{1}}f_{t(\rho)}\circ e_{s(\rho)}\to \coprod_{i\in Q_0} f_i\circ e_i$ is the unique morphism such that 
 \begin{center}
$
\Gamma\circ\mu_{\rho}=\mu_{s(\rho)}\circ(f_{\rho}\cdot e_{s(\rho)})-\mu_{t(\rho)}\circ(f_{t(\rho)}\cdot e_{\rho}).
$ 
\end{center}
Moreover, by Theorem \ref{thm:colim fe}, we get that $\Cok(\Gamma)=\colim(\mathfrak{F})=1_{\Rep(Q,\C)}.$ It  remains to show that 
$\Gamma:\coprod_{\rho\in Q_{1}}f_{t(\rho)}\circ e_{s(\rho)}\to \coprod_{i\in Q_0} f_i\circ e_i$ is a monomorphism. For this, we
will prove that, for $F\in\Rep(Q,\C),$  the morphism 
\begin{center}
$\Gamma_F:\coprod_{\rho\in Q_{1}}f_{t(\rho)}(F_{s(\rho)})\to \coprod_{i\in Q_0} f_i(F_i)$ 
\end{center}
satisfies that $(\Gamma_F)_k$ is a split monomorphism in $\C,$ for all $k\in Q_{0}$.
Let $k\in Q_{0}$. Consider the set  $M_{1}:=\{(\alpha,\alpha_{0})\in Q(i,k)\times Q_{1}(j,i)\}_{j,i\in Q_{0}}.$ Notice that 
$\coprod_{\rho\in Q_{1}}(f_{t(\rho)}(F_{s(\rho)}))_k=\coprod_{(\alpha,\alpha_{0})\in M_{1}}F_{s(\alpha\alpha_{0})}=:A.$ On the other hand, we also have that   
$\coprod_{i\in Q_{0}}(f_{i}(F_i))_k=\coprod_{\alpha\in Q(-,k)}F_{s(\alpha)}=:B. $ Consider the natural inclusions in the coproduct
\begin{gather*}
\mu_{(\alpha,\alpha_0)}:F_{s(\alpha_0)}\to  \coprod_{(\alpha,\alpha_{0})\in M_{1}}F_{s(\alpha\alpha_{0})}=A,\\
\mu_{(\gamma, s(\gamma))}:F_{s(\gamma)}\to  \coprod_{\gamma\in Q(-,k)}F_{s(\gamma)}=B.
\end{gather*}

We assert that $(\Gamma_F)_k :A\to B$ is the unique morphism in $\C$ such that 
\begin{center}
$(\Gamma_F)_k\circ\mu_{(\lambda,\rho)}=\mu_{(\lambda\rho,s(\rho))}-\mu_{(\lambda,s(\lambda))}\circ F_\rho,$ for all $\rho\in Q_1$ and $\lambda\in Q(t(\rho),k).$
\end{center}
Indeed, we know that 
$\Gamma\circ\mu_{\rho}=\mu_{t(\rho)}\circ(f_{t(\rho)}\cdot e_{\rho})-\mu_{s(\rho)}\circ(f_{\rho}\cdot e_{s(\rho)}).$ Therefore, 
\begin{center}
$(\Gamma_F)_k\circ((\mu_{\rho})_F)_k=((\mu_{t(\rho)})_F)_k\circ(f_{t(\rho)}(F_\rho))_k-((\mu_{s(\rho)})_F)_k\circ(f_{\rho}(F_{s(\rho)}))_k.$
\end{center}
After doing some calculations we get that:
\begin{gather*}
(\Gamma_F)_k\circ((\mu_{\rho})_F)_k\circ \mu_{\lambda}^{F_{s(\rho)}}=(\Gamma_F)_k\circ\mu_{(\lambda,\rho)},\\
((\mu_{t(\rho)})_F)_k\circ(f_{t(\rho)}(F_\rho))_k\circ \mu_{\lambda}^{F_{s(\rho)}}=((\mu_{t(\rho)})_F)_k\circ \mu_{\lambda}^{F_{t(\rho)}}\circ F_\rho=\mu_{(\lambda,t(\rho))}\circ F_\rho,\\
((\mu_{s(\rho)})_F)_k\circ(f_{\rho}(F_{s(\rho)}))_k\circ \mu_{\lambda}^{F_{s(\rho)}}= ((\mu_{s(\rho)})_F)_k\circ \mu_{\lambda\rho}^{F_{s(\rho)}}=\mu_{(\lambda\rho,s(\rho))},
\end{gather*}
and thus our assertion follows.
\

Notice that if $|Q(-,k)|=1,$ we get that $A=0$ and thus  $(\Gamma_F)_k :A\to B$ is an split-mono. Thus, we can assume that 
$|Q(-,k)|>1.$ Now, by using the universal property of the coproduct $\mu_{(\gamma, s(\gamma))}:F_{s(\gamma)}\to  \coprod_{\gamma\in Q(-,k)}F_{s(\gamma)}=B,$  we define $\Lambda_k:B\rightarrow A$ to 
be the unique morphism such that $\Lambda_k\circ\mu_{(\gamma,s(\gamma))}=\Lambda_{k,\gamma},$ where $\Lambda_{k,\gamma}:F_{s(\gamma)}\to A$ is defined, for each $\gamma\in Q(-,k),$ as follows: for the trivial path $\epsilon_k$ at the vertex $k,$ we set $\Lambda_{k,\epsilon_k}:=0,$ for an arrow $\alpha$ with $t(\alpha)=k,$ we set $\Lambda_{k,\alpha}:=\mu_{(\epsilon_k,\alpha)}.$ 
Let now $\gamma=\alpha_n\alpha_{n-1}\cdots\alpha_1\in Q(-,k)$ be a path of length $n\geq 2.$ Consider $\alpha_{[k,l]}:=\alpha_{k}\cdots\alpha_{l},$ for $n\geq k\geq l\geq1.$ Then, for such $\gamma=\alpha_n\alpha_{n-1}\cdots\alpha_1,$ we set
\[\Lambda_{k,\gamma}:=\mu_{(\epsilon_{t(\alpha_n)},\alpha_n)}\circ F_{\alpha_{[n-1,1]}}+\sum_{m=2}^{n-1}\mu_{(\alpha_{[n,m+1]}, \alpha_m})\circ F_{\alpha_{[m-1,1]}}+\mu_{(\alpha_{[n,2]},\alpha_1)}.\]
We assert that $\Lambda_k\circ (\Gamma_F)_k\circ \mu_{(\lambda,\rho)}=\mu_{(\lambda,\rho)},$ for each $\rho\in Q_1$ and 
$\lambda\in Q(t(\rho),k).$ Indeed, notice firstly that 
\begin{center}
$\Lambda_k\circ (\Gamma_F)_k\circ \mu_{(\lambda,\rho)}=\Lambda_{k,\lambda\rho}-\Lambda_{k,\lambda}\circ F_\rho.$
\end{center}
Thus, for $\lambda=\epsilon_k,$ we have $\Lambda_{k,\lambda\rho}=\Lambda_{k,\rho}=\mu_{(\epsilon_k,\rho)}$ and $\Lambda_{k,\epsilon_k}\circ F_\rho=0.$ Now, if $\lambda\rho=\alpha_2\alpha_1$ is a path of length 2, we have that 
$\Lambda_{k,\lambda\rho}=\mu_{(\epsilon_{t(\alpha_2)},\alpha_2)}\circ F_{\alpha_1}+\mu_{(\alpha_2,\alpha_1)}$ and 
$\Lambda_{k,\lambda}\circ F_\rho=\mu_{(\epsilon_k,\alpha_2)}\circ F_{\alpha_1}$ and thus $\Lambda_{k,\lambda\rho}-\Lambda_{k,\lambda}\circ F_\rho=\mu_{(\alpha_2,\alpha_1)}=\mu_{(\lambda,\rho)}.$ Finally, let $\lambda\rho=\alpha_n\alpha_{n-1}\alpha_2\alpha_1$ be a path of length $n\geq 3.$ Then 
$ \Lambda_{k,\lambda}=\Lambda_{k,\alpha_{[n,2]}}=\mu_{(\epsilon_{t(\alpha_n)},\alpha_n)}\circ F_{\alpha_{[n-1,2]}}+\sum_{m=3}^{n-1}\mu_{(\alpha_{[n,m+1]}, \alpha_m})\circ F_{\alpha_{[m-1,2]}}+\mu_{(\alpha_{[n,3]},\alpha_2)}$ and therefore
\begin{center}
$\Lambda_{k,\lambda}\circ F_\rho=\Lambda_{k,\lambda}\circ F_{\alpha_1}=\mu_{(\epsilon_{t(\alpha_n)},\alpha_n)}\circ F_{\alpha_{[n-1,1]}}+\sum_{m=2}^{n-1}\mu_{(\alpha_{[n,m+1]}, \alpha_m})\circ F_{\alpha_{[m-1,1]}}.$
\end{center}
Then, $\Lambda_{k,\lambda\rho}-\Lambda_{k,\lambda}\circ F_\rho=\mu_{(\alpha_{[n,2]},\alpha_1)}=\mu_{(\lambda,\rho)},$ proving our assertion, and thus we conclude that $\Lambda_k\circ (\Gamma_F)_k=1_A.$
\end{proof}

In some applications, the dual of the previous corollary will be useful (see \cite{AM22}). For the convenience  of the reader, we state this result below about the existence of the canonical co-presentation
given by the family of functors $\{g_i\}_{i\in Q_0}.$ 

\begin{cor}\label{cor:copresentacion canonica} Let $Q$ be a quiver, $\kappa\geq\max\{\rtccn(Q),\lccn(Q)\}$ be an infinite cardinal and $\C$ be an $AB3^*(\kappa)$
abelian category. Then, there is an exact sequence 
\[
\eta:\;\suc[1_{\Rep(Q,\C)}][\prod_{i\in Q_{0}}g_{i}e_{i}][\prod_{\rho\in Q_{1}}g_{s(\rho)}e_{t(\rho)}][][\Gamma]
\]
in $\End(\Rep(Q,\C))$. Moreover, for each $F\in\Rep(Q,\C)$,
the  exact sequence 
$
\eta_F:\;\suc[F][\prod_{i\in Q_{0}}g_{i}e_{i}(F)][\prod_{\rho\in Q_{1}}g_{s(\rho)}e_{t(\rho)}(F)][][\Gamma_F]$
 in $\Rep(Q,\C)$
 splits vertex-wise in $\C.$
\end{cor}

\subsection{A right adjoint for $g_{i}$ and a left one for $f_i$}

We start with the following two Lemmas. 

\begin{lem}\cite[Lem. 3.7]{bauer2020cotorsion}\label{lem:left adjointness closed under coker}
Let $F\hookrightarrow G\twoheadrightarrow H$ be an exact sequence
of functors between abelian categories $\mathcal{A}$ and $\mathcal{B}$.
If both $F,G:\mathcal{A}\rightarrow\mathcal{B}$ admit right adjoints,
then so does $H$. 
\end{lem}

\begin{lem}\label{lem: lemita g}Let $Q$ be a quiver, $y,z\in Q_{0}$ such that
$|Q(y,z)|<\infty$, $\kappa\geq\rccn(Q)$ an infinite cardinal and 
 $\mathcal{C}$ be an $AB3(\kappa)$ abelian category. Then, for
$x\in Q_{0}$, we have that $f_{x}\circ e_{y}\circ g_{z}=\coprod_{\rho\in Q(y,z)}f_{x}$. 
\end{lem}

\begin{proof} Let $C\in\C.$ Since $Q(y,z)$ is finite, we have that 
\begin{center}
$f_{x}\circ e_{y}\circ g_{z}(C)=\left(C^{Q(z,y)}\right)^{(Q(x,-))}=\left(C^{(Q(x,-))}\right)^{(Q(z,y))}=\coprod_{\rho\in Q(y,z)}f_{x}(C)$
\end{center}
and thus the lemma follows.
\end{proof}

Now, by considering the canonical presentation and Lemma \ref{lem:left adjointness closed under coker},
we get the following result. It should be mentioned that this is a
generalization of \cite[Lem. 3.8]{bauer2020cotorsion}, where it
is stated for finite and acyclic quivers. 

\begin{prop}\label{prop:exactitud g} Let $Q$ be a quiver, $z\in Q_{0},$  $\kappa\geq\max\{\rccn(Q),\ltccn(Q)\}$ be an infinite cardinal 
and $\mathcal{C}$ be an $AB3(\kappa)$ abelian category. Consider the following conditions:
\begin{itemize}
\item[(a)] the set $Q(-,z)$ is finite;
\item[(b)] $\mathcal{C}$ is $AB3^{*}(\kappa)$ and $\lccn_z(Q)\leq\aleph_0.$
\end{itemize}
If one of the above conditions holds true, then  the functor
$g_{z}:\C\to \Rep(Q,\C)$ exists and 
admits a right adjoint. 
\end{prop}

\begin{proof} Assume that $|Q(-,z)|<\infty$ or $\lccn_z(Q)\leq\aleph_0.$ Notice that $g_z$ exists and then induces the exact functor
\begin{center}
 $\left(g_{z}\right)^{*}:\End(\Rep(Q,\C))\rightarrow\Fun(\C,\Rep(Q,\C)),\;(H_1\xrightarrow{b}H_2)\mapsto(H_1\circ g_z\xrightarrow{b\cdot g_z}H_2\circ g_z).$
\end{center}
Thus, by Corollary \ref{cor:presentacion canonica}
and Lemma \ref{lem: lemita g}, we get an exact sequence 
\[
\suc[\coprod_{\rho\in Q_{1}}\coprod_{\rho'\in Q(s(\rho),z)}f_{t(\rho)}][\coprod_{i\in Q_{0}}\coprod_{\rho'\in Q(i,z)}f_{i}][g_{z}][\Gamma\cdot g_{z}].
\]

Let $|Q(-,z)|<\infty.$ Then, the coproducts above defined are finite and thus by Proposition \ref{rem:adjuntos de ei} (a) and Lemmas \ref{lem:left adjointness closed under coker}
and \ref{lem: coproducto pares adjuntos}, it follows that $g_{z}$ admits a right
adjoint. 
\

Let $\mathcal{C}$ be $AB3^{*}(\kappa)$ and $\lccn_z(Q)\leq\aleph_0.$ Consider  
$H:=\{t\in Q_0\;:\;Q(t,z)\neq \emptyset\}$ and $\tilde{H}:=\bigcup_{t\in H} Q(t,z).$ Observe that $\coprod_{i\in Q_{0}}\coprod_{\rho'\in Q(i,z)}f_{i}=\coprod_{\sigma\in\tilde{H}}f_{s(\sigma)}.$ Moreover, 
$|\tilde{H}|=\sum_{t\in H}|Q(t,z)|\leq |H|\,\lccn_z(Q)<\kappa $
since $|H|<\kappa$ (see the proof of Lemma \ref{LemAux}) and $\lccn_z(Q)\leq\aleph_0.$
\

Now, consider the sets $H':=\{\rho\in Q_1\;:\;Q(t(\rho),z)\neq \emptyset\}$ and $\tilde{H'}:=\bigcup_{\lambda\in H'}Q(t(\lambda),z).$ Then $\coprod_{\rho\in Q_{1}}\coprod_{\rho'\in Q(s(\rho),z)}f_{t(\rho)}=\coprod_{\sigma\in \tilde{H'}} f_{s(\sigma)}$
and $|\tilde{H'}|=\sum_{\sigma\in H'}|Q(s(\sigma),z)|\leq |H'|\,\lccn_z(Q)<\kappa $
since $|H'|<\kappa$ (see the proof of Lemma \ref{LemAux}) and $\lccn_z(Q)\leq\aleph_0.$ Therefore by Proposition \ref{rem:adjuntos de ei} (a) and Lemmas \ref{lem:left adjointness closed under coker}
and \ref{lem: coproducto pares adjuntos}, it follows that $g_{z}$ admits a right
adjoint. 
\end{proof}
We have the next result by duality. 

\begin{prop}\label{prop:exactitud f}Let $Q$ be a quiver, $z\in Q_{0},$ $\kappa\geq\max\{\lccn(Q),\rtccn(Q)\}$ be an infinite cardinal 
and $\mathcal{C}$ be an $AB3^{*}(\kappa)$ abelian category. Consider the following conditions:
\begin{itemize}
\item[(a)] the set $Q(z,-)$ is finite;
\item[(b)] $\mathcal{C}$ is $AB3(\kappa)$ and $\rccn_z(Q)\leq\aleph_0.$
\end{itemize}
If one of the above conditions holds true, then  the functor 
$f_{z}:\C\to \Rep(Q,\C)$ exists and admits a left adjoint. 
\end{prop}

We end this section with the following corollaries that are very useful, for example, in tilting theory (see \cite{AM22}). In order to prove them,  we use the 
next well-known fact whose proof is given for the sake of completeness.

\begin{lem}\label{adj-prec-preen} Let $\C$ and $\mathcal{D}$ be abelian categories, and $(S:\mathcal{D}\rightarrow\C,T:\C\rightarrow\mathcal{D})$
be an adjoint pair between $\C$ and $\mathcal{D}$. Then, the following
statements hold true. 
\begin{enumerate}
\item If $\X$ is a precovering class in $\mathcal{D}$, then $S(\X)$ is
precovering in $\C$.
\item If $\mathcal{Y}$ is a preenveloping class in $\mathcal{C}$, then
$T(\Y)$ is preenveloping in $\mathcal{D}$.
\end{enumerate}
\end{lem}

\begin{proof}
Let us prove (a). Consider $C\in\C$, and let $x:X\rightarrow T(C)$
be an $\X$-precover in $\mathcal{D}$. We claim that $\psi_{C}\circ S(x)$
is a $S(\X)$-precover. Indeed, given a $x':X'\rightarrow C$ with
$X'\in S(\X)$, say $X'=S(X_{0})$, consider the morphism $T(x')\circ\varphi_{X_{0}}$.
Since $x$ is an $\X$-precover, there is a morphism $z:X_{0}\rightarrow X$
such that $x\circ z=T(x')\circ\varphi_{X_{0}}$. Hence, $\psi_{C}\circ S(x)\circ S(z)=\psi_{C}\circ ST(x')\circ S(\varphi_{X_{0}})=x'$
(see Lemma \ref{lem: mitchell adjoint pairs}). Therefore, $\psi_{C}\circ S(x)$
is a $S(\X)$-precover.
\end{proof}

\begin{cor}\label{cor: g(X) es precovering II} Let $Q$ be a quiver, $\kappa\geq\max\{\rccn(Q),\ltccn(Q)\}$ be an infinite cardinal, $Z\subseteq Q_{0}$
be such that $|Z|<\kappa$ and $\mathcal{C}$
be an $AB3(\kappa)$ abelian category. 
Consider the following conditions:
\begin{itemize}
\item[(a)] the set $Q(-,z)$ is finite, for any $z\in Z;$
\item[(b)] $\mathcal{C}$ is $AB3^{*}(\kappa)$ and $\lccn_z(Q)\leq\aleph_0,$ for any $z\in Z.$
\end{itemize}
If one of the above conditions holds true and $\mathcal{X}$ is a precovering
class in $\C$, then for each $F\in\Rep(Q,\C)$ there exists  
$p:\coprod_{z\in Z}g_{z}(X_{z})\rightarrow F$ in $\Rep(Q,\C)$ which is an $\Add_{\leq |Z|}\left(\bigcup_{z\in Z}g_{z}(\X)\right)$-precover. Moreover, each $p_z:g_z(X_z)\to F$ is a $g_{z}(\mathcal{X})$-precover.
\end{cor}

\begin{proof} Let $\mathbb{X}:=\Add_{\leq|Z|}\left(\bigcup_{z\in Z}g_{z}(\X)\right)$.
Consider $F\in\Rep(Q,\C)$ and $z\in Z.$ Then, by Proposition \ref{prop:exactitud g} and Lemma \ref{adj-prec-preen} (a), we get a $g_{z}(\mathcal{X})$-precover $p_{z}:g_{z}(X_{z})\rightarrow F.$ 
We let the reader to show that the induced morphism $p:\coprod_{z\in Z}g_{z}(X_{z})\rightarrow F$ is an
$\mathbb{X}$-precover. 
\end{proof}

\subsection{Global dimension of $\protect\Rep(Q,\protect\C)$ }

Let $\mathcal{C}$ be an abelian category and $C\in\C.$ The projective dimension of $C$ is $\pd(C):=\min\left\{ n\in\mathbb{N}\,|\:\Ext_{\C}^{k}(C,-)=0\:\forall k>n\right\}.$ Moreover, for $\mathcal{X}\subseteq\mathcal{C},$ its projective dimension is 
$\pd(\X):=\sup\{\pd(X)\,|\:X\in\mathcal{X}\}.$ Finally, the global dimension of $\C$ is defined as $\gldim(\C):=\pd(\C)$.

\begin{lem}\label{pd-coprod}
Let $\kappa$ be an infinite cardinal, $\C$ be an abelian $AB4(\kappa)$
category and $\{X_{i}\}_{i\in I}$  in $\mathcal{C}$
with $|I|<\kappa$. Then $\Ext_{\C}^{n}(\coprod_{i\in I}X_{i},C)\cong \prod_{i\in I}\Ext_{\C}^{n}(X,C),$ $\forall\,C\in\C.$ In particular, 
  $\pd(\coprod_{i\in I}X_{i})=\sup_{i\in I}(\pd(X_{i}))$. 
\end{lem}

\begin{proof} We proceed following some ideas from \cite{argudin2022exactness}. 
Let $Q$ be the quiver with $Q_{0}=I$ and $Q_{1}=\emptyset$. Observe
that, by the universal property of coproducts, we have an adjoint
pair $(S,T)$, where $S:\Rep(Q,\C)\rightarrow\C$ is the functor defined
as $S(F):=\coprod_{i\in I}F_{i}$, and $T:\C\rightarrow\Rep(Q,\C)$
is the functor with $T(C)_{i}=C$ for all $i\in I$ and  $C\in\C.$
Moreover, $T$ is always exact, and $S$ is exact since $\C$ is $AB4(\kappa).$
Hence, by Proposition \ref{lem:Ext vs adjoint 3}, we have a natural
isomorphism $\Ext_{\C}^{n}(S(X),C)\cong\Ext_{\Rep(Q,\C)}^{n}(X,T(C))$
for all $n\geq1$. Lastly, note that $\Ext_{\Rep(Q,\C)}^{n}(X,T(C))\cong\prod_{i\in I}\Ext_{\C}^{n}(X_{i},C)$
since $Q_{1}=\emptyset$. Therefore, for all $n\geq 1$, we get 
$
\Ext_{\C}^{n}(\coprod_{i\in I}X_{i},C)=\Ext_{\C}^{n}(S(X),C)\cong\Ext_{\Rep(Q,\C)}^{n}(X,T(C))\cong\prod_{i\in I}\Ext_{\C}^{n}(X,C)
$
and thus the lemma follows. 
\end{proof}

\begin{cor}\label{cor:dimensiones homologicas}  Let $Q$ be a quiver, $\kappa\geq\max\{\rccn(Q),\ltccn(Q)\}$ be an infinite cardinal such that $\kappa>\max\{|Q_0|, |Q_1|\}$ and $\C$ be an $AB4(\kappa)$
abelian category. Then, the following statements hold true. 
\begin{enumerate}
\item $\pd(F)\leq\sup_{i\in Q_{0}}(\pd(F_i))+1,$ for any $F\in\Rep(Q,\C)$.
\item $\gldim\left(\Rep(Q,\C)\right)\leq\gldim(\mathcal{C})+1\text{.}$
\end{enumerate}
\end{cor}

\begin{proof} We only need to prove (a) since (b) follows from (a).
\

Let $F\in\Rep(Q,\C)$. By Corollary \ref{cor:presentacion canonica},
we have the exact sequence 
\[
\suc[\coprod_{\rho\in Q_{1}}f_{t(\rho)}(F_{s(\rho)})][\coprod_{i\in Q_{0}}f_{i}(F_i)][F][\Gamma_F]
\]
in $\Rep(Q,\C).$ Observe that, by Remark \ref{AB3+Inj=AB4} (a), the functor $f_{i}:\C\to \Rep(Q,\C)$ is exact for all $i\in Q_{0}.$ In particular, by 
 Proposition \ref{prop: Ext vs f vs e}, we conclude that $\pd(f_{i}(C))=\pd(C)$ for all
$C\in\C$ and $i\in Q_{0}.$ 
Hence, by the above exact sequence, we get
$
\pd(F) \leq\max\{\pd(\coprod_{i\in Q_{0}}f_{i}(F_{i})),\pd(\coprod_{\rho\in Q_{1}}f_{t(\rho)}(F_{s(\rho)}))+1\}.$
Moreover, by Lemma \ref{pd-coprod} $\pd(\coprod_{i\in Q_{0}}f_{i}(F_{i}))=\sup_{i\in Q_0}(\pd(f_i(F_i)))=\sup_{i\in Q_{0}}(\pd(F_i))$ and similarly 
$\pd(\coprod_{\rho\in Q_{1}}f_{t(\rho)}(F_{s(\rho)}))=\sup_{\rho\in Q_1}(\pd(F_{s(\rho)}))$ since $\Rep(Q,\C)$ is $AB4(\kappa)$ and 
$\kappa>\max\{|Q_0|, |Q_1|\}.$ Thus  $\pd(F)\leq\sup_{i\in Q_{0}}(\pd(F_i))+1$ and the result follows.
\end{proof}

\begin{rem} (a)  Let $R$ be a ring and $Q$ be any quiver. By Corollary \ref{cor:dimensiones homologicas}, we get that $\gldim(\Rep(Q,\Mod(R)))\leq \gldim(R)+1.$ In particular, if $R$ is semisimple, then $\Rep(Q,\C)$ is hereditary. Notice that this was previously proved
by Gabriel and Roiter in \cite[Sect. 8.2]{gabriel1992representations} if $R$ is a field.
\

(b) Let $K$ be a commutative ring, $\Sigma$ be a small $K$-category such
that $\Hom_{\Sigma}(x,y)$ is a projective (respectively, flat) $K$-module, $\C$ be an AB4 (respectively, AB5) abelian
category, and let $\mathcal{C}^{\Sigma}$ be the category of additive functors
from $\Sigma$ to $\C$. It is a known fact that $\mathcal{C}^{\Sigma}$
is an abelian category \cite[Prop. 1.4.5]{borceux1994handbook}).
In \cite[Cor. 13.4]{ringoids}, Mitchell proved that: 
(b1) $\pd(F)\leq\dim\left(\Sigma\right)+\sup_{s\in\Sigma}(\pd(F(s))$
$\forall F\in\mathcal{C}^{\Sigma},$ and 
(b2) $\gldim(\mathcal{C}^{\Sigma})\leq\dim\left(\Sigma\right)+\gldim(\mathcal{C})$,
where $\dim(\Sigma)$ is the Hochschild-Mitchell dimension of $\Sigma$
\cite[Sect. 12]{ringoids}. Observe that every quiver induces
an small category $\Sigma_{Q}$, where $\operatorname{Ob}(\Sigma_{Q})=Q_{0}$
and $\Hom_{\Sigma_{Q}}(x,y)=\mathbb{Z}^{(Q(x,y))}$ for all $x,y\in Q_{0}$.
Moreover, it can be shown that $\Rep(Q,\C)\cong\mathcal{C}^{\Sigma_{Q}}$
(see \cite[p.11]{ringoids}). Therefore, Corollary \ref{cor:dimensiones homologicas}
can be seen as a  specialization of \cite[Cor. 13.4]{ringoids}
with weaker hypotheses (e.g., if $Q$ if finite, the condition $AB4$ is not needed).
\end{rem}

We suspect that, in case $Q_{1}\neq\emptyset$, the equality $\gldim\left(\Rep(Q,\C)\right)=\gldim(\mathcal{C})+1$
holds true under mild conditions. In the following we will prove this for some
special cases. 

\begin{lem}\label{lem: lemita A2} Let $Q:=A_{2}=\{1\rightarrow2\}$, $\C$ be
an abelian category, $X\in\C$ be a non-zero object, $F:=(B\overset{f}{\rightarrow}A)$
be an object in $\Rep(Q,\C)$, and $c:f_{2}(X)\rightarrow F$
be a morphism in $\Rep(Q,\C)$. Then, there is $0\neq\overline{\delta_X}\in\Ext_{\Rep(Q,\C)}^{1}(g_{1}(X),f_{2}(X))$
satisfying that: $c\cdot\overline{\delta_X}=0$ $\Leftrightarrow$  $\exists h:X\rightarrow B$ in $\C$ such that $c_2=f\circ h.$
\end{lem}

\begin{proof} We have $f_2(X)=X^{(Q(2,-))}=(0\to X),$ $g_1(X)=X^{Q(-,1)}=(X\to 0)$ and $g_2(X)=X^{Q(-,2)}=(X\xrightarrow{1_X} X).$
Consider the following split-exact and commutative diagrams in $\C$ 
\[\begin{tikzcd}
	0 && X && X & B && {B\coprod X} && X \\
	\\
	X && X && 0 & A && A && 0
	\arrow[hook, from=1-1, to=1-3]
	\arrow["{1_X}", two heads, from=1-3, to=1-5]
	\arrow[from=1-1, to=3-1]
	\arrow["{1_X}", from=1-3, to=3-3]
	\arrow[from=1-5, to=3-5]
	\arrow["{1_X}"', hook, from=3-1, to=3-3]
	\arrow[two heads, from=3-3, to=3-5]
	\arrow["{}", hook, from=1-6, to=1-8]
	\arrow["{}", two heads, from=1-8, to=1-10]
	\arrow["f"', from=1-6, to=3-6]
	\arrow["{1_A}"', hook, from=3-6, to=3-8]
	\arrow["{t}"', from=1-8, to=3-8]
	\arrow[two heads, from=3-8, to=3-10]
	\arrow[from=1-10, to=3-10]
\end{tikzcd}\]
where $t=(f\;c_2):B\coprod X\to A.$
Let $\delta_X:\;\suc[f_{2}(X)][g_{2}(X)][g_{1}(X)][a][b]$ be the
exact sequence in $\Rep(Q,\C)$ given by the diagram on the left. Observe that $\overline{\delta_X}\neq 0$ since $X\neq 0.$ Moreover, 
by doing the pushout of $a$ and $c$, we get the short exact sequence
$\eta:\:\suc[F][G][g_{1}(X)]$ given by the diagram on the right.
Lastly, it can be shown that $\eta$ splits if and only if there is $h:X\rightarrow B$
such that $c_2=f\circ h$. 
\end{proof}

Let $Q$ be a quiver and $\C$ be an abelian category. For every $i\in Q_{0},$ we recall that the \textbf{stalk functor} $s_{i}:\C\to\Rep(Q,\C)$ is defined as
follows: for $C\in\C,$ $(s_{i}(C))_i:=C$ and $(s_{i}(C))_j:=0$ $\forall j\in Q_{0}-\{i\};$
and $(s_{i}(C))_\alpha:=0$ $\forall\alpha\in Q_{1}.$ 

In case it is
necessary to highlight in which quiver we are working, we use
the notation $s_{i}^{Q}$ instead of $s_{i}$. Notice that the stalk functor is always an exact one.

\begin{lem}\label{lem: lemita A1} Let $Q:=\widetilde{A}_{1}=\{1\rightarrow1\}$,
$\C$ be an $AB3^{*}(\aleph_{1})$ abelian category, $X\in\C$ a non-zero
object, $F=(A\overset{a}{\rightarrow}A)$ be an object in $\Rep(Q,\C)$ 
and $r:s_{1}(X)\rightarrow F$ in $\Rep(Q,\C)$. Then,
there is $0\neq\overline{\delta_X}\in\Ext_{\Rep(Q,\C)}^{1}(g_{1}(X),s_{1}(X))$
satisfying that:  $r_1:X\rightarrow A$  is factorized in $\C$ through $a:A\to A$ if  $r\cdot\overline{\delta_X}=0.$ 
\end{lem}

\begin{proof} Notice that $s_{1}(X)=(X\overset{0}{\rightarrow}X)$ and 
 $g_{1}(X)=(X^{\mathbb{N}}\overset{s}{\rightarrow}X^{\mathbb{N}})$
with $s$ the unique morphism such that $\pi_{n}\circ s=\pi_{n+1}$
for all $n\geq0$, where $\pi_{n}:X^{\mathbb{N}}\rightarrow X$ is
the canonical projection. Moreover, since $X^{\mathbb{N}}=X\coprod X^{\mathbb{N}^{*}},$ for $\mathbb{N}^{*}:=\mathbb{N}-\{0\}$,
we have that $g_{1}(X)=(X\coprod X^{\mathbb{N}^{*}}\xrightarrow{\left[\begin{smallmatrix}0 & \pi'_{1}\\
0 & s'
\end{smallmatrix}\right]}{}X\coprod X^{\mathbb{N}^{*}})$ with $s'$ the unique morphism such that $\pi'_{n}\circ s'=\pi'_{n+1}$
for all $n\geq1$, where $\pi'_{n}:X^{\mathbb{N}^{*}}\rightarrow X$
is the canonical projection. Notice that we have an exact sequence $\varepsilon:\;s_{1}(X)\overset{\left[\begin{smallmatrix}1\\
0
\end{smallmatrix}\right]}{\hookrightarrow}g_{1}(X)\overset{\left[\begin{smallmatrix}0 & 1\end{smallmatrix}\right]}{\twoheadrightarrow}g_{1}(X)^{*}$ which does not split, where $g_{1}(X)^{*}=(X^{\mathbb{N}^{*}}\overset{s'}{\rightarrow}X^{\mathbb{N}^{*}})$.
Indeed, if $\overline{\varepsilon}=0$, then there is $h:X^{\mathbb{N}^{*}}\rightarrow X$ in $\C$
such that $h\circ s'=\pi'_{1}$. But this implies that 
\[
0_{X}=h\circ0=h\circ\left(s'\circ\mu'_{1}\right)=\left(h\circ s'\right)\circ\mu'_{1}=\pi'_{1}\circ\mu'_{1}=1_{X},
\]
where $\mu'_{1}:X\rightarrow X^{\mathbb{N}^{*}}$ is the canonical
inclusion, contradicting that $X\neq 0.$ Now, by the universal property of products, there is an isomorphism $\varphi:X^{\mathbb{N}}\to X^{\mathbb{N}^*}$ in $\C$ such that $\pi_{n+1}\circ\varphi=\pi_n$ for all $n\in\mathbb{N}$ and thus we get that 
$\varphi:g_1(X)\to g_{1}(X)^{*}$ is an isomorphism in $\Rep(Q,\C).$ Hence $\delta_X:\;s_{1}(X)\overset{\left[\begin{smallmatrix}1\\
0
\end{smallmatrix}\right]}{\hookrightarrow}g_{1}(X)\overset{\left[\begin{smallmatrix}0 & \varphi^{-1}\end{smallmatrix}\right]}{\twoheadrightarrow}g_{1}(X)$ is an exact sequence in $\Rep(Q,\C)$ which does not split. 
\

Assume now that $r\cdot\overline{\delta_X}=0.$ Then, there is a morphism $t=(t_1\;t_2):X\coprod X^{\mathbb{N}^{*}}\to A$ in $\C$ such that 
$t:g_1(X)\to F$ is a morphism in $\Rep(Q,\C)$ satisfying that $r=t\circ\begin{pmatrix}1\\0\end{pmatrix}=t_1.$ Using now that $t$ is a morphism of representations, it follows that $at_2=r\pi'_{1}+t_2s'.$ Therefore $r= at_2\mu'_{1}$ since $\pi'_{1}\mu'_{1}=1_X$ and $s'\mu'_{1}=0.$
\end{proof}
\begin{lem}\label{lem: lemita subcarcajes A2,A1}
For a quiver $Q$ having at most two vertices $x,y$ and only one arrow $\rho_0:x\to y,$ an abelian category $\C$ and $0\neq X\in\C,$ the following statements hold true, for  $n\geq1.$  
\begin{enumerate}
\item Let $x\neq y.$ Then there is $0\neq\overline{\delta_X}\in\Ext_{\Rep(Q,\C)}^{1}(g_{x}(X),f_{y}(X))$
such that the Yoneda's product $\overline{f_{y}(\eta)}\cdot\overline{\delta_X}\neq0$
for any non-zero extension $\overline{\eta}\in\Ext_{\mathcal{C}}^{n}(X,Y)$.

\item Let $x=y$ and $\C$ be $AB3^{*}(\aleph_{1}).$  Then there is $\overline{\delta_X}\in\Ext_{\Rep(Q,\C)}^{1}(g_{x}(X),s_{x}(X))$ which is non zero and
such that the Yoneda's product $\overline{s_{x}(\eta)}\cdot\overline{\delta_X}\neq0$
for any non-zero extension $\overline{\eta}\in\Ext_{\mathcal{C}}^{n}(X,Y).$
\end{enumerate}
\end{lem}

\begin{proof} (a)  Let $\delta_X$ be the exact sequence built in Lemma
\ref{lem: lemita A2} and $\overline{\eta}\in\Ext_{\mathcal{C}}^{n}(X,Y)$
be a non-zero extension. Since $f_y:\C\to \Rep(Q,\C)$ is an exact functor (see Remark \ref{AB3+Inj=AB4} (c)), we get the 
extension $\overline{f_{y}(\eta)}\in\Ext_{\mathcal{C}}^{n}(f_y(X),f_y(Y)).$ Suppose that $\overline{f_{y}(\eta)}\cdot\overline{\delta_X}=0.$ Then, by \cite[Chap. VII, Lem. 4.1]{mitchell}, 
there is $d\in\Hom_{\Rep(Q,\C)}(f_{y}(X),F)$ and an extension
$\overline{\eta'}\in\Ext_{\Rep(Q,\C)}^{n}(F,f_{y}(Y))$ such that $\overline{\eta'}\cdot d=\overline{f_{y}(\eta)}$
and $d\cdot\overline{\delta_X}=0$.\\
\begin{minipage}[t]{0.64\columnwidth}%
Now, let $F=(B\overset{f}{\rightarrow}A).$ Since $f_y(X)=(0\to X),$ we have that $d_x=0:0\to A$ and $d_y:X\to A.$ Using that $\overline{\eta'}\cdot d=\overline{f_{y}(\eta)}$, we
have the commutative diagram on the right, where the top rows are
the last term of $f_{y}(\eta)$, and the bottom rows are the last
term of $\eta'$. In particular, note that $f=uf'$. On the other
hand, by Lemma \ref{lem: lemita A2}, we know that there is $h:X\rightarrow B$
such that $d_y=f h$. Consider the restriction functor $\pi_{\{y\}}:\Rep(Q,\C)\to \Rep(\{y\},\C).$ Using the equality 
 $\overline{\eta'}\cdot d=\overline{f_{y}(\eta)}$ and Lemma \ref{lem: Ext vs pi, iota} (b),  we get 
\end{minipage}\hfill{}%
\fbox{\begin{minipage}[t]{0.3\columnwidth}%
\[
\begin{tikzpicture}[-,>=to,shorten >=1pt,auto,node distance=1.5cm,main node/.style=]

 \begin{scope}
\node (A1) at (0,0)  {$0$};
\node (A2) at (1,0)  {$0$};
\node (A3) at (2,0)  {$0$};
\node (A1') at (0,-1)  {$Y_{n-1}$};
\node (A2') at (1,-1)  {$Z_n$};
\node (A3') at (2,-1)  {$X$};
\end{scope}

 \begin{scope}[xshift=-1cm,yshift=-2cm]
\node (B1) at (0,0)  {$0$};
\node (B2) at (1,0)  {$B$};
\node (B3) at (2,0)  {$B$};
\node (B1') at (0,-1)  {$U'_{n-1}$};
\node (B2') at (1,-1)  {$Z'_n$};
\node (B3') at (2,-1)  {$A$};
\end{scope}

\draw[-, double]  (A1)  to  node  {$$}  (B1);
\draw[-, double]  (A1')  to  node  {$$}  (B1');
\draw[->, thin]   (A2)  to  node  {$$} (B2);
\draw[->, thin]   (A2')  to  node  {$$} (B2');
\draw[->, thin]   (A3)  to  node  {$$} (B3);
\draw[->, thin]   (A3')  to  node  {$d_y$} (B3');

\draw[->>, thin]  (A2)  to  node  {$$} (A3);
\draw[->>, thin]  (A2')  to  node  {$$} (A3');
\draw[->, thin]  (A1')  to  node  {$$} (A2');
\draw[->, thin]  (A1)  to  node  {$$} (A2);
\draw[->, thin]  (A1)  to  node  {$$} (A1');
\draw[->, thin]  (A2)  to  node  {$$} (A2');
\draw[->, thin]  (A3)  to  node  {$$} (A3');

\draw[->>, thin]  (B2)  to  node  {$1$} (B3);
\draw[->>, thin]  (B2')  to  node  {$u$} (B3');
\draw[->, thin]  (B1')  to  node  {$$} (B2');
\draw[->, thin]  (B1)  to  node  {$$} (B2);
\draw[->, thin]  (B1)  to  node  {$$} (B1');
\draw[->, thin]  (B2)  to [left] node  {$f'$} (B2');
\draw[->, thin]  (B3)  to [left] node  {$f$} (B3');

\end{tikzpicture}
\]%
\end{minipage}}\\
$$
\overline{\eta}=\pi_{\{y\}}(\overline{f_{y}(\eta)})=\overline{\pi_{\{y\}}(\eta')}\cdot d_y=(\overline{\pi_{\{y\}}(\eta')}\cdot u)\cdot( f'h) =0
$$
since $\overline{\pi_{\{y\}}(\eta')}\cdot u=0,$ which gives a contradiction and thus (a) follows.
\

(b) Let $\delta_X$ be the exact sequence built in Lemma
\ref{lem: lemita A1} and $\overline{\eta}\in\Ext_{\mathcal{C}}^{n}(X,Y)$
be a non-zero extension. Since $s_x:\C\to \Rep(Q,\C)$ is an exact functor, we get the 
extension $\overline{s_x(\eta)}\in \Ext_{\mathcal{C}}^{n}(s_x(X),s_x(Y)).$ Suppose that $\overline{s_{x}(\eta)}\cdot\overline{\delta_X}=0.$ Then, by  \cite[Chap. VII, Lem. 4.1]{mitchell}, 
there is $d\in\Hom_{\Rep(Q,\C)}(s_{x}(X),F)$ and an extension
$\eta'\in\Ext_{\Rep(Q,\C)}^{n}(F,s_{x}(Y))$ such that $\overline{\eta'}\cdot d=\overline{s_{x}(\eta)}$
and $d\cdot\overline{\delta_X}=0$.\\
\begin{minipage}[t]{0.64\columnwidth}%
Now, let $F=(A\overset{a}{\rightarrow}A)$. Using that $\overline{\eta'}\cdot d=\overline{s_{x}(\eta)}$,
we have the commutative diagram on the right, where the top row is
the last term of $s_{x}(\eta)$, and the bottom row is the last term
of $\eta'$. In particular, note that there is $\gamma:A\rightarrow A'$
such that $a'=\gamma\beta$ since $a'\alpha=0$. In particular $a\beta=\beta a'=\beta\gamma\beta$ and thus  $a=\beta\gamma$ since $\beta$ is an epimorphism.
On the other hand, by Lemma \ref{lem: lemita A1}, we know that $d=a d'$ for some $d':X\to A.$ Let 
$\pi_{\{x\}}:\Rep(Q,\C)\to \Rep(\{x\},\C)$ be the restriction functor.  Using that $\overline{\eta'}\cdot d=\overline{s_{x}(\eta)}$ 
we have that 
\[
\overline{\eta}=\pi_{\{x\}}(\overline{s_x(\eta)})=\overline{\pi_{\{x\}}(\eta')}\cdot d=(\overline{\pi_{\{x\}}(\eta')}\cdot\beta)\cdot(\gamma d')=0
\]
\end{minipage}\hfill{}%
\fbox{\begin{minipage}[t]{0.3\columnwidth}%
\[
\begin{tikzpicture}[-,>=to,shorten >=1pt,auto,node distance=2cm,main node/.style=,x=1.5cm,y=1.5cm]

\node (0) at (0,0) {$Y$};
\node (1) at (1,0) {$Z$};
\node (2) at (2,0) {$X$};

\node (0') at (0,-1) {$Y$};
\node (1') at (1,-1) {$A'$};
\node (2') at (2,-1) {$A$};

\draw[->, thin]  (0)  to  node  {$$} (1);
\draw[->>, thin]  (1)  to node  {$$} (2);

\draw[->, thin]  (0')  to node  {$\alpha$} (1');
\draw[->>, thin]  (1')  to  node  {$\beta$} (2');

\draw[-, double]  (0)  to  node  {$$} (0');
\draw[->, thin]  (1)  to  node  {$d'$} (1');
\draw[->, thin]  (2)  to  node  {$d$} (2');

\draw[->, thin]  (0)  to [out=60,in=120,distance=5mm,above] node  {$0$} (0);
\draw[->, thin]  (1)  to [out=60,in=120,distance=5mm,above] node  {$0$} (1);
\draw[->, thin]  (2)  to [out=60,in=120,distance=5mm,above] node  {$0$} (2);

\draw[->, thin]  (0')  to [out=240,in=300,distance=5mm,below] node  {$0$} (0');
\draw[->, thin]  (1')  to [out=240,in=300,distance=5mm,below] node  {$a'$} (1');
\draw[->, thin]  (2')  to [out=240,in=300,distance=5mm,below] node  {$a$} (2');
   
\end{tikzpicture}
\]%
\end{minipage}}\\
since $\overline{\pi_{\{x\}}(\eta')}\cdot\beta=0,$ which gives a contradiction proving (b).
\end{proof}

For a subquiver $S\subseteq Q,$ we consider the extension $\iota_S:\Rep(S,\C)\to \Rep(Q,\C),$ see details before Lemma \ref{lem: Ext vs pi, iota}.

\begin{prop}\label{prop-previo-thm-dh} Let $Q$ be a quiver having an arrow $\rho_0:x\to y,$  $S:=\{x\xrightarrow{\rho_0} y\}$ be a subquiver of $ Q$  and $\C\neq 0$ be an abelian category. Consider the following conditions:
\begin{itemize}
\item[(a)] $x\neq y;$
\item[(b)] $x=y$ and $\C$ is $AB3^{*}(\aleph_1).$
\end{itemize}
If one of the above conditions holds true and $C\in\C$ is such that $\Ext^n_\C(C,-)\neq 0,$ for some $n\geq 0,$ then $\Ext^{n+1}_{\Rep(Q,\C)}(\iota_Sg^S_x(C),-)\neq 0.$ Moreover $\gldim(\Rep(Q,\C))\geq \gldim(\C)+1.$ 
\end{prop}
\begin{proof}  Let $n=0.$ If (a) or (b) holds true and $0\neq C\in\C,$  then by Lemmas \ref{lem: lemita A2} and \ref{lem: lemita A1}, we get that $\Ext^{1}_{\Rep(S,\C)}(g^S_x(C),-)\neq 0.$ Therefore, by Lemma \ref{lem: Ext vs pi, iota} (c), it follows that $\Ext^{1}_{\Rep(Q,\C)}(\iota_S g^S_x(C),-)\neq 0.$
\

Let $n\geq 1.$ Assume that $x\neq y$ and let $C\in\C$  be such that $\Ext^n_\C(C,-)\neq 0.$   Then, there is $D\in\C$ and  $0\neq\overline{\eta}\in \Ext^n_\C(C,D).$  By  Lemma \ref{lem: lemita subcarcajes A2,A1} (a), there is $\overline{\delta_C}\in\Ext_{\Rep(S,\C)}^{1}(g_{x}^{S}(C),f_{y}^{S}(C))$
such that $\theta_C:=\overline{f_{y}^{S}(\eta)}\cdot\overline{\delta_C}\neq 0.$  Then, from  Lemma \ref{lem: Ext vs pi, iota} (c), we get that 
$\iota_S(\theta_C)\in \Ext^{n+1}_{\Rep(Q,\C)}(\iota_Sg^S_x(C),\iota_S f^S_y(C))$
is non zero. For the case when (b) holds true, the proof is very similar to the given one (use Lemma \ref{lem: lemita subcarcajes A2,A1} (b)) and is left to the reader. Finally, we show that $\gldim(\Rep(Q,\C))\geq \gldim(\C)+1.$ Indeed,  by the previous arguments, we may assume that $n:=\gldim(\C)<\infty.$ Thus, there is some $C\in\C$ with $\Ext^n_\C(C,-)\neq 0$ and $n\geq 0.$ By the proved above, we get that 
$n+1\leq \pd(\iota_Sg^S_x(C))\leq\gldim(\Rep(Q,\C))$ and thus the result follows.
\end{proof}

\begin{thm}\label{thm:dimensiones homologicas} Let $Q$ be a quiver having an arrow $\rho_0:x\to y,$ $\kappa$ be an infinite cardinal such that $\kappa\geq\max\{\rccn(Q),\ltccn(Q)\}$ and $\kappa>\max\{|Q_0|, |Q_1|\}$ and let $\C\neq 0$ be an $AB4(\kappa)$ abelian
category.
Consider the following conditions:
\begin{itemize}
\item[(a)] $x\neq y;$ 
\item[(b)] $x=y$ and  $\C$ is $AB3^{*}(\aleph_1).$ 
\end{itemize}
If one of the above conditions holds true, then $\gldim(\Rep(Q,\C))=\gldim(\C)+1.$
\end{thm}
\begin{proof} It follows from Corollary \ref{cor:dimensiones homologicas} and Proposition \ref{prop-previo-thm-dh}. 
\end{proof}

In what follows we will give some  applications of the theorem above. In order to do that, we introduce the following notation which is inspired by the work of Z. Leszczy{\'n}ski in 
\cite{leszczynski1994representation}. 

\begin{defn}
Let $Q^{1},Q^{2}.\cdots,Q^{n},\cdots$ be quivers and $\C$ be an
abelian category. Define:
\begin{enumerate}
\item $\Rep(Q^{1}\otimes Q^{2},\C):=\Rep(Q^{1},\Rep(Q^{2},\C))$;
\item $\Rep(Q^{1}\otimes\cdots\otimes Q^{n},\C):=\Rep(Q^{1},\Rep(Q^{2}\otimes\cdots\otimes Q^{n},\C))\text{ for all }n>2$;
\item $\Rep(Q^{\otimes n},\C):=\Rep(Q^{1}\otimes\cdots\otimes Q^{n},\C)$
for all $n\geq1$ if $Q=Q^{k}$ for all $1\leq k\leq n$. 
\end{enumerate}
\end{defn}

\begin{example} The global dimension of tensor quivers. Consider a family of quivers $Q^{1},Q^{2}.\cdots,Q^{n}$ with non-empty arrow sets, and $R$ be a non zero ring. Then, by  Theorem \ref{thm:dimensiones homologicas}, we get $\gldim(\Rep(Q^{1}\otimes\cdots\otimes Q^{n},\Mod(R)))=\gldim(R)+n.$\\
A particular case of the above result is the following one: Assume in addition that the mentioned  quivers are finite and do not have oriented cycles, $R=k$ is a field and take 
$T:=kQ^{1}\otimes_{k}kQ^{2}\otimes_{k}\cdots\otimes_{k}kQ^{n}$ the tensor $k$-algebra. Then, by \cite[Lem. 1.3]{leszczynski1994representation},
it can be seen that $\Rep(Q^{1}\otimes\cdots\otimes Q^{n},\Mod(k))\simeq \Mod(T)$ and thus $\gldim(T)=n.$
\end{example}

\begin{example}\label{E-HM} The Hochschild-Mitchell's theorem for abelian categories. Let $\C\neq 0$ be an $AB4(\aleph_1)$ and $AB3^*(\aleph_1)$ abelian category. Consider a quiver $Q$ with $|Q_0|=1$ and $1\leq |Q_1|=n<\infty.$ Then $\gldim(\Rep(Q,\C))=\gldim(\C)+1.$\\
{\it Indeed, observe that $\rccn(Q)=\aleph_1=\ltccn(Q).$ Thus by Theorem \ref{thm:dimensiones homologicas} we get that $\gldim(\Rep(Q,\C))=\gldim(\C)+1.$}\\
Notice that the same result was obtained by Mitchell in \cite[Chap. IX, Thm. 1.6]{mitchell} by assuming that $\C$ is $AB4$ and has enough projectives. Another particular case is the following one: take a non zero ring $R$ and let 
$R\langle x_1,x_2,\cdots,x_n\rangle$ be the free ring over $R$ with $n$ variables. Since 
$\Rep(Q,\Mod(R))\simeq \Mod(R\langle x_1,x_2,\cdots,x_n\rangle),$ we get that $\gldim(R\langle x_1,x_2,\cdots,x_n\rangle)=\gldim(R)+1$ which was also obtained by Hochschild in \cite{hochschild1958note}.
\end{example}

\begin{example}\label{E-ERZM} The Eilenberg-Rosenberg-Zelinsky-Mitchell's theorem for abelian categories.
Let $Q$ be a quiver with $|Q_{0}|=1=|Q_{1}|$ and $\C\neq 0$ be an $AB4(\aleph_1)$ and $AB3^*(\aleph_1)$ abelian category. Then $\gldim(\Rep(Q^{\otimes n},\C))=\gldim(\C)+n.$\\
{\it Indeed, notice that $\rccn(Q)=\aleph_1=\ltccn(Q).$ Thus by Theorem \ref{thm:dimensiones homologicas} we get that $\gldim(\Rep(Q^{\otimes n},\C))=\gldim(\C)+n.$}\\
Notice that the same result was obtained by Mitchell in \cite[Chap. IX, Thm. 2.1]{mitchell} by assuming that $\C$ is $AB4$ and has enough
projectives. Another particular case is  the following one: take a non zero ring $R.$ Since $\Rep(Q^{\otimes n},\Mod(R))\simeq \Mod(R[x_{1},\cdots,x_{n}]),$ we get that $\gldim(R[x_{1},\cdots,x_{n}])=\gldim(R)+n$ which was obtained  by Eilenberg, Rosenberg and Zelinsky in \cite{eilenberg1957dimension}.
\end{example}

\subsection{Projectives and generators in the category of quiver representations}

Let $\mathcal{C}$ be an abelian category. We recall that a class
$\mathcal{U}\subseteq\mathcal{C}$ is \textbf{generating} (resp. \textbf{cogenerating})
in $\mathcal{C}$ if, for every $C\in\mathcal{C}$, there is $U\twoheadrightarrow C$
(resp. $C\hookrightarrow U$) in $\C$ with $U\in\mathcal{U}$. In
this section we will show that if $\mathcal{C}$ has a generating
(resp. cogenerating) class, then $\operatorname{Rep}(Q,\mathcal{C})$
also has one. 

The following result is a generalization of \cite[Proposition 3.10]{holm2019cotorsion}
and it is inspired on its proof. We recall that the \textbf{support}
of the representation $F\in\Rep(Q,\C)$ is the set of vertices $\Supp(F):=\{i\in Q_{0}\;|\;F(i)\neq0\}.$

\begin{prop}\label{prop:f y g vs clases ortogonales} For a quiver $Q,$ $\kappa\geq\max\{\rccn(Q),\ltccn(Q)\}$
an infinite cardinal number and an abelian category $\C$ which is
AB3($\kappa$), the following statements are true. 
\begin{itemize}
\item[(a)] Let $F\in\Rep(Q,\C)$ and $\mathcal{U}$ be a generating class of
$\C.$ Then, there is a family $\{U_{i}\}_{i\in\mathrm{Supp}(F)}$
of objects in $\mathcal{U}$ and  $\coprod_{i\in\mathrm{Supp}(F)}f_{i}(U_{i})\twoheadrightarrow F$
in $\Rep(Q,\C).$ In particular, $\coprod_{\leq|Q_{0}|}f_{*}(\mathcal{U})$
is a generating class for $\Rep(Q,\C).$
\item[(b)] If $\C$ has enough projectives, then $\Rep(Q,\C)$ has enough projectives
and $\Add_{\leq|Q_{0}|}(f_{*}(\Proj(\C)))=\Proj(\Rep(Q,\C)).$ 
\item[(c)] In general, $\Add_{\leq|Q_{0}|}(f_{*}(\Proj(\C)))\subseteq\Proj(\Rep(Q,\C))\subseteq\Add_{\leq|Q_{0}|}(f_{*}(\C))$. 
\item[(d)] If $\C$ is Ab3{*}($\kappa'$) where $\kappa'\geq\lccn(Q)$ is an infinite cardinal and $g_{i}$ is exact for all $i\in I$, then $\Proj(\Rep(Q,\C))=\Add_{\leq|Q_{0}|}(f_{*}(\Proj(\C)))$.
\end{itemize}
\end{prop}

\begin{proof}
(a) We may assume that $\Supp(F)\neq\emptyset.$ Since $\mathcal{U}$
is a generating class of $\C,$ for each $i\in\Supp(F),$ there is
an epimorphism $\pi_{i}:U_{i}\to e_{i}(F)$ in $\C.$ We extend the
family $\{U_{i}\}_{i\in\mathrm{Supp}(F)}$ to a family $\{U_{i}\}_{i\in Q_{0}}$
by defining $U_{i}:=0$ if $e_{i}(F)=F_i=0.$ Thus, by Lemma \ref{LemAux},
we have that $\coprod_{i\in Q_{0}}f_{i}(U_{i})$ exists in $\Rep(Q,\C)$
and $\coprod_{i\in Q_{0}}f_{i}(U_{i})=\coprod_{i\in\Supp(F)}f_{i}(U_{i}).$
Moreover, $\coprod_{i\in Q_{0}}\pi_{i}:\coprod_{i\in Q_{0}}f_{i}(U_{i})\to\coprod_{i\in Q_{0}}f_{i}e_{i}(F)$ is an epimorphism.
Thus, by doing the composition with the epimorphism appearing in Corollary
\ref{cor:presentacion canonica}, we obtain the sought epimorphism.
\\
(b) Let $\C$ have enough projectives. Then $\Proj(\C)$ is a generating
class in $\C$ and thus, from (a), we get that $\coprod_{\leq|Q_{0}|}f_{*}(\Proj(\C))$
is a generating class for $\Rep(Q,\C).$ On the other hand, by Proposition
\ref{rem:adjuntos de ei}(a) $f_{*}(\Proj(\C))\subseteq\Proj(\Rep(Q,\C)).$
Therefore, $\Add_{\leq|Q_{0}|}(f_{*}(\Proj(\C)))=\Proj(\Rep(Q,\C)).$
\\
(c) The first inclusion follows from Proposition \ref{rem:adjuntos de ei} (a).   If $F\in\Proj(\Rep(Q,\C))$, then the short exact sequence appearing
in Corollary \ref{cor:presentacion canonica} splits. Hence, $F$
is a direct summand of $\coprod_{i\in Q_{0}}f_{i}e_{i}(F)$, and thus,
$F\in\Add_{\leq|Q_{0}|}(f_{*}(\C))$. \\
(d) Let $F\in\Proj(\Rep(Q,\C)).$  Consider the arguments made in (c). If $g_{i}$ is exact for all
$i\in I$, we have that $e_{i}(F)\in\Proj(\C)$ by Proposition \ref{prop: Ext vs f vs e} (b).
Then, by Proposition \ref{rem:adjuntos de ei} (a), it follows that $f_{i}e_{i}(F)\in\Proj(\Rep(\C))$
for all $i\in Q_{0}$. Therefore, $F\in\Add_{\leq|Q_{0}|}(f_{*}(\Proj(\C)))$ and thus $\Proj(\Rep(Q,\C))\subseteq \Add_{\leq|Q_{0}|}(f_{*}(\Proj(\C))).$
\end{proof}
Note that Proposition \ref{prop:f y g vs clases ortogonales} can
be applied in some situations where \cite[Corollary 3.10]{holm2019cotorsion}
can not be used. For more details, see the following examples.

\begin{example}\label{exa: proyectivos/inyectivos}  Let $Q$ be a quiver and $\mathcal{C}$
be an abelian category. 
\begin{enumerate}
\item Let $Q$ be acyclic and finite. Then 
$\Proj(\Rep(Q,\C))=\Add_{\leq|Q_{0}|}(f_{*}(\Proj(\C))$ and 
$\Inj(\Rep(Q,\C))=\Prod_{\leq|Q_{0}|}(g_{*}(\Inj(\C)).$
Moreover, if $\C$ has enough projectives
(resp. injectives) then so is $\Rep(Q,\C).$ 
\\
{\it  Indeed, since $Q$ is acyclic and finite, we have $\ccn(Q)$ and
$\tccn(Q)$ are finite. Hence, we can take $\kappa:=\aleph_{0}$ and
thus $\C$ is AB3($\kappa$) and AB3{*}($\kappa$). Therefore, by
Remark \ref{AB3+Inj=AB4} (c), Proposition \ref{prop:f y g vs clases ortogonales} (d) and its dual,
we get the description given above of the projectives and injectives in $\Proj(\Rep(Q,\C)).$  Finally, by Proposition \ref{prop:f y g vs clases ortogonales} (a) and its dual, we get (a).}

\item Let $Q$ be a finite-cone-shape quiver. Then, we have 
$\qquad\Proj(\Rep(Q,\C))=\Add_{\leq|Q_{0}|}(f_{*}(\Proj(\C))$ and 
$\Inj(\Rep(Q,\C))=\Prod_{\leq|Q_{0}|}(g_{*}(\Inj(\C)).$
Moreover, if $\C$ has enough projectives
(resp. injectives) then so is $\Rep(Q,\C).$\\ 
{\it Indeed, since $Q$ is cone-shape-finite, we have that $\ccn(Q)\leq\tccn(Q)\leq\aleph_{0};$
and thus, we can do as in the example above by taking $\kappa:=\aleph_{0}.$}

\item Let $Q$ be the quiver with $Q_{0}=\{1\}$ and $Q_{1}=\{\alpha\}$,
and $\mathcal{C}:=\modu\,(k)$ be the abelian category of finite dimensional
$k$-vector spaces over a field $k$. In this case, we have that $\rccn(Q)=\lccn(Q)=\ltccn(Q)=\rtccn(Q)=\aleph_{0}^{+}=\aleph_{1}$
and since $\C$ is not AB3($\aleph_{1}$), we can not apply Proposition
\ref{prop:f y g vs clases ortogonales}. Observe that, although $\mathcal{C}$
has enough projectives and enough injectives, $\operatorname{Rep}(Q,\mathcal{C})$
does not. Moreover, we have that $\Inj(\operatorname{Rep}(Q,\mathcal{C}))=0=\Proj(\operatorname{Rep}(Q,\mathcal{C}))$.
The same can be said for the quiver $Q=\widetilde{A}_{n}$ with $n>0$,
where the vertex set is $\{0,\cdots,n\}$ and the arrow set is $\{\alpha_{n}:n\rightarrow0\}\cup\{\alpha_{i}:i\rightarrow i+1\}_{i=0}^{n-1}$.
Indeed, in this case it is known that $\Rep(Q,\C)$ is a finite length
uniserial category (see \cite[p.697]{smalo2001almost} and \cite[Theorem 1.7.1]{chen2009introduction}).
Therefore, for every indecomposable $0\neq M\in\Rep(Q,\C)$, we can
find a non-splitting monomorphism $M\rightarrow N$ and a non-splitting
epimorphism $N\rightarrow M$ (see \cite[p.13]{chen2009introduction}).
Hence, $\Inj(\operatorname{Rep}(Q,\mathcal{C}))=0=\Proj(\operatorname{Rep}(Q,\mathcal{C}))$.

\item Let $k$ be a field, $\mathcal{C}=\modu(k)$ and $Q:\;\cdots\rightarrow-1\rightarrow0\rightarrow1\rightarrow\cdots$.
Notice that $\rccn(Q)=\lccn(Q)=\aleph_{0}$ and $\ltccn(Q)=\rtccn(Q)=\aleph_{0}^{+}=\aleph_{1}.$
By Proposition \ref{rem:adjuntos de ei}, we get that $f_{*}(\C)\subseteq\Proj(\Rep(Q,\C))$
and $g_{*}(\C)\subseteq\Inj(\Rep(Q,\C)).$ Moreover, we can not apply
Proposition \ref{prop:f y g vs clases ortogonales}, nor its dual,
since $\C$ is not AB3($\aleph_{1}$) and it is not AB3{*}($\aleph_{1}$).
Observe that $\Rep(Q,\C)$ does not have enough projectives or injectives.
For example, if $T$ is the representation defined by $T(i)=k$ for
all $i\in Q_{0}$ and $T(\alpha)=0$ for all $\alpha\in Q_{1}$, then
there is no epimorphism $P\rightarrow T$ with $P\in\Proj(\Rep(Q,\C))$. 

\item For any quiver $Q$ and any ring $R,$ we have that $\Rep(Q,\Mod(R))$ has enough projectives and injectives, 
$\Proj(\Rep(Q,\Mod(R)))=\Add_{\leq|Q_{0}|}(f_{*}(\Proj(R))$ and $\Inj(\Rep(Q,\Mod(R)))=\Prod_{\leq|Q_{0}|}(g_{*}(\Inj(R)).$ Moreover, 
we also have that $\gldim(\Rep(Q,\Mod(R))=\gldim(R)+1$ if $R\neq 0$ and $Q_1\neq \emptyset.$   Indeed, it follows from Proposition \ref{prop:f y g vs clases ortogonales} and Theorem \ref{thm:dimensiones homologicas}.
\end{enumerate}
\end{example}

\section{Cotorsion pairs induced by functors}

By carefully reviewing the proof of Theorem \ref{thm:carcajes 1},
we can weaken the conditions on $\mathcal{C}$ if we ask for stronger
conditions on the quiver $Q$ (this has already been partially done in
\cite{holm2019cotorsion}, \cite{odabacsi2019completeness} and \cite{di2021completeness}).
This section is a detailed summary of the results that can be obtained
by doing this using mesh and cone-shape cardinal numbers introduced in section 2. We will see that every cotorsion pair in $\mathcal{C}$
induces a cotorsion pair in $\operatorname{Rep}(Q,\mathcal{C})$ through
the functors defined below.

Let us begin recalling the following construction from \cite{holm2019cotorsion}.

\begin{defn}\cite{enochs2004flat,holm2019cotorsion}\label{def:phi, psi} Let
$\C$ be an abelian category, $\X\subseteq\C$ and $\kappa$ an infinite cardinal
number. 
\begin{itemize}
\item[(a)] Let $\C$ be AB3{*}($\kappa$) and $\kappa\geq\rmcn(Q).$ For every
$F\in\operatorname{Rep}(Q,\mathcal{\mathcal{C}})$ and $i\in Q_{0}$,
\[
\psi_{i}^{F}:F_i\rightarrow\prod_{\alpha\in Q_{1}^{i\rightarrow*}}F_{t(\alpha)}
\]
 is the morphism induced by the morphisms $F_{\alpha}:F_i\to F_{t(\alpha)}.$
Moreover, we have the functor 
\[
k_{i}:\Rep(Q,\C)\to\C,\;F\mapsto\Ker(\psi_{i}^{F}).
\]
We denote by $\Psi(\X)$ the class of the representations $F\in\operatorname{Rep}(Q,\mathcal{\mathcal{C}})$
such that $\Ker\left(\psi_{i}^{F}\right)\in\mathcal{X}$ and $\psi_{i}^{F}$
is an epimorphism $\forall i\in Q_{0}$. If it is necessary to indicate
the quiver we are working with, we will denote the morphism $\psi_{i}^{F}$
as $\psi_{i}^{F,Q}$ and the class $\Psi(\X)$ as $\Psi_{Q}(\X)$.

\item[(b)] Let $\C$ be AB3($\kappa$) and $\kappa\geq\lmcn(Q).$ For every
$F\in\operatorname{Rep}(Q,\mathcal{\mathcal{C}})$ and $i\in Q_{0}$,
\[
\varphi_{i}^{F}:\coprod_{\alpha\in Q_{1}^{*\rightarrow i}}F_{s(\alpha)}\rightarrow F_i
\]
 is the morphism induced by the morphisms $F_{\alpha}:F_{s(\alpha)}\to F_i.$
Moreover, we have the functor 
\[
c_{i}:\Rep(Q,\C)\to\C,\;F\mapsto\Coker(\varphi_{i}^{F}).
\]
We denote by $\Phi(\X)$ the class of the representations $F\in\operatorname{Rep}(Q,\mathcal{\mathcal{C}})$
such that $\Coker(\varphi_{i}^{F})\in\X$ and $\varphi_{i}^{F}$ is
a monomorphism $\forall i\in Q_{0}.$
\item[(c)] For every $i\in Q_{0},$ the \textbf{stalk functor} $s_{i}:\C\to\Rep(Q,\C)$
is defined as follows: For $C\in\C,$ $s_{i}(C)(i):=C$ and $s_{i}(C)(j):=0$
$\forall j\in Q_{0}-\{i\};$ and $s_{i}(C)(\alpha):=0$ $\forall\alpha\in Q_{1}.$
For a class $\mathcal{S}\subseteq\mathcal{C},$ define $s_{*}(\mathcal{S}):=\{s_{i}(S)\,|\:i\in Q_{0}\text{ and }S\in\mathcal{S}\}$. 
\end{itemize}
\end{defn}

\begin{rem}\label{rem:prod,corprod,s} Let $\{X_{i}\}_{i\in Q_{0}}$ be a family
of objects in $\C.$ This family induces a representation $\overline{X}\in\operatorname{Rep}(Q,\mathcal{C}),$
where $\overline{X}_i:=X_{i}$ for all $i\in Q_{0}$ and $\overline{X}_{\alpha}=0$ for
all $\alpha\in Q_{1}$. Observe that $\prod_{i\in Q_{0}}s_{i}(X_{i})=\overline{X}=\coprod_{i\in Q_{0}}s_{i}(X_{i})$
in $\Rep(Q,\C)$ even if $\mathcal{C}$ is not $AB3$ or $AB3^{*}.$ 
\end{rem}

The functors $c_{i}$ and $k_{i}$ defined above are important because
they are, respectively, left and right adjoint functors of the stalk
functor $s_{i},$ see the details below.

\begin{rem}\label{rem:adjuntos de si} \cite[Thm. 4.5]{holm2019cotorsion},\cite[Prop. 2.16]{odabacsi2019completeness}\\
 For a quiver $Q,$ a vertex $i\in Q_{0},$ an abelian category $\C$
and an infinite cardinal number $\kappa,$ the following statements
hold true. 
\begin{itemize}
\item[(a)] $\Hom_{\C}(c_{i}(?),-)\cong\Hom_{\operatorname{Rep}(Q,\mathcal{C})}(?,s_{i}(-))$
if $\C$ is AB3($\kappa$) and $\kappa\geq\lmcn(Q).$ 
\item[(b)] $\Hom_{\operatorname{Rep}(Q,\C)}(s_{i}(?),-)\cong\Hom_{\mathcal{C}}(?,k_{i}(-))$
if $\C$ is AB3{*}($\kappa$) and $\kappa\geq\rmcn(Q).$ 
\end{itemize}
\end{rem}

In the next proposition we can see that, under some mild conditions
on $\mathcal{C}$, there is a similar isomorphism for the groups of
extensions.
\begin{prop}\label{cisi-iso} \cite[Prop. 5.4]{holm2019cotorsion}\cite[Prop. 3.10]{odabacsi2019completeness}\label{prop:s,c,k ext}
 For a quiver $Q,$ a vertex $i\in Q_{0},$ an abelian category $\C$,
and an infinite cardinal number $\kappa,$ the following statements
hold true.
\begin{enumerate}
\item Let $\mathcal{C}$ be AB3($\kappa$) with $\kappa\geq\lmcn(Q).$ Then,
for any $C\in\C$ and $X\in\Rep(Q,\C)$ such that $\varphi_{i}^{X}$
is a monomorphism, we have 
\[
\operatorname{Ext}_{\mathcal{C}}^{1}(c_{i}(X),C)\cong\operatorname{Ext}_{\operatorname{Rep}(Q,\mathcal{C})}^{1}(X,s_{i}(C)).
\]
\item Let $\mathcal{C}$ be AB3{*}($\kappa$) with $\kappa\geq\rmcn(Q).$
Then, for any $C\in\C$ and $X\in\Rep(Q,\C)$ such that $\psi_{i}^{X}$
is an epimorphism, we have 
\[
\operatorname{Ext}_{\mathcal{C}}^{1}(C,k_{i}(X))\cong\operatorname{Ext}_{\operatorname{Rep}(Q,\mathcal{C})}^{1}(s_{i}(C),X).
\]
\end{enumerate}
\end{prop}

The results shown below are intended to facilitate the proof that,
under certain conditions, $(\Phi(\mathcal{A}),\Phi(\mathcal{A})^{\bot_{1}})$
and $(^{\bot_{1}}\Psi(\mathcal{B}),\Psi(\mathcal{B}))$ are cotorsion
pairs in $\operatorname{Rep}(Q,\mathcal{C})$ for all cotorsion pair
$(\mathcal{A},\mathcal{B})$ in $\mathcal{C}.$ 

With the same proof, we have that \cite[Prop.  5.6]{holm2019cotorsion}
can be rewritten as follows.

\begin{prop}\label{Bcogen-Agen} For a quiver $Q,$ a vertex $i\in Q_{0},$ an
abelian category $\mathcal{C},$ an infinite cardinal $\kappa,$ $X\in\operatorname{Rep}(Q,\mathcal{C})$
and $C\in\mathcal{C},$ the following statements hold true. 
\begin{enumerate}
\item Let $\mathcal{C}$ be Ab3($\kappa$) and $\kappa\geq\lmcn(Q).$ If
$s_{i}(C)\in X^{\bot_{1}},$ then $\Hom_{\C}(\varphi_{i}^{X},C):\Hom_{\C}(X_i,C)\to\Hom_{\C}(\coprod_{a\in Q_{1}^{*\to i}}X_{s(a)},C)$
is surjective.\\
 Thus, if $\mathcal{B}$ is a cogenerating class in $\mathcal{C}$
such that $s_{i}(\mathcal{B})\subseteq X^{\bot_{1}}$, then $\varphi_{i}^{X}$
is a monomorphism. 
\item Let $\mathcal{C}$ be Ab3{*}($\kappa$) and $\kappa\geq\rmcn(Q).$
If $s_{i}(C)\in{}^{\bot_{1}}X,$ then $\Hom_{\C}(C,\psi_{i}^{X}):\Hom_{\C}(C,X_i)\to\Hom_{\C}(C,\prod_{a\in Q_{1}^{i\to*}}X_{t(a)},C)$
is surjective.\\
 Thus, if $\mathcal{A}$ is a generating class in $\mathcal{C}$ such
that $s_{i}(\mathcal{A})\subseteq{}^{\bot_{1}}X$, then $\psi_{i}^{X}$
is an epimorphism. 
\end{enumerate}
\end{prop}

With the same proof, \cite[Prop. 7.2]{holm2019cotorsion} can
be rewritten as follows.
\begin{prop}\label{FiRep} For a quiver $Q,$ an abelian category $\C,$ an infinite
cardinal number $\kappa$ and a class $\X\subseteq\C$ closed under
extensions, the following statements hold true: 
\begin{enumerate}
\item Let $\mathcal{C}$ be Ab3($\kappa$) with $\kappa\geq\lmcn(Q)$ and
$\coprod_{<\kappa}\X=\X.$ If $Q$ is left rooted, then $\Phi(\mathcal{X})\subseteq\operatorname{Rep}(Q,\mathcal{X})$. 
\item Let $\mathcal{C}$ be Ab3{*}($\kappa$) with $\kappa\geq\rmcn(Q)$
and $\prod_{<\kappa}\X=\X.$ If $Q$ is right rooted, then $\Psi(\mathcal{X})\subseteq\operatorname{Rep}(Q,\mathcal{X})$. 
\end{enumerate}
\end{prop}

The following result is a generalization of \cite[Prop 7.3]{holm2019cotorsion}.

\begin{prop}\label{prop:f y clases ortogonales} For a quiver $Q,$ an abelian
category $\C,$ an infinite cardinal number $\kappa$ and a class
$\X\subseteq\C,$ the following statements hold true: 
\begin{enumerate}
\item Let $\mathcal{C}$ be Ab3($\kappa$) with $\kappa\geq\lmcn(Q).$ If
$\mathcal{C}$ admits a cogenerating class $\mathcal{B}$ such that
$\mathcal{B}\subseteq\mathcal{X}$, then $\Phi({}^{\bot_{1}}\mathcal{X})={}^{\bot_{1}}s_{*}(\mathcal{X})$. 
\item Let $\mathcal{C}$ be Ab3($\kappa$) with $\kappa\geq\rccn(Q).$ If
$f_{i}:\C\to\Rep(Q,C)$ is exact $\forall\,i\in Q_{0},$ then $f_{*}(\mathcal{X})^{\bot_{n}}=\operatorname{Rep}(Q,\mathcal{X}^{\bot_{n}})$
for all $n\geq1$. 
\item Let $\mathcal{C}$ be Ab3{*}($\kappa$) with $\kappa\geq\rmcn(Q).$
If $\mathcal{C}$ admits a generating class $\mathcal{A}$ such that
$\mathcal{A}\subseteq\mathcal{X}$, then $\Psi(\mathcal{X}^{\bot_{1}})=s_{*}(\mathcal{X})^{\bot_{1}}$. 
\item Let $\mathcal{C}$ be Ab3{*}($\kappa$) with $\kappa\geq\lccn(Q).$
If $g_{i}:\C\to\Rep(Q,C)$ is exact $\forall\,i\in Q_{0},$ then $^{\bot_{n}}g_{*}(\mathcal{X})=\operatorname{Rep}(Q,{}^{\bot_{n}}\mathcal{X})$
for all $n\geq1$. 
\end{enumerate}
\end{prop}

\begin{proof}
The items (a) and (c) follow from Propositions \ref{cisi-iso} and
\ref{Bcogen-Agen}. Finally, (b) and (d) can be shown from Proposition
\ref{prop: Ext vs f vs e}. 
\end{proof}
The following result (and its dual) is a generalization of \cite[Thms.  7.4 and 7.9]{holm2019cotorsion}.

\begin{thm}\label{thm:cotrsion pairs rep} For a quiver $Q,$ an infinite cardinal
number $\kappa,$ an AB3($\kappa$) abelian category $\C$ and a
cotorsion pair $\p$ in $\C,$ the following statements hold true. 
\begin{enumerate}
\item Let $f_{i}:\C\to\Rep(Q,\C)$ be exact $\forall\,i\in Q_{0}$ and $\kappa\geq\rccn(Q).$
Then 
\begin{center}
$f_{*}(\A)^{\perp_{1}}=\Rep(Q,\B)\;$ and $\;f_{*}(\A)\subseteq{}^{\perp_{1}}s_{*}(\B).$ 
\par\end{center}
\item Let $\B$ be a cogenerating class in $\C$ and $\kappa\geq\lmcn(Q).$
Then 
\begin{center}
$\Phi(\A)={}^{\perp_{1}}s_{*}(\B).$ 
\par\end{center}
\item Let $Q$ be left rooted and $\kappa\geq\lmcn(Q).$ Then 
\begin{center}
$\Phi(\A)\subseteq\Rep(Q,\A)\;$ and $\;\Rep(Q,\B)\subseteq\Phi(\A)^{\perp_{1}}.$ 
\par\end{center}
\item If $f_{i}:\C\to\Rep(Q,\C)$ is exact $\forall\,i\in Q_{0},$ $\kappa\geq\max\{\rccn(Q),\lmcn(Q)\},$
$Q$ is left rooted and $\B$ is a cogenerating class in $\C,$ then
$\Rep(Q,\B)=\Phi(\A)^{\perp_{1}}.$ 
\end{enumerate}
\end{thm}

\begin{proof}
The equality in (a) follows from Proposition \ref{prop: Ext vs f vs e}
(a), and the inclusion in (a) can be shown (using Proposition \ref{prop: Ext vs f vs e}
(a)) as in the beginning of the proof of \cite[Prop.  7.9]{holm2019cotorsion}.
The item (b) follows from Proposition \ref{prop:f y clases ortogonales}
(a) since $\A={}^{\perp_{1}}\B.$ \
 We assert that $\coprod_{<\kappa}\A=\A.$ Indeed, consider a family
$\{A_{i}\}_{i\in I}$ in $\A$ with $|I|<\kappa.$ Since $\C$ is
Ab3($\kappa$), we have a monomorphism $\Ext_{\C}^{1}(\coprod_{i\in I}A_{i},B)\to\prod_{i\in I}\Ext_{\C}^{1}(A_{i},B)$
$\forall\,B\in\B,$ and hence $\coprod_{i\in I}A_{i}\in{}^{\perp_{1}}\B=\A.$
\

Let us show (c). Indeed, the inclusion $\Phi(\A)\subseteq\Rep(Q,\A)$
follows from Proposition \ref{FiRep} (a) since $\coprod_{<\kappa}\A=\A$
and it is closed under extensions. In order to prove the inclusion
$\Rep(Q,\B)\subseteq\Phi(\A)^{\perp_{1}},$ we proceed as follows:
we use the transfinite sequence of subsets introduced in Remark \ref{rem:transfinite sequence and rooted quivers}.
Since $Q$ is left-rooted, there is an ordinal $\tau$ such that $Q_{0}=V_{\tau}$.
Recall that the sequence $\{V_{\gamma}\}_{\gamma\leq\tau}$ satisfies
that $V_{\gamma}\subseteq V_{\gamma'}$ for all $\gamma\leq\gamma'\leq\tau$
\cite[Lem. 2.7]{holm2019cotorsion}. It is also important to recall
that, for every $i,j\in Q_{0}$ and $\gamma\leq\tau$, we have that
if $i\notin V_{\gamma}$ and $j\in V_{\gamma+1}$, then there is no
arrow $i\rightarrow j$ \cite[Cor. 2.8]{holm2019cotorsion}.
In particular, if $Q_{\gamma}$ is the full subquiver generated by
$V_{\gamma}$, then we have that $Q_{\gamma}^{-}=\emptyset$ $\forall\gamma\leq\tau$.
Hence, by Lemma \ref{lem:par adjunto i p}, for every $\gamma\leq\gamma'\leq\tau$,
the functor 
\[
\iota_{\gamma}^{\gamma'}:=\pi_{Q_{\gamma'}}\circ\iota_{Q_{\gamma}}:\operatorname{Rep}(Q_{\gamma},\mathcal{C})\rightarrow\operatorname{Rep}(Q_{\gamma'},\mathcal{C})
\]
is right adjoint to the functor 
\[
\pi_{\gamma'}^{\gamma}:=\pi_{Q_{\gamma}}\circ\iota_{Q_{\gamma'}}:\operatorname{Rep}(Q_{\gamma'},\mathcal{C})\rightarrow\operatorname{Rep}(Q_{\gamma},\mathcal{C}).
\]
Let $Y\in\operatorname{Rep}(Q,\mathcal{C})$. For $\gamma\leq\tau$,
define $Y_{\gamma}:=\iota_{Q_{\gamma}}\pi_{Q_{\gamma}}(Y)$. Observe
that (this can be proved straightforward or by using that $\iota_{\gamma}^{\gamma'}$
is right adjoint to $\pi_{\gamma'}^{\gamma}$), for all $\gamma\leq\gamma'\leq\tau$,
there is an epimorphism $g_{\gamma'}^{\gamma}:Y_{\gamma'}\rightarrow Y_{\gamma}$
defined as 
\[
g_{\gamma'}^{\gamma}(i)=\begin{cases}
1_{Y(i)} & \text{ if }i\in V_{\gamma},\\
0 & \text{ if }i\notin V_{\gamma}.
\end{cases}
\]
We point out that these morphisms define an inverse co-continuous system.
That is, $g_{\gamma'}^{\gamma}g_{\gamma''}^{\gamma'}=g_{\gamma''}^{\gamma}$
for all $\gamma\leq\gamma'\leq\gamma''\leq\tau$, and $Y_{\gamma}=\underleftarrow{\lim}_{\gamma'\leq\gamma}Y_{\gamma'}$
for all limit ordinal $\gamma\leq\tau$. In particular, $Y=\underleftarrow{\lim}_{\gamma'\leq\tau}Y_{\gamma'}$.

Let us prove that, if $Y\in\operatorname{Rep}(Q,\mathcal{B})$, then
$Y\in\Phi(\mathcal{A})^{\bot_{1}}$. By the Eklof Lemma (see \cite[Lem. 3.7]{odabacsi2019completeness}),
it is enough to show that $K_{\gamma}:=\Ker\left(g_{\gamma+1}^{\gamma}\right)\in\Phi(\mathcal{A})^{\bot_{1}}$
for all $\gamma<\tau$. For this, note that $K_{\gamma}=\prod_{i\in V_{\gamma+1}-V_{\gamma}}s_{i}(Y(i))$
(see Remark \ref{rem:prod,corprod,s}); and thus, since $\Phi(\mathcal{A})^{\bot_{1}}$
is closed under products and $s_{*}(\mathcal{B})\subseteq\Phi(\mathcal{A})^{\bot_{1}}$
by Proposition \ref{prop:s,c,k ext}, $K_{\gamma}\in\Phi(\mathcal{A})^{\bot_{1}}$
$\forall\gamma\leq\tau$. Therefore, $Y\in\Phi(\mathcal{A})^{\bot_{1}};$
proving (c). \

Finally, the item (d) follows from (a), (b) and (c). 
\end{proof}
The dual of Theorem \ref{thm:cotrsion pairs rep} can be stated now
as follows.

\begin{thm}\label{thm:cotrsion pairs rep dual} For a quiver $Q,$ an infinite
cardinal number $\kappa,$ an AB3{*}($\kappa$) abelian category $\C$
and a cotorsion pair $\p$ in $\C,$ the following statements hold
true. 
\begin{enumerate}
\item Let $g_{i}:\C\to\Rep(Q,\C)$ be exact $\forall\,i\in Q_{0}$ and $\kappa\geq\lccn(Q).$
Then 
\begin{center}
$^{\perp_{1}}g_{*}(\B)=\Rep(Q,\A)\;$ and $\;g_{*}(\B)\subseteq s_{*}(\A)^{\perp_{1}}.$ 
\par\end{center}
\item Let $\A$ be a generating class in $\C$ and $\kappa\geq\rmcn(Q).$
Then 
\begin{center}
$\Psi(\B)=s_{*}(\A)^{\perp_{1}}.$ 
\par\end{center}
\item Let $Q$ be right rooted and $\kappa\geq\rmcn(Q).$ Then 
\begin{center}
$\Psi(\B)\subseteq\Rep(Q,\B)\;$ and $\;\Rep(Q,\A)\subseteq{}^{\perp_{1}}\Psi(\B).$ 
\par\end{center}
\item If $g_{i}:\C\to\Rep(Q,\C)$ is exact $\forall\,i\in Q_{0},$ $\kappa\geq\max\{\lccn(Q),\rmcn(Q)\},$
$Q$ is right rooted and $\A$ is a generating class in $\C,$ then
$\Rep(Q,\A)={}^{\perp_{1}}\Psi(\B).$ 
\end{enumerate}
\end{thm}

In what follows, we apply Theorems \ref{thm:cotrsion pairs rep} and
\ref{thm:cotrsion pairs rep dual} and see what happens if the quiver
$Q$ is: interval finite, locally finite or strongly locally finite.

\begin{cor}\label{Q-IF} For an interval finite quiver $Q,$ an abelian category
$\C,$ a cotorsion pair $\p$ in $\C$ and an infinite cardinal number
$\kappa,$ the following statements hold true. 
\begin{itemize}
\item[(a)] $(^{\bot_{1}}\operatorname{Rep}(Q,\mathcal{B}),\operatorname{Rep}(Q,\mathcal{B}))$
and $(\operatorname{Rep}(Q,\mathcal{A}),\operatorname{Rep}(Q,\mathcal{A})^{\bot_{1}})$
are cotorsion pairs in $\Rep(Q,\C).$ 
\item[(b)] Let $\C$ be AB3($\kappa$) and AB3{*}($\kappa$) with $\kappa\geq\mcn(Q).$
Then: 
\begin{itemize}
\item[(b1)] If $\A$ is generating in $\C,$ then $\left(^{\bot_{1}}\Psi(\mathcal{B}),\Psi(\mathcal{B})\right)$
is a cotorsion pair. If in addition $Q$ is right-rooted, then $\Rep(Q,\A)={}^{\bot_{1}}\Psi(\mathcal{B}).$ 
\item[(b2)] If $\B$ is cogenerating in $\C,$ then $\left(\Phi(\A),\Phi(\A)^{\bot_{1}}\right)$
is a cotorsion pair. If in addition $Q$ is left-rooted, then $\Rep(Q,\B)=\Phi(\A)^{\bot_{1}}.$ 
\end{itemize}
\end{itemize}
\end{cor}

\begin{proof}
Since $\ccn(Q)\leq\aleph_{0},$ the result follows from Remark \ref{AB3+Inj=AB4}
(c), Theorem \ref{thm:cotrsion pairs rep} and its dual. 
\end{proof}

\begin{cor}\label{Q-LF} For a locally finite quiver $Q,$ an abelian category
$\C,$ a cotorsion pair $\p$ in $\C$ and an infinite cardinal number
$\kappa,$ the following statements hold true. 
\begin{itemize}
\item[(a)] Let $\A$ be generating in $\C.$ Then 
\begin{itemize}
\item[(a1)] $(^{\bot_{1}}\Psi(\B),\Psi(\B))$ is a cotorsion pair in $\C;$ 
\item[(a2)] $\Rep(Q,\A)={}^{\bot_{1}}\Psi(\B)$ if $Q$ is right-rooted, $\C$
is AB3{*}($\kappa$), with $\kappa\geq\lccn(Q),$ and $g_{i}:\C\to\Rep(Q,\C)$
is exact $\forall\,i\in Q_{0}.$ 
\end{itemize}
\item[(b)] Let $\B$ be cogenerating in $\C.$ Then 
\begin{itemize}
\item[(a1)] $(\Phi(\A),\Phi(\A){}^{\bot_{1}})$ is a cotorsion pair in $\C;$ 
\item[(a2)] $\Rep(Q,\B)=\Phi(\A){}^{\bot_{1}}$ if $Q$ is left-rooted, $\C$
is AB3($\kappa$), with $\kappa\geq\rccn(Q),$ and $f_{i}:\C\to\Rep(Q,\C)$
is exact $\forall\,i\in Q_{0}.$ 
\end{itemize}
\end{itemize}
\end{cor}

\begin{proof}
Since $\mcn(Q)\leq\aleph_{0},$ the result follows from Theorem \ref{thm:cotrsion pairs rep}
and its dual. 
\end{proof}

\begin{cor}\label{Q-SLF} For a strongly locally finite quiver $Q,$ an abelian
category $\C,$ and a cotorsion pair $\p$ in $\C,$ the following
statements hold true. 
\begin{itemize}
\item[(a)] $(^{\bot_{1}}\operatorname{Rep}(Q,\mathcal{B}),\operatorname{Rep}(Q,\mathcal{B}))$
and $(\operatorname{Rep}(Q,\mathcal{A}),\operatorname{Rep}(Q,\mathcal{A})^{\bot_{1}})$
are cotorsion pairs in $\Rep(Q,\C).$ 
\item[(b)] If $\A$ is generating in $\C,$ then $\left(^{\bot_{1}}\Psi(\mathcal{B}),\Psi(\mathcal{B})\right)$
is a cotorsion pair. If in addition $Q$ is right-rooted, then $\Rep(Q,\A)={}^{\bot_{1}}\Psi(\mathcal{B}).$ 
\item[(c)] If $\B$ is cogenerating in $\C,$ then $\left(\Phi(\A),\Phi(\A)^{\bot_{1}}\right)$
is a cotorsion pair. If in addition $Q$ is left-rooted, then $\Rep(Q,\B)=\Phi(\A)^{\bot_{1}}.$ 
\end{itemize}
\end{cor}

\begin{proof}
It follows from Corollary \ref{Q-IF} since $\mcn(Q)\leq\aleph_{0}$
and $Q$ is interval finite. 
\end{proof}

\begin{rem} In general, $^{\bot_{1}}\Psi(\mathcal{B})\neq\operatorname{Rep}(Q,\mathcal{A})$
and $\Phi(\mathcal{A})^{\bot_{1}}\neq\operatorname{Rep}(Q,\mathcal{B})$.
Indeed, in Example \ref{exa: proyectivos/inyectivos}(d), we have
shown that $\Rep(Q,\mathcal{C})^{\bot_{1}}=0={}^{\bot_{1}}\Rep(Q,\mathcal{C})$
for $Q=\widetilde{A}_{n}$ and $\C$ the category of finite dimension
$k$-vector spaces. However, $\Psi(\C^{\bot_{1}})=\Psi(\C)\neq0$
and $\Phi({}^{\bot_{1}}\C)=\Phi(\C)\neq0$. 
\end{rem}

\subsection{Complete cotorsion pairs}

In this subsection we will show enough conditions for the cotorsion
pairs found above to be complete. The main result in this direction
in the following one.
\begin{thm}\cite[Thm. A]{di2021completeness} \label{Thm-completeness} For
a quiver $Q,$ an infinite cardinal number $\kappa$ and a complete
cotorsion pair $\p$ in $\mathcal{C},$ the following statements hold
true. 
\begin{enumerate}
\item If $Q$ is left-rooted and $\mathcal{C}$ is Ab3($\kappa$) for $\kappa\geq\lmcn(Q),$
then $(\Phi(\mathcal{A}),\Phi(\mathcal{A})^{\bot_{1}})$ is a complete
cotorsion pair in $\Rep(Q,\C).$ 
\item If $Q$ is right-rooted and $\mathcal{C}$ is Ab3{*}($\kappa$) for
$\kappa\geq\rmcn(Q),$ then $(^{\bot_{1}}\Psi(\mathcal{B}),\Psi(\mathcal{B}))$
is a complete cotorsion pair in $\Rep(Q,\C).$ 
\end{enumerate}
\end{thm}

\subsection{Hereditary cotorsion pairs}

In this subsection we will seek to find conditions for the cotorsion
pairs found above to be hereditary. We recall that a cotorsion pair
$\p$ in an abelian category $\C$ is hereditary if $\Ext_{\C}^{i}(A,B)=0$
$\forall\,i\geq1,$ $\forall\,A\in\A$ and $\forall\,B\in\B.$ Given
a class $\X\subseteq\C,$ we recall that $\X$ is closed under \textbf{epi-kernels}
if for any exact sequence $0\to K\to X_{1}\to X_{0}\to0,$ with $X_{0},X_{1}\in\X,$
we have that $K\in\X.$ Dually, we have the notion of being closed
under \textbf{mono-cokernels}. If $\C$ has enough projectives and
injectives, there is a result (known as the Garc\'ia-Rozas's Lemma,
see more details in \cite[Sect. 6.2]{holm2019cotorsion}) which
says that the following conditions are equivalent: (a) $\p$ is hereditary,
(b) $\A$ is closed under epi-kernels and (c) $\B$ is closed under
mono-cokernels. \

The following result can be seen as the version of Garc\'ia-Rozas's
Lemma for the case when $\C$ does not have enough projectives or
injectives.

\begin{lem}\label{lem:garcia-rozas} Let $\mathcal{C}$ be an abelian category
and $\p$ be a cotorsion pair in $\mathcal{C}$. Then, the following
statements hold true: 
\begin{enumerate}
\item If $^{\bot}\mathcal{B}$ is a generating class in $\mathcal{C}$ and
$\mathcal{A}$ is closed under epi-kernels, then $\p$ is hereditary. 
\item If $\mathcal{A}^{\bot}$ is a cogenerating class in $\mathcal{C}$
and $\mathcal{B}$ is closed under mono-cokernels, then $\p$ is hereditary. 
\item If $\p$ is hereditary, then $\mathcal{A}$ is closed under epi-kernels
and $\mathcal{B}$ is closed under mono-cokernels. 
\end{enumerate}
\end{lem}

\begin{rem}\label{Rk-her} Let $\C$ be an abelian category. 
\begin{itemize}
\item[(a)] If $\C$ has enough projectives, then $^{\perp}\X$ is generating
in $\C,$ for any $\X\subseteq\C.$ 
\item[(b)] If a hereditary cotorsion pair $\p$ in $\C$ is complete, then $\A={}^{\perp}\B$
and $\B=\A^{\perp}.$ 
\end{itemize}
\end{rem}

Our strategy to study the hereditary property in our cotorsion pairs
will be to exhibit that the conditions of the Garcia-Rozas's Lemma
are satisfied. We will start with the simple case where there are
enough projectives or injectives.

The following result and its dual is a generalization of \cite[Prop.  7.6]{holm2019cotorsion}.

\begin{prop}\label{hered-1} For a quiver $Q,$ an infinite cardinal number $\kappa,$
an abelian AB3($\kappa$) category $\C$ which has enough projectives,
with $\kappa\geq\max\{\rccn(Q),\ltccn(Q)\},$ and a cotorsion pair
$\p$ in $\C$ such that $\A$ is closed under epi-kernels, the following
statements hold true. 
\begin{itemize}
\item[(a)] If $\B$ is cogenerating in $\C,$ then $\left(\Phi(\mathcal{A}),\Phi(\mathcal{A})^{\bot_{1}}\right)$
is a hereditary cotorsion pair in $\C.$ 
\item[(b)] Let $\C$ be AB3{*}($\kappa$). Then: 
\begin{itemize}
\item[(b1)] $(\operatorname{Rep}(Q,\mathcal{A}),\operatorname{Rep}(Q,\mathcal{A})^{\bot_{1}})$
is a hereditary cotorsion pair in $\C.$ 
\item[(b2)] $\operatorname{Rep}(Q,\mathcal{A})={}^{\perp_{1}}\Psi(\B)$ if $\kappa\geq\rmcn(Q)$
and $Q$ is right-rooted. 
\end{itemize}
\end{itemize}
\end{prop}

\begin{proof}
Observe that, from Proposition \ref{prop:f y g vs clases ortogonales}
(b), we have that $\Rep(Q,\C)$ has enough projectives. 
\

 (a) Let $\B$ be cogenerating in $\C.$ Since $\lmcn(Q)\leq\ltccn(Q),$
we get from Theorem \ref{thm:cotrsion pairs rep} (b), that $\left(\Phi(\mathcal{A}),\Phi(\mathcal{A})^{\bot_{1}}\right)$
is a cotorsion pair in $\C.$ Moreover, by following the proof of
\cite[Prop.  7.6]{holm2019cotorsion}, it follows that $\Phi(\mathcal{A})$
is closed under epi-kernels. Thus, by Lemma \ref{lem:garcia-rozas}
(a), we get that such pair is hereditary since $\Rep(Q,\C)$ has enough
projectives. 
\

 (b) By the dual of Remark \ref{AB3+Inj=AB4} (b), it follows
that $\C$ is AB4{*}($\kappa$). Moreover, using that $\lccn(Q)\leq\ltccn(Q),$
we get form Remark \ref{AB3+Inj=AB4} (a) that $g_{i}:\C\to\Rep(Q,\C)$
is exact $\forall\,i\in Q_{0}.$ Thus, from Theorem \ref{thm:cotrsion pairs rep dual}
(a) and using that $\lccn(Q)\leq\ltccn(Q),$ we conclude that $(\operatorname{Rep}(Q,\mathcal{A}),\operatorname{Rep}(Q,\mathcal{A})^{\bot_{1}})$
is a cotorsion pair in $\C.$ The fact that this pair is also hereditary
follows as in the proof of (a) since it is clear that $\operatorname{Rep}(Q,\mathcal{A})$
is closed under epi-kernels; proving (b1).
Let us show (b2). Assume that $\kappa\geq\rmcn(Q)$ and $Q$ is right-rooted.
Observe that $\A$ is generating in $\C$ since $\C$ has enough projectives.
Thus, by Theorem \ref{thm:cotrsion pairs rep dual} (d), we get (b2)
since $\lccn(Q)\leq\ltccn(Q).$ 
\end{proof}

In case we do not have enough injectives or projectives, the following
result and its dual is a generalization of \cite[Prop.  7.6]{holm2019cotorsion}.

\begin{prop}\label{hered-2} For a quiver $Q,$ $\kappa\geq\max\{\rccn(Q),\ltccn(Q)\}$
an infinite cardinal number, an abelian AB4($\kappa$) category $\C,$
and a cotorsion pair $\p$ in $\C$ such that $\A$ is closed under
epi-kernels, $\B$ is cogenerating in $\C$, and $^{\perp}\B$ is
generating in $\C,$ the following statements hold true. 
\begin{itemize}
\item[(a)] $\left(\Phi(\mathcal{A}),\Phi(\mathcal{A})^{\bot_{1}}\right)$ is
a hereditary cotorsion pair in $\C.$ 
\item[(b)] $\Rep(Q,\B)=\Phi(\mathcal{A})^{\bot_{1}}$ if $Q$ is left-rooted. 
\end{itemize}
\end{prop}

\begin{proof}
(a) As in the proof of Proposition \ref{hered-1} (a), we get that
$\left(\Phi(\mathcal{A}),\Phi(\mathcal{A})^{\bot_{1}}\right)$ is
a cotorsion pair in $\C$ and $\Phi(\mathcal{A})$ is closed under
epi-kernels. Thus, by Lemma \ref{lem:garcia-rozas} (a), it is enough
to show that $^{\perp}(\Phi(\mathcal{A})^{\bot_{1}})$ is generating
in $\C.$

We assert that $^{\perp}\Rep(Q,\B)$ is generating in $\C.$ Indeed,
by Proposition \ref{prop:f y g vs clases ortogonales} (a), we get
that $\coprod_{\leq|Q_{0}|}f_{*}({}^{\perp}\B)$ is generating in
$\Rep(Q,\C).$ On the other hand, since $f_{i}:\C\to\Rep(Q,\C)$ is
exact $\forall\,i\in Q_{0}$ (see Remark \ref{AB3+Inj=AB4}
(a)), it follows from Proposition \ref{prop:f y clases ortogonales}
(b) that $f_{*}({}^{\perp}\B)\subseteq{}^{\perp}\Rep(Q,\B).$ Therefore
$\coprod_{\leq|Q_{0}|}f_{*}({}^{\perp}\B)\subseteq{}^{\perp}\Rep(Q,\B)$
since $\C$ is AB4($\kappa$) (see Lemma \ref{pd-coprod});
proving that $^{\perp}\Rep(Q,\B)$ is generating in $\C.$ Finally,
from Theorem \ref{thm:cotrsion pairs rep} (a,b), we get that $^{\perp}\Rep(Q,\B)\subseteq{}^{\perp}(\Phi(\mathcal{A})^{\bot_{1}})$
and thus $^{\perp}(\Phi(\mathcal{A})^{\bot_{1}})$ is generating in
$\C.$ 

(b) It follows from Theorem \ref{thm:cotrsion pairs rep} (d). 
\end{proof}

\begin{thm}\label{hccp+ep} For a quiver $Q,$ $\kappa\geq\max\{\rccn(Q),\ltccn(Q)\}$
an infinite cardinal number, an abelian AB3($\kappa$) category $\C,$
a complete hereditary cotorsion pair $\p$ in $\C,$ and $\mathfrak{C}(\A):=\left(\Phi(\mathcal{A}),\Phi(\mathcal{A})^{\bot_{1}}\right),$
the following statements hold true. 
\begin{itemize}
\item[(a)] Let $\C$ be with enough projectives. Then $\mathfrak{C}(\A)$ is
a hereditary cotorsion pair. Moreover $\mathfrak{C}(\A)$ is complete
if $Q$ is left-rooted. 
\item[(b)] Let $\C$ be AB4($\kappa$). Then $\mathfrak{C}(\A)$ is a hereditary
cotorsion pair. Moreover $\mathfrak{C}(\A)$ is complete and $\Rep(Q,\B)=\Phi(\mathcal{A})^{\bot_{1}}$
if $Q$ is left-rooted. 
\end{itemize}
\end{thm}

\begin{proof}
It follows from Remark \ref{lem:garcia-rozas} (c), Theorem \ref{Thm-completeness}
(a), Proposition \ref{hered-1} (a) and Proposition \ref{hered-2}. 
\end{proof}

We close this section with a description of the cotorsion pairs 
$(\Rep(Q,\A),\Psi(\B))$ and $(\Phi(\A),\Rep(Q,\B))$.
 It is worth mentioning that this description will be used in \cite{AM22} to study the 
 tilting cotorsion pairs in $\Rep(Q,\C )$.

\begin{lem}\label{lem:canonic presentation for s}  For a quiver $Q,$ a vertex $x\in Q_{0}$,
an infinite cardinal $\kappa$  and an abelian category $\C,$ the following statements hold true. 
\begin{enumerate}
\item If $\kappa\geq\max\{\rccn(Q),\ltccn(Q)\}$ and $\C$ is $AB3(\kappa),$
then there is an exact sequence $\suc[\coprod_{\rho\in Q_{1}^{x\rightarrow*}}f_{t(\rho)}][f_{x}][s_{x}]$
in $\Fun(\C,\Rep(Q,\C))$. 
\item If $\kappa\geq\max\{\lccn(Q),\rtccn(Q)\}$ and $\C$ is $AB3^{*}(\kappa)$,
then there is an exact sequence $\suc[s_{x}][g_{x}][\prod_{\rho\in Q_{1}^{*\rightarrow x}}g_{s(\rho)}]$
in $\Fun(\C,\Rep(Q,\C))$. 
\end{enumerate}
\end{lem}
\begin{proof}
By applying the exact functor $\left(s_{x}\right)^{*}:\End(\Rep(Q,\C))\rightarrow\Fun(\C,\Rep(Q,\C))$, defined as 
$(H_{1}\overset{b}{\rightarrow}H_{2})\mapsto(H_{1}\circ s_{x}\overset{b\cdot s_{x}}{\rightarrow}H_{2}\circ s_{x})$,
to the exact sequences in Corollaries \ref{cor:presentacion canonica} and \ref{cor:copresentacion canonica}, we get the desired exact sequences. 
\end{proof}

We can use the lemma above to build $s_{*}$-filtrations and $s_{*}$-cofiltrations
for $f_{x}$ and $g_{x}$ respectively for all $x\in Q_{0}$, where
$s_{*}=\{s_{i}\}_{i\in Q_{0}}$.

\begin{lem}\label{lem:filtraciones} For a quiver $Q,$ a vertex $x\in Q_{0}$,
an infinite cardinal $\kappa$  and an abelian category $\C,$ the following statements hold true. 
\begin{enumerate}
\item Let $\kappa\geq\max\{\rccn(Q),\ltccn(Q)\}$, $Q(x,-)$ be finite and
$\C$ be $AB3(\kappa).$ Then there is a sequence of short exact sequences
\[
\left\{ \eta_{x}^{k}:\;\suc[\coprod_{\rho\in Q_{k}^{x\rightarrow*}}f_{t(\rho)}][\coprod_{\rho\in Q_{k-1}^{x\rightarrow*}}f_{t(\rho)}][\coprod_{\rho\in Q_{k-1}^{x\rightarrow*}}s_{t(\rho)}][q_{k}]\right\} _{k\geq1},
\]
where $Q_{k}^{x\rightarrow*}$ is the set of paths of length $k$
beginning at $x$.

\item Let $\kappa\geq\max\{\lccn(Q),\rtccn(Q)\}$, $Q(-,x)$be finite and
$\C$ be $AB3^{*}(\kappa).$ Then there is a sequence of short exact
sequences 
\[
\left\{ \epsilon_{x}^{k}:\;\suc[\prod_{\rho\in Q_{k-1}^{*\rightarrow x}}s_{s(\rho)}][\prod_{\rho\in Q_{k-1}^{*\rightarrow x}}g_{s(\rho)}][\prod_{\rho\in Q_{k}^{*\rightarrow x}}g_{s(\rho)}][][p_{k}]\right\} _{k\geq1},
\]
where $Q_{k}^{*\rightarrow x}$ is the set of paths of length $k$
ending at $x$.
\end{enumerate}
\end{lem}
\begin{proof}
It follows by applying Lemma \ref{lem:canonic presentation for s}
recursively. 
\end{proof}
\begin{lem}\label{lem:Lemita par de cotorsion s g}  For an abelian category $\C$ and 
 the clases of objects $\mathcal{S},\mathcal{G}\subseteq\mathcal{C},$
 the following statements hold true. 
\begin{enumerate}
\item If every $G\in\mathcal{G}$ admits a  sequence $G=G_{0}\overset{p_{1}}{\twoheadrightarrow}G_{1}\overset{p_{2}}{\twoheadrightarrow}\cdots\overset{p_{s}}{\twoheadrightarrow}G_{s}=0$
of epimorphisms 
 such that $G_{t}\in\add(\mathcal{G})$ and $\Ker(p_{t})\in\add(\mathcal{S})$
for all $t\in\{1,\cdots,s\}$, then $\mathcal{S}^{\bot}\subseteq\mathcal{G}^{\bot}$. 
\item If every $S\in\mathcal{S}$ admits a short exact sequence $\suc[S][G][M]$
with $G,M\in\add(\mathcal{G})$, then $\mathcal{G}^{\bot}\subseteq\mathcal{S}^{\bot}$.
\end{enumerate}
\end{lem}
\begin{proof} The proof is straightforward and it is left to the reader.
\end{proof}
\begin{prop}\label{prop:ort f vs s vs g}  For a quiver $Q,$ an infinite cardinal $\kappa,$  an abelian category $\C$ and a class of objects $\mathcal{W}\subseteq\mathcal{C},$
	 the following statements hold true. 
\begin{enumerate}
\item Let $Q(x,-)$ be finite for all $x\in Q_{0}$, $\kappa\geq \ltccn(Q)$
and $\C$ be $AB3(\kappa).$  Then $^{\bot}s_{*}(\mathcal{W})={}^{\bot}f_{*}(\mathcal{W})$. 
\item Let $Q(-,x)$ be finite for all $x\in Q_{0}$, $\kappa\geq \rtccn(Q)$
and $\C$ be $AB3^*(\kappa).$  Then $s_{*}(\mathcal{W})^{\bot}=g_{*}(\mathcal{W})^{\bot}$. 
\end{enumerate}
\end{prop}
\begin{proof}
Let us prove (b). Consider $x\in Q_{0}$ and $W\in\mathcal{W}$. Since
$Q(-,x)$ is finite, there is $n\geq0$ such that $n$ is the maximum
length of the paths ending at $x$. In particular, $\prod_{\rho\in Q_{n}^{*\rightarrow x}}g_{s(\rho)}=\prod_{\rho\in Q_{n}^{*\rightarrow x}}s_{s(\rho)}$.
Therefore, by Lemma \ref{lem:filtraciones}(b), we have a sequence
of epimorphisms
\[
g_{x}(W)\overset{p_{1}}{\twoheadrightarrow}\prod_{\rho\in Q_{1}^{*\rightarrow x}}g_{s(\rho)}(W)\overset{p_{2}}{\twoheadrightarrow}\cdots\overset{p_{n}}{\twoheadrightarrow}\prod_{\rho\in Q_{n}^{*\rightarrow x}}g_{s(\rho)}(W)\overset{p_{n+1}}{\twoheadrightarrow}0,
\]
such that $\prod_{\rho\in Q_{k}^{*\rightarrow x}}g_{s(\rho)}(W)\in\add(g_{*}(\mathcal{W}))$
and $\Ker(p_{k})\in\add(s_{*}(\mathcal{W}))$ for all $t\in\{1,\cdots,n+1\}$.
Moreover, by Lemma \ref{lem:canonic presentation for s}(b), for every
$S\in s_{*}(\mathcal{W})$ there is a short exact sequence $\suc[S][G][M]$
with $G,M\in\add(g_{*}(\mathcal{W})).$ Therefore, by Lemma \ref{lem:Lemita par de cotorsion s g},
$s_{*}(\mathcal{W})^{\bot}=g_{*}(\mathcal{W})^{\bot}$. 
\end{proof}

\begin{prop}\label{Descrip-Phi-Psi}  For a quiver $Q,$ an infinite cardinal  $\kappa,$ an abelian category $\C$ and  a complete hereditary cotorsion
pair $\p$ in $\C,$  the following statements hold true. 
\begin{enumerate}
\item Let $Q(x,-)$ be finite for all $x\in Q_{0}$, $\kappa\geq\ltccn(Q),$ $Q$ be left rooted
and $\C$ be $AB4(\kappa).$  Then $^{\bot}s_{*}(\mathcal{B})={}{}^{\bot}f_{*}(\mathcal{B})=\Phi(\mathcal{A})$. 

\item Let $Q(-,x)$ be finite for all $x\in Q_{0}$, $\kappa\geq\rtccn(Q),$ $Q$ be right rooted
and $\C$ be $AB4^*(\kappa).$  Then $s_{*}(\mathcal{A})^{\bot}=g_{*}(\mathcal{A})^{\bot}=\Psi(\mathcal{B})$. 
\end{enumerate}
\end{prop}
\begin{proof}
Let us prove (b). By the dual of Theorem \ref{hccp+ep} (b),
$(\Rep(Q,\A),\Psi(\B))$ is a complete hereditary cotorsion pair in
$\Rep(Q,\C)$. Now, observe the following equalities and contentions
\[s_{*}(\A){}^{\bot}\subseteq s_{*}(\A){}^{\bot_{1}}\overset{(1)}{=}\Psi(\A^{\bot_1})=\Psi(\mathcal{B})=\Rep(Q,\A)^{\bot}\overset{(2)}{\subseteq}s_{*}(\mathcal{A})^{\bot}\text{,}
\]
where (1) follows from Proposition \ref{prop:f y clases ortogonales} (c) since 
 $\mathcal{A}$ is a generating class, and (2) follows from $s_{*}(\mathcal{A})\subseteq\Rep(Q,\A).$
We have proved that $s_{*}(\mathcal{A})^{\bot}=\Psi(\mathcal{B})$.
Hence, by Proposition \ref{prop:ort f vs s vs g} (b), we have that
$g_{*}(\mathcal{A})^{\bot}=\Psi(\mathcal{B})$. 
\end{proof}

\section{Subcategories of quiver representations}

Let $\C$ be an abelian category, $Q$ be a quiver and $F\in\operatorname{Rep}(Q,\mathcal{C}).$
We recall that the \textbf{support} of $F$ is the set $\operatorname{Supp}(F):=\{i\in Q_{0}\,|\:F(i)\neq0\}.$
If $\operatorname{Supp}(F)\in\F(Q)$ we say that $F$ is of (or has)
\textbf{finite-support}. In case $\operatorname{Supp}(F)\in\B(Q),$
we say that $F$ is of (or has) \textbf{bottom-support}. Finally,
$F$ is of \textbf{finite-bottom-support }if $\operatorname{Supp}(F)\in\F\B(Q).$
The corresponding terminology is introduced in the case that $\operatorname{Supp}(F)\in\T(Q)$
or $\operatorname{Supp}(F)\in\F\T(Q).$

\begin{defn} Let $\C$ be an abelian category, $Q$ be a quiver and $\Z$ be a
set of full subquivers of $Q.$ We consider the class 
\[
\operatorname{Rep}_{\Z}(Q,\mathcal{C}):=\{F\in\operatorname{Rep}(Q,\mathcal{C})\;:\;\operatorname{Supp}(F)\in\Z\}\cup\{0\}.
\]
We have the classes: $\operatorname{Rep}^{f}(Q,\mathcal{C}):=\operatorname{Rep}_{\F(Q)}(Q,\mathcal{C}),$
$\operatorname{Rep}^{fb}(Q,\mathcal{C}):=\operatorname{Rep}_{\F\B(Q)}(Q,\mathcal{C}),$
$\operatorname{Rep}^{ft}(Q,\mathcal{C}):=\operatorname{Rep}_{\F\T(Q)}(Q,\mathcal{C})$
and $\operatorname{Rep}^{fbt}(Q,\mathcal{C}):=\operatorname{Rep}_{\F\B\T(Q)}(Q,\mathcal{C}).$
\
 We say that the set $\Z$ is \textbf{directed} if $\forall\,Z_{1},Z_{2}\in\Z$
$\exists\,Z\in\Z$ such that $Z_{1}\cup Z_{2}\subseteq Z.$ The set
$\Z$ is called \textbf{convex} if $\forall\,Z\in\Z$ any non-empty
full subquiver $Z'\subseteq Z$ belongs to $\Z.$ 
\end{defn}

Notice that full subquivers are totally determined by their set of vertices. For
this reason, in general we will be identifying full subquivers with
their sets of vertices. In particular, we will be thinking of a family
$\Z$ of full subquivers as a family of subsets of vertices. For example,
the set $\Z$ is convex if for all $Z\in\Z$ any non-empty subset $Z'\subseteq Z$
belongs to $\Z$.

Let $\mathcal{C}$ be an abelian category and $Q$ be a quiver. For
every subquiver $Q'\subseteq Q$, we recall that we have the extension
functor $\iota_{Q'}:\operatorname{Rep}(Q',\mathcal{C})\rightarrow\operatorname{Rep}(Q,\mathcal{C})$
and the restriction functor $\pi_{Q'}:\operatorname{Rep}(Q,\mathcal{C})\rightarrow\operatorname{Rep}(Q',\mathcal{C}).$
Note that $\pi_{Q'}\iota_{Q'}=1_{\operatorname{Rep}(Q',\mathcal{C})}.$
We also consider 
\[
\Ima(\iota_{Q'}):=\{F\in\operatorname{Rep}(Q,\mathcal{C})\;:\;\exists\,F'\in\operatorname{Rep}(Q',\mathcal{C})\text{ with }F=\iota_{Q'}(F')\}.
\]
Furthermore, for a set $\Z$ of subquivers of $Q,$ we set $\Ima(\iota_{\Z}):=\bigcup_{Z\in\Z}\,\Ima(\iota_{Z}).$

\begin{rem}\label{supp-ex-seq} For an abelian category $\C$ and a quiver $Q,$
the following statements hold true. 
\begin{itemize}
\item[$\mathrm{(a)}$] For any exact sequence $\suc[A][B][C]$ in $\operatorname{Rep}(Q,\mathcal{C}),$
it follows that $\operatorname{Supp}(B)=\operatorname{Supp}(A)\cup\operatorname{Supp}(C).$ 
\item[$\mathrm{(b)}$] Let $S$ be a full subquiver of $Q$ and $M\in\operatorname{Rep}(Q,\mathcal{C}).$
If $\operatorname{Supp}(M)\subseteq S,$ then $M=\iota_{S}(\pi_{S}(M))\in\Ima(\iota_{S}).$ 
\end{itemize}
\end{rem}

\begin{prop}\label{Main-prop} For an abelian category $\C,$ a quiver $Q$ and
a set $\Z$ of full subquivers of $Q,$ the following statements hold
true. 
\begin{itemize}
\item[$\mathrm{(a)}$] $\operatorname{Rep}_{\Z}(Q,\mathcal{C})\subseteq\Ima(\iota_{\Z}).$ 

\item[$\mathrm{(b)}$] $\Ima(\iota_{\Z})=\operatorname{Rep}_{\Z}(Q,\mathcal{C})$ if $\Z$
is a convex set.

\item[$\mathrm{(c)}$] $\Ima(\iota_{\Z})$ is closed under direct summands, subobjects and
quotients in the abelian category $\operatorname{Rep}(Q,\mathcal{C}).$
In particular, $\Ima(\iota_{\Z})$ is an abelian subcategory of $\operatorname{Rep}(Q,\mathcal{C}).$ 

\item[$\mathrm{(d)}$] If $\Z$ is directed, then $\Ima(\iota_{\Z})$ is closed under extensions
in $\operatorname{Rep}(Q,\mathcal{C}).$ 
\end{itemize}
\end{prop}

\begin{proof}
(a) It follows from Remark \ref{supp-ex-seq} (b). 

(b) Let $\Z$ be a convex set. From (a), it is enough to show that
$\Ima(\iota_{\Z})\subseteq\operatorname{Rep}_{\Z}(Q,\mathcal{C}).$
Let $0\neq G\in\Ima(\iota_{\Z}).$ Then, there is some $Z\in\Z$ and
$F\in\operatorname{Rep}(Z,\mathcal{C})$ such that $G=\iota_{Z}(F).$
Therefore $\emptyset\neq\operatorname{Supp}(G)=\operatorname{Supp}(F)\subseteq Z$
and thus $\operatorname{Supp}(G)\in\Z;$ proving that $\Ima(\iota_{\Z})\subseteq\operatorname{Rep}_{\Z}(Q,\mathcal{C}).$ 

(c) Let $\alpha:B\to C$ be a morphism in $\operatorname{Rep}(Q,\mathcal{C})$
with $B\in\Ima(\iota_{\Z}).$ In particular, $B=\iota_{Z}(F)$ for
some $Z\in\Z$ and $F\in\operatorname{Rep}(Z,\mathcal{C}).$ We also
have that $\operatorname{Supp}(B)=\operatorname{Supp}(F)\subseteq Z.$
\
 Let us show that $K:=\Ker(\alpha)\in\Ima(\iota_{Z}).$ Indeed, by
Remark \ref{supp-ex-seq} (a), we get that $\operatorname{Supp}(K)\subseteq\operatorname{Supp}(B)\subseteq Z.$
Thus, by Remark \ref{supp-ex-seq} (b), we conclude that $K\in\Ima(\iota_{Z}).$
Similarly, it can be shown that $\Ima(\alpha)\in\Ima(\iota_{\Z})$
and also that $\Ima(\iota_{\Z})$ is closed under direct summands
in $\operatorname{Rep}(Q,\mathcal{C}).$ 

(d) Let $\Z$ be a directed set. Consider an exact sequence $\suc[A][B][C]$
in $\operatorname{Rep}(Q,\mathcal{C}),$ with $A,C\in\Ima(\iota_{\Z}).$
In particular, there are $Z_{1},Z_{2}\in\Z,$ $F\in\operatorname{Rep}(Z_{1},\mathcal{C})$
and $G\in\operatorname{Rep}(Z_{2},\mathcal{C})$ such that $A=\iota_{Z_{1}}(F)$
and $C=\iota_{Z_{2}}(G).$ Since $\Z$ is directed, there is some
$Z\in\Z$ such that $Z_{1}\cup Z_{2}\subseteq Z.$ Therefore, we have
that $\operatorname{Supp}(B)=\operatorname{Supp}(A)\cup\operatorname{Supp}(C)\subseteq Z\in\Z$
and thus, by Remark \ref{supp-ex-seq} (b), we get that $B\in\Ima(\iota_{\Z}).$ 
\end{proof}

\begin{rem}\label{MainF-FB-B} Let $\C$ be an abelian category and $Q$ be a
quiver. Notice that, by Remark \ref{main-quivers}, the sets $\F(Q),$
$\F\B(Q),$ $\F\T(Q)$ and $\F\B\T(Q)$ are directed and convex. However,
the sets $\B(Q)$ and $\T(Q)$ are directed but not necessarily convex.
Thus, by Proposition \ref{Main-prop} we have that $\operatorname{Rep}^{f}(Q,\mathcal{C})=\Ima(\iota_{\F(Q)}),$
$\operatorname{Rep}^{fb}(Q,\mathcal{C})=\Ima(\iota_{\F\B(Q)}),$ $\operatorname{Rep}^{ft}(Q,\mathcal{C})=\Ima(\iota_{\F\T(Q)})$
and $\operatorname{Rep}^{fbt}(Q,\mathcal{C})=\Ima(\iota_{\F\B\T(Q)});$
and also that these classes together with $\Ima(\iota_{\B(Q)})$ and
$\Ima(\iota_{\T(Q)})$ are closed under: extensions, direct summands,
sub-objects and quotients in the abelian category $\operatorname{Rep}(Q,\mathcal{C}).$
In particular, all these classes are full abelian subcategories of
$\operatorname{Rep}(Q,\mathcal{C}).$ 
\end{rem}

With the machinery we have developed in the previous sections we can
prove the following properties of $\Rep^{ft}(Q,\C)$. In order to do that, we start with the following lemma whose proof is left to the reader (the item (c) follows from Remark \ref{main-quivers} (b)).

\begin{lem} \label{previo-Repft} For a quiver $Q,$ the following statements hold true.
\begin{itemize}
\item[(a)] If there is an arrow $x\xrightarrow{\rho}z$ in $Q,$ then $\{z\}^+\subseteq 
\{x\}^+\cup\{x\}.$
\item[(b)] Let $i\in Q_0,$ $\kappa\geq\rccn_i(Q)$ be an infinite cardinal, $\C$ be an $AB3(\kappa)$ abelian category and $0\neq C\in \C.$ Then, for $f_i(C)\in\Rep(Q,\C),$ we have that $\Supp(f_i(C))^+=\emptyset$ and 
$\Supp(f_i(C))=t(Q(i,-))=\{i\}^+\cup\{i\}.$
\item[(c)] Let $\C$ be an abelian category and $F\in \Rep(Q,\C).$ Then, for any $x\in\Supp(F),$ we have $\{x\}^+\subseteq \Supp(F)\cup \Supp(F)^+.$ 
\end{itemize}
\end{lem}

\begin{prop}\label{prop:Repft} For a quiver $Q,$ $TQ_0:=\{i\in Q_0\;:\; \{i\}\in\T(Q)\},$ an infinite cardinal $\kappa\geq\rccn(Q)$  
and an abelian $AB3(\kappa)$  category  $\C,$ the following statements hold true.
\begin{itemize}
\item[(a)] $f_{i}(X)\in\Rep^{ft}(Q,\C)$ for all $i\in TQ_{0}$ and $X\in C$.

\item[(b)] Let $\kappa\geq \ltccn(Q).$ Then, for any $F\in\Rep^{ft}(Q,\C)$,  the canonical presentation
$
\eta_{F}:\quad\suc[\coprod_{\rho\in Q_{1}}f_{t(\rho)}e_{s(\rho)}(F)][\coprod_{i\in Q_{0}}f_{i}e_{i}(F)][F]
$
is in $\Rep^{ft}(Q,\C)$.  
\item[(c)] $f_{i}:\C\rightarrow\Rep^{ft}(Q,\C)$ is left adjoint for $e_{i}:\Rep^{ft}(Q,\C)\rightarrow\C$
$\forall i\in TQ_{0}.$
\item[(d)] Let $f_i$ be exact, for all $i\in TQ_{0}.$ Then, for each $i\in TQ_{0},$ we have the natural isomorphisms 
$
\operatorname{Ext}_{\operatorname{Rep}^{ft}(Q,\mathcal{C})}^{n}(f_{i}(?),-)\cong\operatorname{Ext}_{\mathcal{C}}^{n}(?,e_{i}(-))\;\forall\,n\geq1.
$
\end{itemize}
\end{prop}

\begin{proof} (a) It follows from Lemma \ref{previo-Repft} (b). 
\

(b) Consider $F\in\Rep^{ft}(Q,\C).$ By Corollary \ref{cor:presentacion canonica}, the exact sequence $\eta_{F}$ exists in $\Rep(Q,\C).$ Then, we have that 
$\coprod_{i\in Q_{0}}f_{i}e_{i}(F)=\coprod_{i\in\Supp(F)}f_i(F_i)$ and thus by (a) and  Remark \ref{MainF-FB-B}, we get that $\coprod_{i\in Q_{0}}f_{i}e_{i}(F)\in \Rep^{ft}(Q,\C)$ since $\Supp(F)$ is finite and $\Supp(F)\subseteq TQ_0$ (see Lemma \ref{previo-Repft} (c)). 

Now, for $x\in \Supp(F)\subseteq TQ_0$ and $\rho:x\to z$ in $Q_1,$ let 
$S_{x,z}:=\Supp(f_{t(\rho)}(F_x))=\Supp(f_z(F_x)).$ Then, by  Lemma \ref{previo-Repft} (a,b), we get that $S_{x,z}\subseteq \{x\}^+\cup\{x,z\}$ and 
$S_{x,z}^+=\emptyset.$  Moreover, using that 
\begin{center}
$\coprod_{\rho\in Q_{1}}f_{t(\rho)}e_{s(\rho)}(F)=\coprod_{x\in\Supp(F)}\coprod_{z\in t(Q(x,-))} \coprod_{\rho\in Q_1(x,z)} f_{t(\rho)}(F_x)$
\end{center}
and $S_x:=\Supp(\coprod_{z\,:\;Q(x,z)\neq \emptyset} \coprod_{\rho\in Q_1(x,z)} f_{t(\rho)}(F_x))=\bigcup_{z\in t(Q(x,-))}S_{x,z},$ we get by Remark \ref{main-quivers} (a) that $S_x^+\subseteq \bigcup_{z\in t(Q(x,-))}S_{x,z}^+=\emptyset.$ Furthermore 
\begin{center}
$\bigcup_{z\in t(Q(x,-))}S_{x,z}\subseteq \bigcup_{z\in t(Q(x,-))} (\{x\}^+\cup\{x,z\})\subseteq \{x\}^+\cup\{x\}\bigcup t(Q(x,-)).$
\end{center}
However, from Lemma \ref{previo-Repft} (b), $t(Q(x,-))=\{x\}^+\cup\{x\}$ 
and thus $S_x\subseteq \{x\}^+\cup\{x\}$ which is finite since $x\in \Supp(F)\subseteq TQ_0.$ Hence 
\begin{center}
$S:=\Supp(\coprod_{\rho\in Q_{1}}f_{t(\rho)}e_{s(\rho)}(F))=\bigcup_{x\in \Supp(F)}S_x$ 
\end{center}
is finite and $S^+\subseteq \bigcup_{x\in \Supp(F)}S_x^+=\emptyset;$ proving that $\coprod_{\rho\in Q_{1}}f_{t(\rho)}(F_{s(\rho)})$ belongs to $\Rep^{ft}(Q,\C).$ 
\

(c) It follows from (a) and Proposition \ref{rem:adjuntos de ei} (a) since, for 
$F\in  \Rep^{ft}(Q,\C)$ and $C\in\C,$ we have that 
$
\Hom_{\Rep^{ft}(Q,\C)}(f_{i}(C),F)=\Hom_{\Rep(Q,\C)}(f_{i}(C),F).$
\

(d) It follows from (c) and Proposition \ref{lem:Ext vs adjoint 3}.
\end{proof}

Similarly, as in the proof of Proposition \ref{prop:Repft}, it can be shown the following result.

\begin{prop}\label{prop:Repft-dual} For a quiver $Q,$ $BQ_0:=\{i\in Q_0\;:\; \{i\}\in\B(Q)\},$ an infinite cardinal $\kappa\geq\lccn(Q)$  
and an abelian $AB3^*(\kappa)$  category  $\C,$ the following statements hold true.
\begin{itemize}
\item[(a)] $g_{i}(X)\in\Rep^{fb}(Q,\C)$ for all $i\in BQ_{0}$ and $X\in C$.

\item[(b)] Let $\kappa\geq \rtccn(Q).$ Then, for any $F\in\Rep^{fb}(Q,\C)$,  the canonical co-presentation
{\small{$\eta_{F}:\;F\hookrightarrow \prod_{i\in Q_{0}}g_{i}e_{i}(F)\twoheadrightarrow \prod_{\rho\in Q_{1}}g_{s(\rho)}e_{t(\rho)}(F)$}}
is in {\small{$\Rep^{fb}(Q,\C)$.}}  
\item[(c)] $g_{i}:\C\rightarrow\Rep^{ft}(Q,\C)$ is right adjoint for $e_{i}:\Rep^{fb}(Q,\C)\rightarrow\C$
$\forall i\in TQ_{0}.$
\item[(d)] Let $g_i$ be exact, for all $i\in BQ_{0}.$ Then, for each $i\in BQ_{0},$ we have the natural isomorphisms 
$
\operatorname{Ext}_{\operatorname{Rep}^{fb}(Q,\mathcal{C})}^{n}(-, g_{i}(?))\cong\operatorname{Ext}_{\mathcal{C}}^{n}(e_{i}(-),?)\;\forall\,n\geq1.
$
\end{itemize}
\end{prop}

\begin{defn} For a quiver $Q,$ we introduce the cardinal numbers: (a) the right support cardinal number $\rscn(Q):=\mathrm{size}(\{|t(Q(x,-))|\}_{x\in Q_0}),$ (b) the left support cardinal number $\lscn(Q):=\mathrm{size}(\{|s(Q(-,x))|\}_{x\in Q_0})$ and (c) the arrow cardinal number $\alpha(Q):=\mathrm{size}(|Q_1(x,y)|\}_{x,y\in Q_0}.$
\end{defn}

\begin{prop}\label{prop:Repft2} For a quiver $Q$ and an abelian $AB4(\kappa)$  category  $\C,$  where $\kappa\geq\max\{\rccn(Q),\ltccn(Q),\rscn(Q),\alpha(Q)\}$ is an infinite cardinal number, the following statements hold true.
\begin{itemize}
 \item[(a)] $\pd_{\Rep^{ft}(Q,\C)}(F)\leq \sup_{i\in Q_{0}}(\pd(F_i))+1$
$\forall F\in\Rep^{ft}(Q,\C);$
\item[(b)] $\gldim(\Rep^{ft}(Q,\C))\leq\gldim(\C)+1;$
\item[(c)] $\gldim(\Rep^{ft}(Q,\C))=\gldim(\Rep(Q,\C))=\gldim(\C)+1$ if $\C\neq 0$ and there is an arrow $x\to y$ in $Q,$ with $\{x,y\}\in\T(Q),$ such that one of the following two conditions holds true: (i) $x\neq y;$ and (ii) $x=y$ and $\C$ is $AB3^*(\aleph_1).$  
 \end{itemize}
\end{prop}
\begin{proof} (a) Let $F\in \Rep^{ft}(Q,\C).$ By Proposition \ref{prop:Repft} (b), we have  an exact sequence 
$
\eta_{F}:\quad\suc[\coprod_{\rho\in Q_{1}}f_{t(\rho)}e_{s(\rho)}(F)][\coprod_{i\in Q_{0}}f_{i}e_{i}(F)][F]
$
in $\Rep^{ft}(Q,\C).$ By the proof of Proposition \ref{prop:Repft} (b), we know that $\coprod_{i\in Q_{0}}f_{i}e_{i}(F)=\coprod_{i\in\Supp(F)}f_i(F_i)$ and $\coprod_{\rho\in Q_{1}}f_{t(\rho)}e_{s(\rho)}(F)=\coprod_{x\in\Supp(F)}\coprod_{z\in t(Q(x,-))} \coprod_{\rho\in Q_1(x,z)} f_{t(\rho)}(F_x).$ Therefore, from Lemma \ref{pd-coprod} and Proposition \ref{prop:Repft} (d), we get 
\begin{center}
$\pd(\coprod_{\rho\in Q_{1}}f_{t(\rho)}e_{s(\rho)}(F))=\sup_{x\in\Supp(F)}\pd(F_x)=\pd(\coprod_{i\in Q_{0}}f_{i}e_{i}(F)),$ 
\end{center}
and thus, from the exact sequence $\eta_F,$ we get (a).  
\

(c) Assume that $\C\neq 0$ and there is an arrow $x\to y$ in $Q,$ with $\{x,y\}\in\T(Q),$ such that one of the following two conditions holds true: (i) $x\neq y;$ and (ii) $x=y$ and $\C$ is $AB3^*(\aleph_1).$  In particular, by Theorem \ref{thm:dimensiones homologicas}, we get that $\gldim(\Rep(Q,\C))=\gldim(\C)+1.$ Finally, using (b) and Proposition \ref{prop:Repft}, the proof of the equality $\gldim(\Rep^{ft}(Q,\C))=\gldim(\C)+1$ follows as in the proof of Theorem \ref{thm:dimensiones homologicas} since $\{x,y\}\in\T(Q).$
\end{proof}

Similarly, as in the proof of Proposition \ref{prop:Repft2}, it can be shown the following result.

\begin{prop}\label{prop:Repft2-dual}  For a quiver $Q$ and an abelian $AB4^{*}(\kappa)$  category  $\C,$  where $\kappa\geq\max\{\lccn(Q),\rtccn(Q),\lscn(Q),\alpha(Q)\}$ is an infinite cardinal number, the following statements hold true.
\begin{itemize}
 \item[(a)] $\id_{\Rep^{fb}(Q,\C)}(F)\leq \sup_{i\in Q_{0}}(\id(F_i))+1$
$\forall F\in\Rep^{fb}(Q,\C);$
\item[(b)] $\gldim(\Rep^{fb}(Q,\C))\leq\gldim(\C)+1;$
\item[(c)] $\gldim(\Rep^{fb}(Q,\C))=\gldim(\Rep(Q,\C))=\gldim(\C)+1$ if $\C\neq 0$ and there is an arrow $x\to y$ in $Q,$ with $\{x,y\}\in\B(Q),$ such that one of the following two conditions holds true: (i) $x\neq y;$ and (ii) $x=y$ and $\C$ is $AB3(\aleph_1).$  
 \end{itemize}
\end{prop}

Notice that Propositions \ref{prop:Repft2} and \ref{prop:Repft2-dual} are very useful to study the global dimension of $\Rep(Q,\C)$ by using their full abelian subcategories $\Rep^{ft}(Q,\C)$ and $\Rep^{fb}(Q,\C),$ where the first one is useful for using projectives and the second one for injectives. One interesting situation where we can use both is the case when $Q$ is support finite. We recall that $Q$ is {\bf right support finite} if the set of vertices $t(Q(i,-))$ is finite $\forall\, i\in Q_0.$ Dually, $Q$  is {\bf left support finite} if the opposite quiver $Q^{op}$ is right support finite. Finally, $Q$ is {\bf support finite} if it is left and right support finite. 

\begin{lem}\label{Rk-fcsq}
Let $Q$ be a quiver and $\C\neq 0$ an abelian category. Then
\begin{itemize}
\item[(a)] $\Rep^{ft}(Q,\C)=\Rep^{f}(Q,\C)$ $\Leftrightarrow$ $TQ_0=Q_0$ 
$\Leftrightarrow$ $Q$ is right support finite.
\item[(b)] $\Rep^{fb}(Q,\C)=\Rep^{f}(Q,\C)$ $\Leftrightarrow$ $BQ_0=Q_0$ 
$\Leftrightarrow$ $Q$ is left support finite.
\end{itemize}
\end{lem}
\begin{proof}
Notice that, for any finite non-empty set $H\subseteq Q_0,$ we have by Remark \ref{main-quivers} (a) that $H^+\subseteq \bigcup_{h\in H}\{h\}^+$ and $t(Q(i,-))=\{i\}^+\cup\{i\}$ $\forall\,i\in Q_0.$ Moreover, for $0\neq C\in\C$ and $i\in Q_0,$ we know that that $\Supp(s_i(C))=\{i\}$ for the stalk representation $s_i(C)$ at the vertex $i.$ Thus the equivalences in (a) follows. 
\end{proof}

We get the following corollary for the full abelian subcategory $\Rep^{f}(Q,\C)\subseteq\Rep(Q,\C)$ if $Q$ is support finite. 

\begin{cor}\label{basic-fcsq} For  a support finite quiver $Q,$  $\kappa\geq\ccn(Q)$ an infinite cardinal number and an $AB3(\kappa)$ and $AB3^*(\kappa)$ abelian category $\C,$ the following statements hold true.
\begin{itemize}
\item[(a)] $f_i(X),g_{i}(X)\in\Rep^{f}(Q,\C)$ for all $i\in Q_{0}$ and $X\in C.$ Moreover,  $f_i$ is left adjoint for $e_i:\Rep^{f}(Q,\C)\to \C$ and $g_i$ is right adjoint for $e_i.$

\item[(b)] Let $\kappa\geq \tccn(Q).$ Then, for any $F\in\Rep^{f}(Q,\C)$,  we have the canonical (co)presentations in $\Rep^{f}(Q,\C)$ 
 \begin{itemize}
 \item[(b1)] $\suc[\coprod_{\rho\in Q_{1}}f_{t(\rho)}e_{s(\rho)}(F)][\coprod_{i\in Q_{0}}f_{i}e_{i}(F)][F];$
 \item[(b2)] $ F\hookrightarrow \prod_{i\in Q_{0}}g_{i}e_{i}(F)\twoheadrightarrow \prod_{\rho\in Q_{1}}g_{s(\rho)}e_{t(\rho)}(F).$
 \end{itemize}
\item[(c)] Let $\C\neq 0,$ $Q$ an infinite quiver with $Q_1\neq\emptyset$ and $\kappa\geq \alpha(Q).$ Then, $\gldim(\Rep^{f}(Q,\C))=\gldim(\Rep(Q,\C))=\gldim(\C)+1$ if one of the following two conditions hold true: (i) $\C$ is 
$AB4^*(\kappa)$ and $\kappa\geq\rtccn(Q);$ and (ii) $\C$ is 
$AB4(\kappa)$ and $\kappa\geq\ltccn(Q).$
\end{itemize}
\end{cor}
\begin{proof} Notice that $\max\{\lscn(Q), \rscn(Q)\}\leq \aleph_0\leq\kappa$ since $Q$ is support finite and $\kappa$ an infinite cardinal. Moreover, if $Q$ has an oriented cycle (or a loop) then $\kappa\geq \aleph_1$ since 
$\ccn(Q)\geq \aleph_1.$ Thus, from Lemma  \ref{Rk-fcsq}, Propositions \ref{prop:Repft}, \ref{prop:Repft-dual}, \ref{prop:Repft2} and \ref{prop:Repft2-dual}, we get the result.
\end{proof}

\begin{example} Let $R$ be a non-zero ring and $Q$ an infinite quiver which is  support finite and $Q_1\neq\emptyset.$  Then by Corollary \ref{basic-fcsq} (c), we get $\gldim(\Rep^f(Q,\Mod(R))=\gldim(\Rep(Q,\Mod(R))=\gldim(R)+1.$
\end{example}

The following result is very useful in some applications, see for example in \cite{AM22}.

\begin{prop}\label{prop:Repf} For a finite-cone-shape quiver $Q$ and an abelian category $\C$ and a class $\X\subseteq\C,$ the following statements hold true.
\begin{itemize}
\item[(a)] $g_{i}:\C\rightarrow\Rep^{f}(Q,\C)$ admits a right adjoint and $f_{i}:\C\rightarrow\Rep^{f}(Q,\C)$ admits a left adjoint, for any $i\in Q_0.$

\item[(b)] Let $\mathcal{X}$ be precovering in $\C.$  Then $g_{i}(\X)$ is precovering
in $\Rep^{f}(Q,\C),$ for all $i\in Q_{0},$ and $\add(g_{*}(\X))$
is precovering in $\Rep^{f}(Q,\C)$.

\item[(c)] Let $\mathcal{X}$ be preenveloping in $\C.$  Then $f_{i}(\X)$ is preenveloping
in $\Rep^{f}(Q,\C),$ for all $i\in Q_{0},$ and $\add(f_{*}(\X))$
is preenveloping in $\Rep^{f}(Q,\C)$.

\item[(d)] $\pd(g_*(\X))\leq \pd(\X)+1$ and $\id(f_*(\X))\leq \id(\X)+1.$
\end{itemize}
\end{prop}

\begin{proof} Observe that $Q$ is support finite and 
$\ccn(Q)\leq\tccn(Q)\leq\aleph_0.$ Thus, the proof of (a) follows as in the proof of Proposition \ref{prop:exactitud g} by using Corollary \ref{basic-fcsq} (a,b). Since (c) is dual to (b), it remains to show (b) and (d).
\

(b) Let $\mathcal{X}$ be precovering in $\C$ and $F\in\Rep^{f}(Q,\C).$ Now, for $z\in Q_0,$ we get from (a), Corollary  \ref{basic-fcsq} (b) and Lemma \ref{adj-prec-preen} (a)  that there is $p_z:g_z(X_z)\to F$ which is a $g_z(\X)$-precover of $F.$ 
Consider the set $\mathcal{D}:=\bigcup_{i\in\Supp(F)}t(Q(i,-)).$ Notice that $\mathcal{D}$ is finite since $Q$ is support finite and $F\in\Rep^{f}(Q,\C).$
We assert that 
\begin{center}
$\Hom_{\Rep^{f}(Q,\C)}(g_{z}(X),F)=0\quad$  $\forall\, z\in Q_{0}-\mathcal{D},$ 
$\forall\,X\in\C$. 
\end{center} 
Indeed, if  $z\in Q_{0}-\mathcal{D},$ we have that $Q(i,z)=\emptyset\;$ $\forall\,i\in\Supp(F)$ and thus our assertion follows since $g_z(X)=X^{Q(-,z)}.$
\

Let $p_{i}:g_{i}(X_{i})\rightarrow F$ be a $g_{i}(\X)$-precover
for all $i\in\mathcal{D},$ and for $z\in Q_{0}-\mathcal{D},$ we set $p_z:0\to F.$ Now, 
consider the morphism $p:\coprod_{i\in Q_{0}}g_{i}(X_{i})\rightarrow F$ induced by the universal property of coproducts. Observe that $\coprod_{i\in Q_{0}}g_{i}(X_{i})=\coprod_{i\in \mathcal{D}}g_{i}(X_{i})\in\Rep^f(Q,\C),$ and we let to the reader to show, using the above assertion, that $p$ is an $\add(g_{*}(\X))$-precover of $F.$
\

(d) Let $X\in\X$ and $k\in Q_0.$ Since $\Rep^{ft}(Q,\C)=\Rep^{f}(Q,\C)$ (see Lemma \ref{Rk-fcsq} (a)), we get from Proposition \ref{prop:Repft2} (a) that 
\begin{center}
$\pd(g_k(X))\leq\sup_{i\in Q_0}(\pd(X^{Q(i,k)}))+1=\pd(X)+1$
\end{center}
 and thus  $\pd(g_*(\X)))\leq \pd(\X)+1.$ 
\end{proof}

Next we will see that, under mild conditions, every cotorsion pair
in $\mathcal{C}$ induces a relative cotorsion pair in the class of
finite-bottom-support representations.

\begin{lem}\label{lem:psi} For a quiver $Q,$ $\kappa\geq\rmcn(Q)$ an infinite
cardinal number, an Ab3{*}($\kappa$) abelian category $\mathcal{C}$
and $S\subseteq Q_0$ a full subquiver, the following statements hold
true. 
\begin{enumerate}
\item[$\mathrm{(a)}$] Let $F\in\operatorname{Rep}(S,\mathcal{C})$ and $i$ a vertex in $S$.
Then $\varepsilon_{i}^{S}\psi_{i}^{F,S}=\psi_{i}^{\iota_{S}(F),Q},$
for some isomorphism $\varepsilon_{i}^{S}$ in $\C.$ 

\item[$\mathrm{(b)}$] If $S^{-}=\emptyset$ then $\psi_{i}^{\iota_{S}(F),Q}$ is the trivial
epimorphism $0\rightarrow0$ $\forall i\in Q{}_{0}-S$. 

\item[$\mathrm{(c)}$] If $0\in\mathcal{B}\subseteq\mathcal{C}$ and $S^{-}=\emptyset,$
then $\left\{ \iota_{S}(F)\;|\:F\in\Psi_{S}(\mathcal{B})\right\} \subseteq\Psi_{Q}(\mathcal{B}).$ 
\item[$\mathrm{(d)}$] If $0\in\mathcal{B}\subseteq\mathcal{C}$ then $\left\{ F\;|\:\iota_{S}(F)\in\Psi_{Q}(\mathcal{B})\right\} \subseteq\Psi_{S}(\mathcal{B})$. 
\item[$\mathrm{(e)}$] Let $S\in\F\B(Q).$ Then there exists $T\in\F(Q)$ such that $S\subseteq T$
and $T^{-}=\emptyset.$ Moreover, for $F\in\operatorname{Rep}(S,\C),$
we have that $G:=\pi_{T}(\iota_{S}(F))\in\operatorname{Rep}(T,\mathcal{C})$
and $\iota_{S}(F)=\iota_{T}(G).$ 
\item[$\mathrm{(f)}$] Let $S_{1},S_{2}\in\F\B(Q).$ Then there exists $T\in\F(Q)$ such
that $S_{1}\cup S_{2}\subseteq T$ and $T^{-}=\emptyset.$ 
\end{enumerate}
\end{lem}

\begin{proof}
To prove (a), consider $X\in\mathcal{C},$ a family $\{f_{i}:X\rightarrow X_{i}\}_{i\in I}$
of morphisms in $\mathcal{C}$ with $|I|<\kappa,$ and the morphism
$f_{I}:X\rightarrow\prod_{i\in I}X_{i}$ given by the universal property
of the product. Observe that, in case of having a subset $J\subseteq I$
such that $X_{j}=0$ $\forall j\in J$, we have an isomorphism $\varepsilon_{J}:\prod_{i\in I-J}X_{i}\to\prod_{i\in I}X_{i}$
such that $\varepsilon_{J}f_{I-J}=f_{I},$ where $f_{I-J}:X\rightarrow\prod_{i\in I-J}X_{i}.$
Hence, we can apply the above to our case since \small{$\{\iota_{S}(F)(\alpha):\iota_{S}(F)(i)\rightarrow\iota_{S}(F)(t(\alpha))\}_{\alpha\in Q_{1}^{i\rightarrow*}}$}
is the disjoint union of $\left\{ F(\alpha):F(i)\rightarrow F(t(\alpha))\right\} {}_{\alpha\in S_{1}^{i\rightarrow*}}$
and $\left\{ \iota_{S}(F)(\alpha):\iota_{S}(F)(i)\rightarrow\iota_{S}(F)(t(\alpha))\right\} {}_{\alpha\in Q_{1}^{i\rightarrow*}-S_{1}^{i\rightarrow*}},$
where $\iota_{S}(F)(t(\alpha))=0$ $\forall\alpha\in Q_{1}^{i\rightarrow*}-S_{1}^{i\rightarrow*}.$
\

In the case of the item (b), it is enough to observe that $F(\alpha)$
is the trivial morphism $0\rightarrow0$ $\forall\alpha:i\rightarrow j$
with $i,j\in Q_{0}-S_{0}$. The items (c) and (d) follow from (a)
and (b). In order to get (e), it is enough to consider the full subquiver
$T$ determined by the set of vertices $S\cup S^{-}$ and then
to apply Remark \ref{supp-ex-seq} (b) to $G=\pi_{T}(\iota_{S}(F))$
since $\operatorname{Supp}(\iota_{S}(F))\subseteq S\subseteq T.$
\

Let us show (f). By (e) there are $T_{1},T_{2}\in\F(Q)$ such that
$S_{i}\subseteq T_{i}$ and $T_{i}^{-}=\emptyset$ for $i=1,2.$ Thus
for $T:=T_{1}\cup T_{2}$ we have that $T^{-}\subseteq T_{1}^{-}\cup T_{2}^{-}=\emptyset.$ 
\end{proof}

\begin{thm}\label{ThmX} For an acyclic quiver $Q,$ $\kappa\geq\max\{\lccn(Q),\rtccn(Q)\}$
an infinite cardinal number, an AB4{*}($\kappa$) abelian category
$\C$ and $\mathcal{X}:=\operatorname{Rep}^{fb}(Q,\mathcal{C}),$
the following statements hold true. 
\begin{enumerate}
\item[$\mathrm{(a)}$] Let $\p$ be a complete hereditary cotorsion pair in $\mathcal{C}.$
Then: 
\begin{itemize}
\item[$\mathrm{(a0)}$] $(\operatorname{Rep}(Q,\mathcal{A}),\Psi(\mathcal{B}))$ is a complete
hereditary cotorsion pair in $\Rep(Q,\C)$ if $Q$ is right-rooted. 
\item[$\mathrm{(a1)}$] The pair $(\operatorname{Rep}(Q,\mathcal{A}),\Psi(\mathcal{B}))$
is $\mathcal{X}$-complete and $\mathcal{X}$-hereditary. 
\item[$\mathrm{(a2)}$] $(\operatorname{Rep}(Q,\mathcal{A})\cap\mathcal{X},\Psi(\mathcal{B})\cap\mathcal{X})$
is a hereditary complete cotorsion pair in the full abelian subcategory
$\mathcal{X}\subseteq\Rep(Q,\C).$ 
\item[$\mathrm{(a3)}$] $\Psi(\Inj(\C))=s_{*}(\C)^{\perp_{1}}.$ 
\end{itemize}
\item[$\mathrm{(b)}$] Let $\C$ be with enough injectives. Then $\Rep(Q,\C)$ has enough injectives and we also have that:
\begin{itemize}
\item[$\mathrm{(b1)}$] $\Psi(\Inj(\mathcal{C}))\cap\mathcal{X}$ is an $\mathcal{X}$-injective
relative cogenerator in $\mathcal{X}.$ 
\item[$\mathrm{(b2)}$] $\Inj(\Rep(Q,\C))=\Prod_{\leq|Q_{0}|}(g_{*}(\Inj(\C))).$ 
\item[$\mathrm{(b3)}$] $\Inj(\Rep(Q,\C))=\Psi(\Inj(\mathcal{C}))$  if $Q$ is right-rooted.
\end{itemize}
\item[$\mathrm{(c)}$] If $\C$ is $AB4(\kappa')$ for some infinite cardinal $\kappa'\geq\rccn(Q),$ then we have that $\Inj(\Rep(Q,\C))=\Prod_{\leq|Q_{0}|}(g_{*}(\Inj(\C))).$ 
\end{enumerate}
\end{thm}

\begin{proof}
(a) In case $Q$ is right-rooted, by the dual of Theorem \ref{hccp+ep}
(b), we get (a0). Let us show that the pair $(\operatorname{Rep}(Q,\mathcal{A}),\Psi(\mathcal{B}))$
is $\X$-complete. Consider $X\in\mathcal{X}$. Then, by Remark \ref{MainF-FB-B}
$\exists\,S\in\F\B(Q)$ and $F\in\operatorname{Rep}(S,\C)$ such that
$X=\iota_{S}\left(F\right).$ On the other hand, by Lemma \ref{lem:psi}
(e), there exists $T\in\F(Q)$ and $G\in\operatorname{Rep}(T,\C)$
such that $S\subseteq T,$ $T^{-}=\emptyset$ and $\iota_{S}(F)=\iota_{T}(G).$
Now, since $T$ is right-rooted, from the dual of Theorem \ref{hccp+ep}
(b), we get that $(\operatorname{Rep}(T,\mathcal{A}),\Psi_{T}(\mathcal{B}))$
is a complete hereditary cotorsion pair in $\operatorname{Rep}(T,\mathcal{C}).$
Thus, there are exact sequences 
\[
\suc[G][B][A]\;\text{ and }\;\suc[B'][A'][G]\text{,}
\]
where $A,A'\in\operatorname{Rep}(T,\mathcal{A})$ and $B,B'\in\Psi_{T}(\mathcal{B}).$
By applying the exact functor $\iota_{T}$ to the above exact sequences,
we get the exact sequences 
\[
\suc[X][\iota_{T}(B)][\iota_{T}(A)]\text{ and }\suc[\iota_{T}(B')][\iota_{T}(A')][X]\text{,}
\]
where $\iota_{T}(A),\iota_{T}(A')\in\operatorname{Rep}(Q,\mathcal{A})\cap\mathcal{X}$
and $\iota_{T}(B),\iota_{T}(B')\in\Psi_{Q}(\mathcal{B})\cap\mathcal{X}$
by Lemma \ref{lem:psi}(c). Therefore, $(\operatorname{Rep}(Q,\mathcal{A}),\Psi(\mathcal{B}))$
is $\mathcal{X}$-complete. \

Let us show that $(\operatorname{Rep}(Q,\mathcal{A}),\Psi(\mathcal{B}))$
is $\mathcal{X}$-hereditary. Consider $A\in\operatorname{Rep}(Q,\mathcal{A})\cap\mathcal{X}$
and $B\in\Psi_{Q}(\mathcal{B})\cap\mathcal{X}.$ Since $\X=\Ima(\iota_{\F\B(Q)})$
(see Remark \ref{MainF-FB-B}), from Lemma \ref{lem:psi} (e,f), there
is $T\in\F(Q)$ with $T^{-}=\emptyset,$ such that $A=\iota_{T}(A')$
and $B=\iota_{T}(B')$, where $A'\in\operatorname{Rep}(T,\mathcal{A})$
and $B'\in\Psi_{T}(\mathcal{B})$ by Lemma \ref{lem:psi}(d). Thus,
by Lemma \ref{lem:Ext subcarcaj}, for $k>0,$ we get that 
\[
\Ext_{\Rep(Q,\C)}^{k}(A,B)=\Ext_{\Rep(Q,\C)}^{k}(\iota_{T}A',\iota_{T}B')\cong\Ext_{\Rep(T,\C)}^{k}(A',B')=0.
\]
We now prove that $(\operatorname{Rep}(Q,\mathcal{A})\cap\mathcal{X},\Psi(\mathcal{B})\cap\mathcal{X})$
is a cotorsion pair in the full abelian subcategory $\mathcal{X}$
(see Remark \ref{MainF-FB-B}). Note that $\Psi(\mathcal{B})\cap\mathcal{X}\subseteq(\operatorname{Rep}(Q,\mathcal{A})\cap\mathcal{X})^{\bot_{1}}$.
Let $X\in(\operatorname{Rep}(Q,\mathcal{A})\cap\mathcal{X})^{\bot_{1}}\cap\mathcal{X}.$
Then, there is an exact sequence $\suc[X][B][A]$, where $A\in\operatorname{Rep}(Q,\mathcal{A})\cap\mathcal{X}$
and $B\in\Psi(\mathcal{B})\cap\mathcal{X}$. This sequence splits
because $X\in(\operatorname{Rep}(Q,\mathcal{A})\cap\mathcal{X})^{\bot_{1}},$
and hence $X\in\Psi(\mathcal{B})\cap\mathcal{X}$ since $\Psi(\mathcal{B})\cap\mathcal{X}$
is closed under direct summands. Thus, we have shown that $\Psi(\mathcal{B})\cap\mathcal{X}=(\operatorname{Rep}(Q,\mathcal{A})\cap\mathcal{X})^{\bot_{1}}\cap\mathcal{X}$.
The equality $^{\bot_{1}}\left(\Psi(\mathcal{B})\cap\mathcal{X}\right)\cap\mathcal{X}=\operatorname{Rep}(Q,\mathcal{A})\cap\mathcal{X}$
follows from similar arguments. Finally, the item (a3) follows from
Proposition \ref{prop:f y clases ortogonales} (c) since $\rmcn(Q)\leq\rtccn(Q).$
\

(b) The item (b1) (resp. (b3)) follows from (a1) (resp. a(0)) by using the complete and hereditary
cotorsion pair $(\mathcal{C},\Inj(\mathcal{C}))$ in $\C.$ The item
(b2) can be shown from the dual of Proposition \ref{prop:f y g vs clases ortogonales}
(b). 
\

(c) It follows from Remark \ref{AB3+Inj=AB4} (a) and the dual of Proposition \ref{prop:f y g vs clases ortogonales} (d)
\end{proof}
In what follows we also establish the dual of Theorem \ref{ThmX}.

\begin{thm}\label{ThmY} For an acyclic quiver $Q,$ $\kappa\geq\max\{\rccn(Q),\ltccn(Q)\}$
an infinite cardinal number, an AB4($\kappa$) abelian category $\C$
and $\mathcal{Y}:=\operatorname{Rep}^{ft}(Q,\mathcal{C}),$ the following
statements hold true. 
\begin{enumerate}
\item[$\mathrm{(a)}$] Let $\p$ be a complete hereditary cotorsion pair in $\mathcal{C}.$
Then: 
\begin{itemize}
\item[$\mathrm{(a0)}$] $(\Phi(\A),\Rep(Q,\B))$ is a complete hereditary cotorsion pair
in $\Rep(Q,\C)$ if $Q$ is left-rooted. 
\item[$\mathrm{(a1)}$] The pair $(\Phi(\A),\Rep(Q,\B))$ is $\mathcal{Y}$-complete and
$\mathcal{Y}$-hereditary.
\item[$\mathrm{(a2)}$] $(\Phi(\A)\cap\Y,\Rep(Q,\B)\cap\Y)$ is a hereditary complete cotorsion
pair in the full abelian subcategory $\Y\subseteq\Rep(Q,\C).$ 
\item[$\mathrm{(a3)}$] $\Phi(\Proj(\C))={}^{\perp_{1}}s_{*}(\C).$ 
\end{itemize}
\item[$\mathrm{(b)}$] Let $\C$ be with enough projectives. Then $\Rep(Q,\C)$ has enough projectives and we also have that:
\begin{itemize}
\item[$\mathrm{(b1)}$] $\Phi(\Proj(\mathcal{C}))\cap\Y$ is a $\Y$-projective relative
generator in $\Y;$ 
\item[$\mathrm{(b2)}$] $\Proj(\Rep(Q,\C))=\Add_{\leq|Q_{0}|}(f_{*}(\Proj(\C))).$ 
\item[$\mathrm{(b3)}$] $\Proj(\Rep(Q,\C))=\Phi(\Proj(\mathcal{C}))$ if $Q$ is left-rooted. 
\end{itemize}
\item[$\mathrm{(c)}$] If $\C$ is $AB4^{*}(\kappa')$ for some infinite cardinal $\kappa'\geq\lccn(Q),$ then we have that $\Proj(\Rep(Q,\C))=\Add_{\leq|Q_{0}|}(f_{*}(\Proj(\C))).$ 
\end{enumerate}
\end{thm}

For the case of a finite-cone-shape quiver $Q,$ we have that  $\tccn(Q)\leq\aleph_{0},$ $Q$ is rooted and $\Rep^{fb}(Q,\C)=\Rep^{f}(Q,\C)=\Rep^{ft}(Q,\C)$ and thus by Theorems \ref{ThmX}
and \ref{ThmY} and Proposition \ref{Descrip-Phi-Psi}, we get the following corollary.

\begin{cor}\label{fcsq-hcp} For a finite-cone-shape quiver $Q,$ an abelian
category $\C,$ a complete hereditary cotorsion pair $\p$ in $\C$
and $\Z:=\Rep^{f}(Q,\C),$ the following statements hold true. 
\begin{itemize}
\item[$\mathrm{(a)}$] $\left(\Phi(\mathcal{A}),\Rep(Q,\B)\right)$ and $\left(\Rep(Q,\A),\Psi(\B)\right)$
are complete hereditary cotorsion pairs in $\Rep(Q,\C).$ Moreover,
these pairs are $\Z$-complete and $\Z$-hereditary. 

\item[$\mathrm{(b)}$] $(\operatorname{Rep}(Q,\mathcal{A})\cap\Z,\Psi(\mathcal{B})\cap\Z)$
and $(\Phi(\A)\cap\Z,\Rep(Q,\B)\cap\Z)$ are hereditary complete cotorsion
pairs in the full abelian subcategory $\Z\subseteq\Rep(Q,\C).$ 

\item[$\mathrm{(c)}$] $\Psi(\Inj(\C))=s_{*}(\C)^{\perp_{1}}$ and $\Phi(\Proj(\C))={}^{\perp_{1}}s_{*}(\C).$ 

\item[$\mathrm{(d)}$] $\Proj(\Rep(Q,\C))=\Add_{\leq|Q_{0}|}(f_{*}(\Proj(\C)))$ and  we also have that\\ $\Inj(\Rep(Q,\C))=\Prod_{\leq|Q_{0}|}(g_{*}(\Inj(\C))).$ 

\item[$\mathrm{(e)}$] Let $\C$ be with enough injectives. Then $\Rep(Q,\C)$ has enough injectives,  $\Inj(\Rep(Q,\C))=\Psi(\Inj(\mathcal{C}))$ and $\Psi(\Inj(\mathcal{C}))\cap\Z$ is a $\Z$-injective relative cogenerator
in $\Z.$ 

\item[$\mathrm{(f)}$] Let $\C$ be with enough projectives. Then $\Rep(Q,\C)$ has enough projectives,  $\Proj(\Rep(Q,\C))=\Phi(\Proj(\mathcal{C}))$ and  $\Phi(\Proj(\mathcal{C}))\cap\Z$ is a $\Z$-projective relative
generator in $\Z.$ 

\item[$\mathrm{(g)}$] $^{\bot}s_{*}(\mathcal{B})={}{}^{\bot}f_{*}(\mathcal{B})=\Phi(\mathcal{A})\;$ and $\,s_{*}(\mathcal{A})^{\bot}=g_{*}(\mathcal{A})^{\bot}=\Psi(\mathcal{B}).$
\end{itemize}
\end{cor}

\section*{Acknowledgements}

The research presented in this paper was conducted while the first
named author was on a post-doctoral fellowship at Centro de Ciencias
Matem\'aticas, UNAM Campus Morelia, funded by the DGAPA-UNAM. The
first named author would like to thank all the academic and administrative
staff of this institution for their warm hospitality, and in particular
Dr. Raymundo Bautista (CCM, UNAM) for all his support.

\bibliographystyle{plain}

\end{document}